\documentclass{article}

\usepackage{geometry}

\usepackage{amssymb}
\usepackage{amsmath}
\usepackage{amsthm}
\usepackage{enumitem}
\usepackage{bm}
\usepackage{tikz}
\usetikzlibrary{shapes}
\usepackage[margin=2cm]{caption}
\usepackage{subcaption}

\usepackage{graphicx}

\usepackage{xcolor}

\usepackage{thmtools}
\usepackage{thm-restate}

\setlist[itemize]{topsep=3pt}
\setlist[enumerate]{topsep=3pt}

\newtheorem{theorem}{Theorem}[section]

\newtheorem{lemma}[theorem]{Lemma}
\newtheorem{claim}{Claim}[theorem]

\newtheorem{cor}[theorem]{Corollary}
\newtheorem{prop}[theorem]{Proposition}
\newtheorem{remark}[theorem]{Remark}
\newtheorem{question}[theorem]{Question}

\theoremstyle{definition}

\newenvironment{subproof}[1][Proof]{\begin{proof}[#1]}{\end{proof}}

\newcommand{\mbn}{\mathbb{N}}

\newcommand{\mbz}{\mathbb{Z}}

\newcommand{\mcb}{\mathcal{B}}
\newcommand{\mcc}{\mathcal{C}}

\newcommand{\mcf}{\mathcal{F}}
\newcommand{\mch}{\mathcal{H}}

\newcommand{\mcl}{\mathcal{L}}

\newcommand{\mcp}{\mathcal{P}}
\newcommand{\mcq}{\mathcal{Q}}
\newcommand{\mct}{\mathcal{T}}

\title{A flat wall theorem for group-labelled graphs in the undirected model}
\author{Robin Thomas \and Youngho Yoo}

\begin{document}

\centerline{\Large \bf Packing cycles in undirected group-labelled graphs}

\bigskip
\bigskip

\centerline{{\bf Robin Thomas}%
\footnote{Partially supported by NSF under Grant No.~DMS-1700157.}
}
\smallskip
\centerline{and}
\smallskip
\centerline{{\bf Youngho Yoo}%
\footnote{This work was done while the second author was at Georgia Tech. Partially supported by the Natural Sciences and Engineering Research Council of Canada (NSERC), PGSD2-532637-2019. }
}
\bigskip
\centerline{$^1$School of Mathematics}
\centerline{Georgia Institute of Technology}
\centerline{Atlanta, Georgia  30332-0160, USA}
\bigskip
\centerline{$^2$Department of Mathematics}
\centerline{Texas A\&M University}
\centerline{College Station, Texas 77843-3368, USA}
\bigskip

\begin{abstract}
We prove a refinement of the flat wall theorem of Robertson and Seymour to undirected group-labelled graphs $(G,\gamma)$ where $\gamma$ assigns to each edge of an undirected graph $G$ an element of an abelian group $\Gamma$.
As a consequence, we prove that $\Gamma$-nonzero cycles (cycles whose edge labels sum to a non-identity element of $\Gamma$) satisfy the half-integral Erd\H{o}s-P\'osa property, and we also recover a result of Wollan that if $\Gamma$ has no element of order two, then $\Gamma$-nonzero cycles satisfy the Erd\H{o}s-P\'osa property.
As another application, we prove that if $m$ is an odd prime power, then cycles of length $\ell \mod m$ satisfy the Erd\H{o}s-P\'osa property for all integers $\ell$.
This partially answers a question of Dejter and Neumann-Lara from 1987 on characterizing all such integer pairs $(\ell,m)$.

\end{abstract}

\section{Introduction}
Erd\H{o}s and P\'osa showed in \cite{ErdPos} that cycles satisfy an approximate packing-covering duality; that is, there exists a function $f(k)=O(k\log k)$ such that in every graph, either there are $k$ vertex-disjoint cycles or there is a set of at most $f(k)$ vertices intersecting every cycle.
This result has generated extensive activity on whether various families of graphs satisfy a similar approximate duality, also known as the \emph{Erd\H{o}s-P\'osa property} (precise definitions are deferred to section \ref{sec:prelim}).
We refer to \cite{RayThi} for a recent survey on the Erd\H{o}s-P\'osa property and also to \cite{BruHeiJoo} for a collection of results on families of cycles and paths.

Group-labelled graphs provide a general framework in which many Erd\H{o}s-P\'osa problems can be studied simultaneously.
For instance, Huynh, Joos, and Wollan \cite{HuyJooWol} proved a structure theorem for \emph{directed} group-labelled graphs and obtained as special cases that the Erd\H{o}s-P\'osa property holds for the families of cycles, $S$-cycles (cycles intersecting some fixed vertex set $S$), $S_1$-$S_2$-cyles (cycles intersecting two fixed vertex sets $S_1$ and $S_2$), and cycles not homologous to zero in graphs embedded on an orientable surface.
Furthermore, their structure theorem describes the canonical obstructions to such cycles satisfying the Erd\H{o}s-P\'osa property, which in turn proves the \emph{half-integral Erd\H{o}s-P\'osa property} of more general families including odd cycles, odd $S$-cycles, and cycles not homologous to zero in graphs embedded on a nonorientable surface.

The framework of directed group-labelled graphs, however, does not seem to address an important class of graph families, namely of cycles and paths with modularity constraints with modulus greater than 2.
To this end, we consider \emph{undirected} group-labelled graphs as the main focus of this paper.

Let $\Gamma$ be an abelian group with additive operation and identity 0.
A \emph{$\Gamma$-labelled graph} is an ordered pair $(G,\gamma)$ where $G$ is an undirected graph and $\gamma:E(G) \to \Gamma$ is a \emph{$\Gamma$-labelling} of $G$.
The \emph{weight} of a subgraph $H$ of $G$ is defined as $\gamma(H):=\sum_{e \in E(H)}\gamma(e)$. 
We say that $H$ is \emph{$\Gamma$-zero} or \emph{$\Gamma$-nonzero} if $\gamma(H)=0$ or $\gamma(H)\neq 0$ respectively.

We remark that since we are mostly concerned with the weights of cycles and paths, it is possible to allow $\Gamma$ to be nonabelian and define the weight of a walk to be the ordered sum of the edge labels in the order of the walk.
However, under this definition, the weight of a cycle can depend on the choice of its starting vertex and direction of traversal, and reversing the direction of traversal can affect whether the cycle is $\Gamma$-nonzero (this is in contrast to \textit{directed} group-labelled graphs
 where reversing the direction of traversal simply inverts the weight of the cycle).
In this paper, we avoid this difficulty and only consider abelian groups.

In the context of Erd\H{o}s-P\'osa problems, studying cycles with constraints modulo $m$ is equivalent to studying cycles in $\mbz/m\mbz$-labelled graphs. Indeed, given a graph $G$, there is a natural $\mbz/m\mbz$-labelling $\gamma$ of $G$ defined by setting $\gamma(e)$ to be the congruence class of 1 modulo $m$ for all $e \in E(G)$.
Conversely, given a $\mbz/m\mbz$-labelled graph $(G,\gamma)$ there is a corresponding unlabelled graph $G'$ obtained from $G$ by replacing each edge $e \in E(G)$ with a path whose length is in the congruence class $\gamma(e)$.

Our main contribution is a structure theorem for undirected group-labelled graphs (Theorem \ref{flatwallundirectedtheorem}) analogous to that of Huynh, Joos, and Wollan (Theorem 22 in \cite{HuyJooWol}) in the directed setting.
As one consequence, we recover a result of Wollan \cite{WolCycle} that, if $\Gamma$ is an abelian group with no element of order two, then the family of $\Gamma$-nonzero cycles in undirected $\Gamma$-labelled graphs satisfies the Erd\H{o}s-P\'osa property.
In particular, if $m > 0$ is an odd integer, then the family of cycles of length $\not \equiv 0 \mod m$ satisfies the Erd\H{o}s-P\'osa property.
This also implies that $S$-cycles satisfy the Erd\H{o}s-P\'osa property (set $\Gamma = \mbz$, label an edge 1 if it is incident to $S$, and 0 otherwise).

If $\Gamma$ has an element $g$ of order two, then such a conclusion is not possible.
Consider the $n\times n$-grid with vertex set $[n]\times [n]$ (defined in section \ref{sectionwalls}) and label all edges 0. 
Add the edge $(1,i)(n,n-i+1)$ with label $g$ for each $i\in[k]$.
The resulting graph can be embedded on the projective plane so that a cycle is $\Gamma$-nonzero if and only if it is a one-sided closed curve.
This implies that there does not exist two disjoint $\Gamma$-nonzero cycles.
On the other hand, there does not exist a hitting set for $\Gamma$-nonzero cycles with less than $n$ vertices.
Indeed, a set $S$ of less than $n$ vertices is disjoint from a path $P_m$ for some $m\in [n]$, where for each $j\in[n]$, $P_j$ denotes the path in the grid induced by the vertex set $\{(j, i):i\in[n]\}$. 
Since $|S|<n$, there are $n_1$ disjoint paths from $P_1$ to $P_m$ and there are $n_2$ disjoint paths from $P_m$ to $P_n$ such that $n_1+n_2 \geq n+1$.
This yields a $\Gamma$-nonzero cycle disjoint from $S$, and therefore a hitting set for $\Gamma$-nonzero cycles has at least $n$ vertices.
Since the minimum size of a hitting set can be arbitrarily large in relation to the maximum size of a packing, this shows that the family of $\Gamma$-nonzero cycles does not satisfy the Erd\H{o}s-P\'osa property if $\Gamma$ has an element of order two.

Another consequence of Theorem \ref{flatwallundirectedtheorem} is that this projective planar grid is essentially the only obstruction for the Erd\H{o}s-P\'osa property of $\Gamma$-nonzero cycles (see section \ref{sec:nzcycles}).
In particular, if $\Gamma = \mbz/2\mbz$, then this (essentially) recovers a theorem of Reed (Theorem 1 in \cite{Ree99}) that every graph contains either $k$ disjoint odd cycles, a bounded hitting set for odd cycles, or a large \emph{Escher wall} which is a projective planar grid-like graph similar to this construction.
Also note that this projective planar grid contains a large \emph{half-integral packing} of $\Gamma$-nonzero cycles.
This leads to the following theorem.

\begin{restatable}{theorem}{nzcyclesEPtheoremres}
\label{nzcyclesEPtheorem}
Let $\Gamma$ be an abelian group.
Then the family of $\Gamma$-nonzero cycles satisfies the half-integral Erd\H{o}s-P\'osa property. 
Moreover, if $\Gamma$ has no element of order two, then the family of $\Gamma$-nonzero cycles satisfies the Erd\H{o}s-P\'osa property. 
\end{restatable}

In particular, for every positive integer $m$, the family of cycles of length $\not\equiv 0 \mod m$ satisfies the half-integral Erd\H{o}s-P\'osa property. 
Recently, Gollin et al. \cite{GHKKO} proved a more general half-integral Erd\H{o}s-P\'osa result on cycles labelled by multiple abelian groups, which includes the first statement of Theorem \ref{nzcyclesEPtheorem} as a special case.

As another application, we consider the following question of Dejter and Neumann-Lara from 1987:
\begin{question}[Question 4 in \cite{DejNeu}] \label{questionDNL}
For which pairs of positive integers $(\ell,m)$ does the family of cycles of length $\ell \mod m$ satisfy the Erd\H{o}s-P\'osa property?
\end{question}
Thomassen \cite{Tho88} proved that for all positive integers $m$, the family of cycles of length $0 \mod m$ satisfies the Erd\H{o}s-P\'osa property.
On the other hand, Dejter and Neumann-Lara gave infinitely many pairs $(\ell,m)$ where the Erd\H{o}s-P\'osa property fails (Theorem 3 in \cite{DejNeu}). 
Their argument can be rephrased as follows. 
Consider the projective planar $n\times n$-grid construction with $\Gamma = \mbz/m\mbz$.
Give the label $\ell$ to all edges of the form $(1,i)(n,n-i+1)$ and label all other edges 0.
Then, as before, there does not exist a hitting set for cycles of weight $\ell$ with less than $n$ vertices.
Now if the order of $\ell$ in $\Gamma$ is even, then all cycles of weight $\ell$ are one-sided closed curves on the projective plane and no two such cycles are disjoint.
Therefore, for all pairs $(\ell,m)$ such that $\ell \not\equiv 0 \mod m$ and $\ell$ has even order in $\mbz/m\mbz$, the family of cycles of length $\ell \mod m$ does not satisfy the Erd\H{o}s-P\'osa property.

Notice that, in all such pairs, the modulus $m$ is even.
For odd $m$, Question \ref{questionDNL} has remained open for all $\ell \not\equiv 0\mod m$, even for $m=3$.
Using Theorem \ref{flatwallundirectedtheorem}, we prove the following:
\begin{theorem}
\label{cyclesoddprimepowertheorem}
Let $m$ be an odd prime power. 
Then for all integers $\ell$, the family of cycles of length $\ell \mod m$ satisfies the Erd\H{o}s-P\'osa property.
\end{theorem}

Let us remark on a further application of Theorem \ref{flatwallundirectedtheorem}.
It is known that \textit{$A$-paths} of even length satisfy the Erd\H{o}s-P\'osa property \cite{BruHeiJoo}.
Bruhn and Ulmer \cite{BruUlm} showed that the same is true for $A$-paths of length $0 \mod 4$ which is particularly interesting because for all other composite moduli $m$, $A$-paths of length $0 \mod m$ do \emph{not} satisfy the Erd\H{o}s-P\'osa property \cite{BruHeiJoo}.
They then asked the question of whether $A$-paths of length $0\mod p$ for a fixed odd prime $p$ satisfy the Erd\H{o}s-P\'osa property (Problem 22 in \cite{BruUlm}).
We show in \cite{ThoYoob} using Theorem \ref{flatwallundirectedtheorem} that $A$-paths of length $0 \mod p$ for prime $p$ indeed satisfy the Erd\H{o}s-P\'osa property, and characterize all abelian groups $\Gamma$ and elements $\ell \in \Gamma$ for which the family of $A$-paths of weight $\ell$ satisfies the Erd\H{o}s-P\'osa property.

Theorem \ref{flatwallundirectedtheorem} is a refinement of the flat wall theorem of Robertson and Seymour \cite{RobSeyXIII} to undirected group-labelled graphs.
Our proof follows the outline of the proof of Huyhn, Joos, and Wollan \cite{HuyJooWol} for directed group-labelled graphs, which is itself an extension of the proofs in \cite{GeeGer, GeeGerReeSeyVet}.
However, there are (in some places significant) technical complications in dealing with undirected group-labelled graphs that are absent in the directed setting.
Upon giving the necessary definitions to state Theorem \ref{flatwallundirectedtheorem}, we sketch an outline of its proof (see section \ref{sec:maintheorem}) and point out where the complications occur.

The remainder of the paper is organized as follows.
Section \ref{sec:prelim} gives the preliminaries.
In section \ref{sec:maintheorem}, we state Theorem \ref{flatwallundirectedtheorem} and outline its proof, then derive Theorem \ref{nzcyclesEPtheorem} from Theorem \ref{flatwallundirectedtheorem}.
We prove Theorem \ref{cyclesoddprimepowertheorem} (assuming Theorem \ref{flatwallundirectedtheorem}) in section \ref{sec:oddprimecycles}.
Sections \ref{sec:clique} and \ref{sec:flatwall} deal with the two outcomes of the flat wall theorem and provide the main technical lemmas, which are then combined to deduce Theorem \ref{flatwallundirectedtheorem} in section \ref{sec:mainproof}.

\section{Preliminaries}
\label{sec:prelim}

All graphs and group-labelled graphs are assumed to be undirected and may have parallel edges but no loops.
A graph is \textit{simple} if it has no parallel edges.
Unless explicitly stated otherwise, we say disjoint to mean vertex-disjoint whenever applicable.
We denote set differences with the notation $S-T = \{s\in S: s\not\in T\}$.

Let $G$ be a graph and let $A,B \subseteq V(G)$.
The subgraph induced by $A$ in $G$ is denoted $G[A]$.
We write $G-A$ to denote $G[V(G) - A]$ and if $H$ is a graph, then we write $G-H$ to denote $G-V(H)$.
For a positive integer $k$, we say that $G$ is \textit{$k$-connected} if $|V(G)|>k$ and $G-X$ is connected for all $X\subseteq V(G)$ with $|X|<k$.

An \emph{$A$-path} is a nontrivial path in $G$ such that both endpoints are in $A$ and no internal vertex is in $A$.
An \emph{$A$-$B$-path} is a (possibly trivial) path in $G$ such that one endpoint is in $A$, the other endpoint is in $B$, and the path is internally disjoint from $A\cup B$.
If $A$ or $B$ are singletons, say $A=\{a\}$, $B=\{b\}$, or both, then we also refer to such a path as an $a$-$B$-path,  $A$-$b$-path, or $a$-$b$-path respectively.
If $H_1,H_2$ are subgraphs of $G$, we also write $H_1$-$H_2$-path to mean a $V(H_1)$-$V(H_2)$-path.
If $T$ is a tree and $u,v \in V(T)$, then the unique $u$-$v$-path in $T$ is denoted by $uTv$.
Given a sequence of such paths $v_0T_1v_1, v_1T_2v_2,\dots,v_{n-1}T_nv_n$, the concatenation of these paths in their given order is denoted by $v_0T_1v_1T_2v_2\dots v_{n-1}T_nv_n$.

An \textit{$A$-bridge} of $G$ is either a subgraph consisting of an edge in $G$ with both endpoints in $A$, or a connected component $H$ of $G-A$ together with the vertices of $A$ adjacent to $H$ and the edges of $G$ with one endpoint in $A$ and the other in $V(H)$.
The \textit{attachments} of an $A$-bridge  are the vertices of the $A$-bridge that are also in $A$.

Let $\mcf$ be a family of graphs.
An \emph{$\mcf$-packing of size $k$} is a set of $k$ disjoint graphs in $\mcf$, and a \emph{half-integral $\mcf$-packing of size $k$} is a multiset of $2k$ graphs in $\mcf$ such that every vertex occurs in at most two graphs in the multiset.
A half-integral $\mcf$-packing of size $k$ can be obtained from an $\mcf$-packing of size $k$ by duplicating each element of the $\mcf$-packing. 
In a graph $G$, a vertex set $Z \subseteq V(G)$ is an \emph{$\mcf$-hitting set} if $G-Z$ does not contain a subgraph in $\mcf$.
We simply say packing or hitting set if the family $\mcf$ is clear from context.

We say that $\mcf$ satisfies the \emph{(half-integral) Erd\H{o}s-P\'osa property} if for every positive integer $k$, there exists a constant $f(k)$ such that every graph $G$ contains either a (half-integral) $\mcf$-packing of size $k$ or an $\mcf$-hitting set of size at most $f(k)$.
In this case we say that $f$ is a \emph{(half-integral) Erd\H{o}s-P\'osa function for $\mcf$}.
These definitions extend in the obvious way to families of group-labelled graphs.

The family of $\Gamma$-nonzero $A$-paths satisfy the Erd\H{o}s-P\'osa property, as shown by Wollan \cite{WolPath}:
\begin{theorem}[Theorem 1.1 in \cite{WolPath}]
\label{nzapathslemma}
Let $\Gamma$ be an abelian group. Then for all $k\in \mbn$ and $\Gamma$-labelled graph $(G,\gamma)$ with $A \subseteq V(G)$, either there exist $k$ disjoint $\Gamma$-nonzero $A$-paths or there exists $X \subseteq V(G)$ with $|X| \leq 50k^4-4$ such that $X$ intersects every $\Gamma$-nonzero $A$-path in $(G,\gamma)$. 
\end{theorem}
We remark that the theorem stated in \cite{WolPath} gives the bound $|X|\leq 50k^4$ rather than $|X|\leq 50k^4-4$, but this difference is clearly negligible in the proof in \cite{WolPath}.
The modified bound will be convenient in some of our calculations.

\subsection{Group-labelled graphs}
Let $(G,\gamma)$ be a $\Gamma$-labelled graph.
A \emph{$\Gamma$-labelled subgraph of $(G,\gamma)$} is a $\Gamma$-labelled graph $(H,\gamma|_H)$ where $H$ is a subgraph of $G$ and $\gamma|_H$ is the $\Gamma$-labelling of $H$ obtained by restricting $\gamma$ to $E(H)$.
When it is understood that $H$ is a subgraph of $G$, we simply write $(H,\gamma)$ to denote $(H,\gamma|_H)$. 
A $\Gamma$-labelled graph is \emph{$\Gamma$-bipartite} if it does not contain a $\Gamma$-nonzero cycle.

Let $g \in \Gamma$ be an element such that $2g=0$ (that is, either $g=0$ or $g$ has order two).
Given a vertex $v \in V(G)$, define a new $\Gamma$-labelling $\gamma'$ of $G$ where
\begin{align*} 
\gamma'(e) = \left\{
	\begin{array}{ll}
	\gamma(e)+g & \text{if $e$ is incident with $v$} \\
	\gamma(e) & \text{if $e$ is not incident with $v$}
	\end{array}
\right.
\end{align*}
We call this operation \emph{shifting at $v$ by $g$}.
Since $2g=0$, this preserves the weights of cycles and also of paths which do not contain $v$ as an endpoint.
We say that $(G,\gamma_1)$ and $(G,\gamma_2)$ are \emph{shift-equivalent} if one can be obtained from the other by a sequence of shifting operations.

Let $\bm 0$ denote the $\Gamma$-labelling that labels all edges 0.
Clearly, if $(G,\gamma)$ is shift-equivalent to $(G,\bm 0)$, then $(G,\gamma)$ is $\Gamma$-bipartite.
If $G$ is 3-connected, then the converse also holds as we now show. First we need the following lemma.

\begin{lemma}\label{lemmaordertwoshiftequivalent}
Let $\Gamma$ be an abelian group and let $(G,\gamma)$ be a $\Gamma$-labelled graph such that $2\gamma(e)=0$ for all $e \in E(G)$.
If $(G,\gamma)$ is $\Gamma$-bipartite, then $(G,\gamma)$ is shift-equivalent to $(G,\bm 0)$.
\end{lemma}
\begin{proof}
We proceed by induction on $|E(G)|$.
If $|E(G)|=0$ then there is nothing to prove.
Otherwise let $e=uv \in E(G)$.
Then $(G-e,\gamma)$ is also $\Gamma$-bipartite so there is a sequence of shift operations in $(G,\gamma)$ resulting in a $\Gamma$-labelling $\gamma'$ such that $\gamma'(f)=0$ for all $f \in E(G)-e$.
If $e$ is a bridge in $G$, then we obtain $(G,\bm 0)$ by possibly shifting by $\gamma'(e)$ at each vertex in one side of the bridge $e$ (here we use the assumption that $2\gamma(e)=0$). 
Otherwise, $e$ belongs to a cycle.
But since shift operations preserve weights of cycles and $(G,\gamma)$ is $\Gamma$-bipartite, it follows that $\gamma' = \bm 0$.
\end{proof}

Recall that graphs are assumed to have no loops; a cycle that is not simple consists of two parallel edges.
\begin{lemma}\label{lemma3connshiftequivalent}
Let $\Gamma$ be an abelian group and let $(G,\gamma)$ be a $\Gamma$-labelled graph such that $G$ is 3-connected and  $(G,\gamma)$ has no simple $\Gamma$-nonzero cycle.
Then $(G,\gamma)$ is $\Gamma$-bipartite and shift-equivalent to $(G,\bm 0)$.
\end{lemma}
\begin{proof}
Let $e=uv$ be an edge of $G$.
Since $G$ is 3-connected, $G$ contains two internally disjoint $u$-$v$-paths $P_1$ and $P_2$, each with at least 3 vertices.
Since the three simple cycles in $P_1 \cup P_2 \cup \{e\}$ are $\Gamma$-zero, we have $\gamma(e) = -\gamma(P_1) = -\gamma(P_2)$ and, hence, $2\gamma(e)=0$.
If there is an edge $e'$ parallel to $e$, then $\gamma(e')=\gamma(e)$ since otherwise either $P_1\cup\{e\}$ or $P_1\cup\{e'\}$ would be a simple $\Gamma$-nonzero cycle.
Thus $(G,\gamma)$ is $\Gamma$-bipartite, and since $2\gamma(e)=0$ for all $e\in E(G)$, it follows from Lemma \ref{lemmaordertwoshiftequivalent} that $(G,\gamma)$ is shift-equivalent to $(G,\bm 0)$.
\end{proof}
We reiterate that shifting in (undirected) group-labelled graphs can only be done by elements $g \in \Gamma$ such that $2g=0$.
In particular, in Lemma \ref{lemma3connshiftequivalent}, if $\Gamma$ has no element of order two, then the conclusion is that in fact $\gamma = \bm 0$.

The next two lemmas show how 3-connectivity can be used to find a $\Gamma$-nonzero path.
\begin{lemma}
\label{threepathscyclelemma}
Let $\Gamma$ be an abelian group, $(G,\gamma)$ a $\Gamma$-labelled graph, and let $C$ be a cycle in $G$.
Let $w_1,w_2,w_3$ be three distinct vertices on $C$ and, for each $i\in[3]$, let $Q_i$ denote the $w_j$-$w_k$-path in $C$ that is disjoint from $w_i$, where $\{j,k\} = [3]-\{i\}$.
\begin{enumerate}
	\item[(a)]
	If $2\gamma(Q_1) \neq 0$, then for some $j \in \{2,3\}$, the two $w_1$-$w_j$-paths in $C$ have different weights.
	\item[(b)]
	If $C$ is $\Gamma$-nonzero, then for some distinct pair $i,j \in [3]$, the two $w_i$-$w_j$-paths in $C$ have different weights.
\end{enumerate}
\end{lemma}

\begin{proof}
If the two $w_1$-$w_2$-paths in $C$ have the same weight, then $\gamma(Q_3) = \gamma(Q_1)+\gamma(Q_2)$.
If the two $w_1$-$w_3$-paths in $C$ also have the same weight, then $\gamma(Q_2) = \gamma(Q_1)+\gamma(Q_3)$.
Adding the two equalities gives $2\gamma(Q_1)=0$, proving (a). 
If, in addition, the two $w_2$-$w_3$-paths in $C$ have the same weight, then $\gamma(Q_1)=\gamma(Q_2)+\gamma(Q_3)$. 
Adding the three equalities gives $\gamma(Q_1)+\gamma(Q_2)+\gamma(Q_3)=0$, proving (b).
\end{proof}
\begin{lemma} \label{lem:threeACpathsnonzero}
	Let $\Gamma$ be an abelian group and let $(G,\gamma)$ be a $\Gamma$-labelled graph.
	Let $C$ be a $\Gamma$-nonzero cycle, let $A\subseteq V(G)$, and let $P_1,P_2,P_3$ be three disjoint $A$-$V(C)$-paths in $(G,\gamma)$. 
	Then $C\cup P_1\cup P_2\cup P_3$ contains a $\Gamma$-nonzero $A$-path.
\end{lemma}
\begin{proof}
	Let $w_i$ denote the endpoint of $P_i$ in $C$ for each $i\in[3]$, and define $Q_i$ as in Lemma \ref{threepathscyclelemma}.
	If $|A\cap V(C)|\geq 2$, then at least one of the $A$-paths in $C$ is $\Gamma$-nonzero.
	If $A\cap V(C)=\emptyset$, then the conclusion follows immediately from Lemma \ref{threepathscyclelemma}(b).
	So we may assume $|A\cap V(C)|=1$ and, without loss of generality, that $w_3\in A\cap V(C)$ (i.e.~$P_3$ is a trivial path).
	Suppose that every $A$-path in $C\cup P_1\cup P_2$ is $\Gamma$-zero.
  Then we have $\gamma(P_1)+\gamma(Q_3)+\gamma(P_2)=0=\gamma(P_1)+\gamma(Q_3)+\gamma(Q_1)$, which implies $\gamma(P_2)=\gamma(Q_1)$, and similarly we have $\gamma(P_1)=\gamma(Q_2)$.
	But this implies $0 = \gamma(P_1)+\gamma(Q_3)+\gamma(P_2) = \gamma(Q_2)+\gamma(Q_3)+\gamma(Q_1)=\gamma(C)\neq 0$, a contradiction.	
\end{proof}

\subsection{Tangles}
A \emph{separation} in a graph $G$ is an ordered pair of subgraphs $(C,D)$ such that $C$ and $D$ are edge-disjoint and $C \cup D = G$.
The \emph{order} of a separation $(C,D)$ is $|V(C) \cap V(D)|$.
A separation of order at most $k$ is a \emph{$k$-separation}.
A \emph{tangle $\mct$ of order $k$} is a set of $(k-1)$-separations of $G$ such that
\begin{enumerate}[label=\textbf{\textup{(T\arabic*)}}]
\itemsep 0.2em \parskip 0em  \partopsep=0pt \parsep 0em  
	\item \label{deftangle1}
	for every $(k-1)$-separation $(C,D)$, either $(C,D) \in \mct$ or $(D,C) \in \mct$,
	\item \label{deftangle2}
	$V(C) \neq V(G)$ for all $(C,D) \in \mct$, and
	\item \label{deftangle3}
	$C_1 \cup C_2 \cup C_3 \neq G$ for all $(C_1,D_1),(C_2,D_2),(C_3,D_3) \in \mct$. 
\end{enumerate}
Given $(C,D) \in \mct$, we say that $C$ and $D$ are the two \emph{sides} of $(C,D)$; $C$ is the \emph{$\mct$-small side} and $D$ is the \emph{$\mct$-large side} of $(C,D)$.

Tangles can be thought of as an orientation of all small order separations so that they point to some ``highly-connected'' part of the graph in a consistent manner.
For example, it is well-known that a connected graph on at least 2 vertices has a tree-decomposition into \emph{blocks} (maximal subgraphs that are either 2-connected or isomorphic to $K_2$). For each block $B$, there is a tangle of order 2 consisting of all 1-separations $(C,D)$ such that $B\subseteq D$.
Examples of higher order tangles associated with large $K_t$-models and with large walls are given in sections \ref{subsec:clique} and \ref{subsec:wall} respectively.

Here, we describe another class of high order tangles which arise from counterexamples to the Erd\H{o}s-P\'osa property.
Suppose $f:\mbn \to \mbn$ is \emph{not} a (half-integral) Erd\H{o}s-P\'osa function for a family $\mcf$ of $\Gamma$-labelled  graphs.
Let us say that $((G,\gamma),k)$ is a \emph{minimal counterexample to $f$ being a (half-integral) Erd\H{o}s-P\'osa function for $\mcf$} if $(G,\gamma)$ does not contain a (half-integral) $\mcf$-packing of size $k$ nor an $\mcf$-hitting set of size at most $f(k)$, and moreover $k$ is minimum among all such $((G,\gamma),k)$.

A standard argument appearing in various forms \cite{BruUlm,Ree99,Tho88,WolCycle} shows that, if $\mcf$ is a family of \emph{connected} $\Gamma$-labelled  graphs that does not satisfy the Erd\H{o}s-P\'osa property, then a minimal counterexample admits a tangle $\mct$ of large order such that no $\mct$-small side of a separation in $\mct$ contains a $\Gamma$-labelled subgraph in $\mcf$.
Recall that if $(G,\gamma)$ is a $\Gamma$-labelled graph and $H$ is a subgraph of $G$, then $(H,\gamma)$ denotes the $\Gamma$-labelled subgraph $(H,\gamma|_H)$ of $(G,\gamma)$.

\begin{lemma}
\label{tangleminctex}
Let $\Gamma$ be an abelian group and let $\mcf$ be a family of connected $\Gamma$-labelled graphs.
Let $f:\mbn\to\mbn$ be a function, and suppose $t$ is a positive integer such that $t\leq f(k)-2f(k-1)$ and $t\leq f(k)/3$.
If $((G,\gamma),k)$ is a minimal counterexample to $f$ being a (half-integral) Erd\H{o}s-P\'osa function for $\mcf$, then $G$ admits a tangle $\mct$ of order $t+1$ such that for each $(C,D) \in \mct$, $(C,\gamma)$ does not contain a $\Gamma$-labelled subgraph in $\mcf$ and $(D-C,\gamma)$ contains a $\Gamma$-labelled subgraph in $\mcf$.
\end{lemma}
\begin{proof}
Let $(C,D)$ be a $t$-separation in $G$.
We first show that exactly one of $(C,\gamma)$ and $(D,\gamma)$ contains a $\Gamma$-labelled subgraph in $\mcf$.
If neither side contains a $\Gamma$-labelled subgraph in $\mcf$, then since every graph in $\mcf$ is connected, $V(C\cap D)$ is an $\mcf$-hitting set of size at most $t\leq f(k)$, a contradiction.
Next suppose that both sides contain a $\Gamma$-labelled subgraph in $\mcf$.
Then neither $(C-D,\gamma)$ nor $(D-C,\gamma)$ contains a (half-integral) $\mcf$-packing of size $k-1$.
By minimality of $k$, $(C-D,\gamma)$ and $(D-C,\gamma)$ contain $\mcf$-hitting sets $X$ and $Y$ respectively, each of size at most $f(k-1)$.
Since every $\Gamma$-labelled graph in $\mcf$ is connected, every $\Gamma$-labelled subgraph of $(G-X-Y,\gamma)$ in $\mcf$ intersects $C\cap D$.
Thus $Z:= X \cup Y \cup V(C\cap D)$ is an $\mcf$-hitting set with $|Z| \leq 2f(k-1)+t \leq f(k)$, a contradiction.

Let $\mct$ be the set of $t$-separations $(C,D)$ of $G$ such that $(C,\gamma)$ does not contain a $\Gamma$-labelled subgraph in $\mcf$.
Note that $(D-C,\gamma)$ contains a $\Gamma$-labelled subgraph in $\mcf$ since otherwise $V(C\cap D)$ would again be a small hitting set. 

It remains to show that $\mct$ is a tangle.
Clearly, $\mct$ satisfies \ref{deftangle1} and \ref{deftangle2}. 
To see \ref{deftangle3}, suppose there exist $(C_1,D_1)$, $(C_2,D_2)$, $(C_3,D_3) \in \mct$ such that $C_1\cup C_2\cup C_3 = G$.
Since no $(C_i,\gamma)$ contains a $\Gamma$-labelled subgraph in $\mcf$ and every $\Gamma$-labelled graph in $\mcf$ is connected, every $\Gamma$-labelled subgraph of $(G,\gamma)$ in $\mcf$ intersects $V(C_j\cap D_j)$ for some $j\in[3]$.
But this implies that $Z:= V(C_1\cap D_1)\cup V(C_2\cap D_2) \cup V(C_3\cap D_3)$ is an $\mcf$-hitting set with $|Z| \leq 3t \leq f(k)$, a contradiction.
\end{proof}

Let $\mct$ be a tangle of order $k$ in a graph $G$.
Given a positive integer $k' \leq k$, the set $\mct'$ of $(k'-1)$-separations $(C,D)$ in $G$ such that $(C,D) \in \mct$ is a tangle of order $k'$, called the \emph{truncation} of $\mct$ to order $k'$ (see section 6 in \cite{RobSeyX}).
If $X\subseteq V(G)$ is a set of fewer than $k$ vertices, then there is a tangle of order $k-|X|$ in $G-X$, consisting of the $(k-|X|-1)$-separations of $G-X$ which can be written as $(C-X, D-X)$ for some $(C,D)\in\mct$ ((6.2) in \cite{RobSeyX}).
We denote this tangle by $\mct-X$.

\subsection{3-blocks}
Due to the 3-connectivity condition that arises naturally in undirected group-labelled graphs (e.g. Lemmas \ref{lemma3connshiftequivalent}--\ref{lem:threeACpathsnonzero}), we will need to work with \emph{3-blocks} of graphs.
The decomposition of 2-connected graphs into a tree structure of 3-connected components was first given by Tutte \cite{Tut}.
Here, we use a definition of 3-blocks adapted from the terminology of \emph{$k$-blocks} studied in \cite{Mad, CarDieHamHun, CarDieHunSte}.

Let $G$ be a graph.
A separation $(C,D)$ of $G$ \emph{properly separates} two vertices $u$ and $v$ if $V(C-D)$ and $V(D-C)$ each contain one of $\{u,v\}$.
A vertex set $U\subseteq V(G)$ is \textit{2-inseparable} in $G$ if no two vertices of $U$ are properly separated by a 2-separation in $G$ (that is, for every 2-separation $(C,D)$ of $G$, we have either $U\subseteq V(C)$ or $U\subseteq V(D)$).
A 2-inseparable set $U$ is \textit{maximal} if there does not exist a 2-inseparable set properly containing $U$. 
Observe that if $U$ is a maximal 2-inseparable set, then every $U$-bridge of $G$ has at most 2 attachments; otherwise, the $U$-bridge would contain a vertex $v$ and three $v$-$U$-paths intersecting only at $v$, which implies that $U\cup\{v\}$ is 2-inseparable, contradicting the maximality of $U$.

Let $(G,\gamma)$ be a $\Gamma$-labelled graph.
A \emph{3-block of $(G,\gamma)$} is a $\Gamma$-labelled graph $(B,\gamma_B)$ obtained from a maximal 2-inseparable set $U=V(B)$ of $G$ as follows:
For each $u,v \in U$ and $\alpha \in \Gamma$, if there exists a $U$-path in $(G,\gamma)$ with endpoints $u,v$ and weight $\alpha$, then add a new (possibly parallel) edge $uv$ with label $\alpha$.
Note that $B$ may not be a subgraph of $G$.
For example, if $G$ is a subdivision of a simple 3-connected graph $H$, then $V(H)$ is a maximal 2-inseparable set in $G$, and the corresponding 3-block is $(H,\gamma_H)$ where for each $e\in E(H)$, $\gamma_H(e)$ is the weight of the path in $(G,\gamma)$ corresponding to $e$.
Also observe that if $|U|\geq 4$, then $B$ is a 3-connected graph.
Since every $V(B)$-bridge of $G$ has at most 2 attachments, the following proposition is immediate.
\begin{prop} \label{prop:3blockpathcycle}
	Let $\Gamma$ be an abelian group and let $(B,\gamma_B)$ be a 3-block of a $\Gamma$-labelled graph $(G,\gamma)$.
	For each subgraph $P_B$ of $(B,\gamma_B)$ that is either a simple cycle or a path, there exists a cycle or a path $P$ respectively in $(G,\gamma)$ with weight equal to the weight of $P_B$ such that $V(P)\cap V(B)=V(P_B)$, and the order of the vertices in $V(P_B)$ appearing in $P$ is the same as the order appearing in $P_B$. 
\end{prop}
We will primarily be concerned with a particular 3-block associated with a given tangle.
\begin{lemma}\label{lem:tangleuniquelarge3block}
	Let $\mct$ be a tangle of order $3$ in a graph $G$. 
	Then there is a unique maximal 2-inseparable set $U$ that is contained in every $\mct$-large side (that is, we have $U\subseteq V(D)$ for all $(C,D)\in\mct$).
	Moreover, we have $|U|\geq 4$.
\end{lemma}
\begin{proof}
	If $U_1$ and $U_2$ are distinct maximal 2-inseparable sets of $G$, then there is a 2-separation of $G$ properly separating a vertex in $U_1$ and a vertex in $U_2$.
	It follows from \ref{deftangle1} that there is at most one maximal 2-inseparable set of $G$ contained in every $\mct$-large side.
	
	We now show that such a maximal 2-inseparable set exists.
	Let us say that a 2-separation $(C,D)$ of $G$ is \emph{good} if $(C,D)\in\mct$ and, if $|V(C\cap D)|=2$, then there is a path in $C$ connecting the two vertices of $V(C\cap D)$.
	Let us also say that a 2-separation $(C,D)$ is \emph{tight} if it is good and there does not exist a good separation $(C',D')\in\mct$ such that $C\subseteq C'$ and $D'\subsetneq D$. 
	Let $U \subseteq V(G)$ be the set of vertices which belong to the intersection $V(C\cap D)$ of some tight separation $(C,D)\in\mct$. 
	We claim that $U$ is a maximal 2-inseparable set contained in every $\mct$-large side, and that $|U|\geq 4$.

	Let us first show that $|U|\geq 3$.
	Suppose $|U| \leq 2$.
	Let $(C,D)\in\mct$ be a separation such that $U\subseteq V(C)$  and, subject to this condition, $D$ is minimal. Note that because $|U|\leq 2$, the first condition is satisfied by the separation $(G_U, G)\in\mct$ where $G_U$ denotes the subgraph of $G$ with vertex set $U$ and no edges.
	Now by \ref{deftangle2}, $V(D-C)$ is nonempty, and it follows from \ref{deftangle3} and the minimality of $D$ that $D-C$ is connected.
	Let $v\in V(D-C)$.
	The separation $(G[\{v\}],G)$ is good but not tight (since $v\not\in U$), so there is a tight separation $(C_v,D_v)\in\mct$ such that $V(C_v\cap D_v) \subseteq U$ and $v\in V(C_v-D_v)$.
  Moreover, since $D-C$ is connected and since $U\subseteq V(C)$, $D-C$ is contained in the connected component of $C_v - V(C_v\cap D_v)$ containing $v$. It follows from the minimality of $D$ that $D \subseteq C_v$, and that $G = C \cup D = C \cup C_v$, contradicting \ref{deftangle3}.
	We thus have $|U|\geq 3$.

	Next we show that $U$ is contained in every $\mct$-large side. 
	Suppose to the contrary that there exists $(C,D)\in \mct$ such that $u \in V(C-D)$ for some $u\in U$.
	By the definition of $U$, there is a tight separation $(C_u,D_u)\in \mct$ such that $u\in V(C_u\cap D_u)$.
	Consider the separation $(C\cup C_u, D\cap D_u)$.
	Since the orders of $(C,D)$ and $(C_u,D_u)$ are each at most 2, and since $u \in V(C-D)$, the order of $(C\cup C_u, D\cap D_u)$ is at most 3. If its order is equal to 3, then the orders of $(C,D)$ and $(C_u,D_u)$ are both 2, $V(C\cap D) \subseteq V(D_u-C_u)$, and $V(C_u\cap D_u) - \{u\} \subseteq V(D-C)$ (see Figure \ref{FigSeparations}). 
	Since $V(C\cap D) \cap V(C_u)$ is empty, this contradicts the assumption that $(C_u,D_u)$ is good.

	\begin{figure}
	\centering
	\begin{subfigure}{0.4\textwidth} \centering
\begingroup%
  \makeatletter%
  \providecommand\color[2][]{%
    \errmessage{(Inkscape) Color is used for the text in Inkscape, but the package 'color.sty' is not loaded}%
    \renewcommand\color[2][]{}%
  }%
  \providecommand\transparent[1]{%
    \errmessage{(Inkscape) Transparency is used (non-zero) for the text in Inkscape, but the package 'transparent.sty' is not loaded}%
    \renewcommand\transparent[1]{}%
  }%
  \providecommand\rotatebox[2]{#2}%
  \newcommand*\fsize{\dimexpr\f@size pt\relax}%
  \newcommand*\lineheight[1]{\fontsize{\fsize}{#1\fsize}\selectfont}%
  \ifx\svgwidth\undefined%
    \setlength{\unitlength}{119.26221549bp}%
    \ifx\svgscale\undefined%
      \relax%
    \else%
      \setlength{\unitlength}{\unitlength * \real{\svgscale}}%
    \fi%
  \else%
    \setlength{\unitlength}{\svgwidth}%
  \fi%
  \global\let\svgwidth\undefined%
  \global\let\svgscale\undefined%
  \makeatother%
  \begin{picture}(1,0.81503702)%
    \lineheight{1}%
    \setlength\tabcolsep{0pt}%
    \put(0,0){\includegraphics[width=\unitlength,page=1]{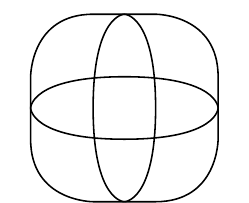}}%
    \put(0.31189484,0.78922731){\makebox(0,0)[lt]{\lineheight{1.25}\smash{\begin{tabular}[t]{l}$C$\end{tabular}}}}%
    \put(0.59488463,0.78922731){\makebox(0,0)[lt]{\lineheight{1.25}\smash{\begin{tabular}[t]{l}$D$\end{tabular}}}}%
    \put(-0.00253839,0.51881477){\makebox(0,0)[lt]{\lineheight{1.25}\smash{\begin{tabular}[t]{l}$C_u$\end{tabular}}}}%
    \put(-0.00253839,0.19180429){\makebox(0,0)[lt]{\lineheight{1.25}\smash{\begin{tabular}[t]{l}$D_u$\end{tabular}}}}%
    \put(0,0){\includegraphics[width=\unitlength,page=2]{FigSeparations.pdf}}%
    \put(0.26787415,0.39933013){\makebox(0,0)[lt]{\lineheight{1.25}\smash{\begin{tabular}[t]{l}$u$\end{tabular}}}}%
    \put(0,0){\includegraphics[width=\unitlength,page=3]{FigSeparations.pdf}}%
    \put(0.93259343,0.41043756){\makebox(0,0)[lt]{\lineheight{1.25}\smash{\begin{tabular}[t]{l}$\ $\end{tabular}}}}%
  \end{picture}%
\endgroup%

	\caption{}
	\label{FigSeparations}
	\end{subfigure}
	\begin{subfigure}{0.4\textwidth}\centering 
\begingroup%
  \makeatletter%
  \providecommand\color[2][]{%
    \errmessage{(Inkscape) Color is used for the text in Inkscape, but the package 'color.sty' is not loaded}%
    \renewcommand\color[2][]{}%
  }%
  \providecommand\transparent[1]{%
    \errmessage{(Inkscape) Transparency is used (non-zero) for the text in Inkscape, but the package 'transparent.sty' is not loaded}%
    \renewcommand\transparent[1]{}%
  }%
  \providecommand\rotatebox[2]{#2}%
  \newcommand*\fsize{\dimexpr\f@size pt\relax}%
  \newcommand*\lineheight[1]{\fontsize{\fsize}{#1\fsize}\selectfont}%
  \ifx\svgwidth\undefined%
    \setlength{\unitlength}{120.11863004bp}%
    \ifx\svgscale\undefined%
      \relax%
    \else%
      \setlength{\unitlength}{\unitlength * \real{\svgscale}}%
    \fi%
  \else%
    \setlength{\unitlength}{\svgwidth}%
  \fi%
  \global\let\svgwidth\undefined%
  \global\let\svgscale\undefined%
  \makeatother%
  \begin{picture}(1,0.80922611)%
    \lineheight{1}%
    \setlength\tabcolsep{0pt}%
    \put(0,0){\includegraphics[width=\unitlength,page=1]{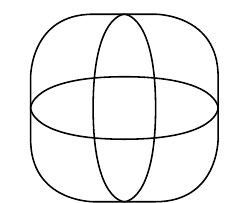}}%
    \put(0.30967084,0.78360042){\makebox(0,0)[lt]{\lineheight{1.25}\smash{\begin{tabular}[t]{l}$C$\end{tabular}}}}%
    \put(0.59064308,0.78360042){\makebox(0,0)[lt]{\lineheight{1.25}\smash{\begin{tabular}[t]{l}$D$\end{tabular}}}}%
    \put(-0.00252029,0.51511585){\makebox(0,0)[lt]{\lineheight{1.25}\smash{\begin{tabular}[t]{l}$C_u$\end{tabular}}}}%
    \put(-0.00252029,0.19043705){\makebox(0,0)[lt]{\lineheight{1.25}\smash{\begin{tabular}[t]{l}$D_u$\end{tabular}}}}%
    \put(0,0){\includegraphics[width=\unitlength,page=2]{FigSeparations2.pdf}}%
    \put(0.27763904,0.37751151){\makebox(0,0)[lt]{\lineheight{1.25}\smash{\begin{tabular}[t]{l}$u$\end{tabular}}}}%
    \put(0,0){\includegraphics[width=\unitlength,page=3]{FigSeparations2.pdf}}%
    \put(0.75281566,0.36831901){\makebox(0,0)[lt]{\lineheight{1.25}\smash{\begin{tabular}[t]{l}$u'$\end{tabular}}}}%
    \put(0.93307403,0.36102219){\makebox(0,0)[lt]{\lineheight{1.25}\smash{\begin{tabular}[t]{l}$\ $\end{tabular}}}}%
  \end{picture}%
\endgroup%

	\caption{}
	\label{FigSeparations2}
	\end{subfigure}
	\caption{
	Since $(C_u,D_u)$ is good, if $|V(C_u\cap D_u)|=2$, then there is a path in $C_u$ connecting the two vertices of $V(C_u\cap D_u)$.
	In (b), $(C\cup C_u, D\cap D_u)$ is not good, so there is a 1-separation $(C',D')$ with $V(C'\cap D')=\{u'\}$ which violates the tightness of $(C_u,D_u)$.}

\end{figure}

	We may thus assume that $(C\cup C_u,D\cap D_u)$ is a 2-separation. 
	It follows from \ref{deftangle1} and \ref{deftangle3} that $(C\cup C_u,D\cap D_u)\in\mct$.
	Note that $C_u\subseteq C\cup C_u$ and, since $u\in V(D_u-D)$, we have  $D\cap D_u \subsetneq D_u$.
	By the assumption that $(C_u,D_u)$ is tight, we have that $(C\cup C_u,D\cap D_u)$ is not good; that is, $|V((C\cup C_u)\cap(D\cap D_u))|=2$ and every $V((C\cup C_u)\cap(D\cap D_u))$-bridge of $C\cup C_u$ has at most one attachment (see Figure \eqref{FigSeparations2} for one possible configuration).
	But this implies that there is a 1-separation $(C',D')\in\mct$ such that $C_u\subseteq C'$, $V(C'\cap D')\subseteq V((C\cup C_u)\cap (D\cap D_u))$, and $D'\subseteq D_u-\{u\} \subsetneq D_u$, contradicting the tightness of $(C_u,D_u)$.

	Hence, $U$ is contained in every $\mct$-large side.
	Note that this also implies that $U$ is 2-inseparable.
	To see that $U$ is a maximal 2-inseparable set, let $u'\in V(G)-U$.
	Then the separation $(G[\{u'\}],G)$ is good but not tight, so there exists a tight separation $(C,D)\in\mct$ such that $u'\subseteq V(C-D)$, $D \subsetneq G$, and $V(C\cap D) \subseteq U$.
	Since $|U|\geq 3$ and $U$ is contained in $V(D)$, this implies that $(C,D)$ properly separates $u'$ from a vertex in $U$.
	We conclude that $U$ is the unique maximal 2-inseparable set that is contained in every $\mct$-large side. 
	
	It remains to show that $|U|\geq 4$.
	We have already shown $|U|\geq 3$, so suppose $|U|=3$ and write $U=\{u_1,u_2,u_3\}$.
	For $i\in[3]$, consider the separation $(C_i,D_i)\in\mct$ such that $V(C_i\cap D_i)=U-\{u_i\}$ and, subject to this condition, $C_i$ is maximal.
	Since $U\subseteq D_i$ for all $i\in[3]$, every $U$-bridge of $G$ (which has at most two attachments in $U$) is contained in $C_i$ for some $i\in[3]$.
	This implies that $C_1\cup C_2 \cup C_3 = G$, contradicting \ref{deftangle3}.
\end{proof}

\begin{lemma} \label{lem:Tlarge3block}
	Let $\mct$ be a tangle of order $k\geq 3$ in a graph $G$, and let $X\subseteq V(G)$ with $|X|\leq k-3$.
	Then there is a unique maximal 2-inseparable set $U$ of $G-X$  such that $U\cup X$ is not contained in any $\mct$-small side.
	Moreover, we have $|U|\geq 4$.
\end{lemma}
\begin{proof}
	Let $\mct_X$ denote the tangle of order 3 in $G-X$ that is a truncation of $\mct-X$ (which has order $k-|X| \geq 3$).
	Let $U$ be the unique maximal 2-inseparable set of $G-X$ contained in every $\mct_X$-large side, given by Lemma \ref{lem:tangleuniquelarge3block}.
	If $U'$ is a maximal 2-inseparable set of $G-X$ distinct from $U$, then there is a 2-separation $(C_X,D_X)\in\mct_X$ such that $U'\subseteq V(C_X)$, so $U'\cup X$ is contained in $G[V(C_X) \cup X]$, which forms the $\mct$-small side of a $(k-1)$-separation in $G$.
	
	Now suppose that $U\cup X$ is contained in a $\mct$-small side. Let $(C,D)\in\mct$ such that $U\cup X\subseteq V(C)$ and, subject to this condition, $D$ is minimal.
	Then $D-X$ is connected by \ref{deftangle3}.
	Since $U$ is a maximal 2-inseparable set of $G-X$ and $U\cup X \subseteq V(C)$, it follows that $D-X$ is contained in a $U$-bridge $H$ of $G-X$. 
	Note that $H$ has at most two attachments in $U$ because $U$ is a maximal 2-inseparable set.
	Hence $H$ forms one side of a 2-separation in $G-X$, and it is in fact the $\mct_X$-small side of this 2-separation by \ref{deftangle3}. This implies that $G[V(H)\cup X]$ forms the $\mct$-small side of a $(k-1)$-separation in $\mct$.
	Since $D\subseteq G[V(H)\cup X]$, we have $G = C\cup D = C \cup G[V(H)\cup X]$,  contradicting \ref{deftangle3}.
	Therefore, $U$ is the unique maximal 2-inseparable set of $G-X$ such that $U\cup X$ is not contained in any $\mct$-small side.
	Finally, note that $|U|\geq 4$ by Lemma \ref{lem:tangleuniquelarge3block}.
\end{proof}

Let $(G,\gamma)$ be a $\Gamma$-labelled graph.
Let $\mct$ be a tangle of order $k\geq 3$ in $G$ and let $X\subseteq V(G)$ with $|X|\leq k-3$.
The \emph{$\mct$-large 3-block of $(G-X,\gamma)$} is the 3-block $(B,\gamma_B)$ of $(G-X,\gamma)$ obtained from the unique maximal 2-inseparable set $U$ of $G-X$ such that $U\cup X$ is not contained in any $\mct$-small side, as given by Lemma \ref{lem:Tlarge3block}.
Note that $B$ is 3-connected, since $|U|\geq 4$ by Lemma \ref{lem:Tlarge3block}.

\subsection{$K_m$-models} \label{subsec:clique}
Let $v_1,\dots,v_m$ denote the vertices of the complete graph $K_m$.
A \emph{$K_m$-model} $\mu$ consists of a collection of disjoint trees $\mu(v_i)$ for $i \in [m]$ and edges $\mu(v_iv_j)$ for distinct $i,j\in [m]$ such that $\mu(v_iv_j)$ has one endpoint in $\mu(v_i)$ and the other in $\mu(v_j)$.
For a subset $U\subseteq \{v_1,\dots,v_m\}$, we denote by $\mu[U]$ the graph defined by
\begin{equation*}
\mu[U] = \bigcup_{v_i \in U} \mu(v_i) \cup \bigcup_{v_i,v_j \in U} \mu(v_iv_j).
\end{equation*}
If we are given $U$ explicitly, say $U=\{v_{i_1},\dots,v_{i_k}\}$, then we simply write $\mu[v_{i_1},\dots,v_{i_k}]$.
When there is no room for ambiguity, we also write $\mu$ to refer to the graph $\mu[V(K_m)]=\mu[\{v_1,\dots,v_m\}]$.
The \emph{$K_{|U|}$-submodel $\pi$ of $\mu$ restricted to $U$} is the $K_{|U|}$-model of the complete graph on the vertex set $U$, given by $\pi(v_i)=\mu(v_i)$ and $\pi(v_iv_j)=\mu(v_iv_j)$ for all $v_i,v_j \in U$.
If $n\leq m$ and $\eta$ is a $K_{n}$-model such that each tree of $\eta$ contains some tree of $\mu$, then we say that $\eta$ is an \emph{enlargement} of $\mu$.
Note that a submodel of $\mu$ is also an enlargement of $\mu$.

It is easy to see that a graph $G$ contains a $K_m$-model (that is, $G$ contains as a subgraph $\mu[V(K_m)]$ for some $K_m$-model $\mu$) if and only if $K_m$ can be obtained from $G$ by a sequence of vertex deletions, edge deletions, and edge contractions; the trees $\mu(v_i)$ correspond to the subgraphs of $G$ that were contracted to form the vertex $v_i$ of $K_m$.

We now describe a tangle associated with a $K_m$-model.
Consider an $(m-1)$-separation $(C,D)$ in a graph $G$ containing a $K_m$-model $\mu$.
Then there is exactly one side of $(C,D)$ that intersects every tree $\mu(v_i)$ of $\mu$.
Indeed, since $\mu(v_1),\dots,\mu(v_m)$ are disjoint trees and $|V(C\cap D)|\leq m-1$, there is a tree $\mu(v_i)$ disjoint from $C\cap D$; assume without loss of generality that $\mu(v_i)\subseteq D-C$.
Since every other tree $\mu(v_j)$ is adjacent to $\mu(v_i)$ by the edge $\mu(v_iv_j) \in E(G)$, $D$ is the unique side of $(C,D)$ intersecting every tree of $\mu$.
The set of all such $(m-1)$-separations satisfies properties \ref{deftangle1} and \ref{deftangle2}.
Note however that it may violate \ref{deftangle3} if $m\geq 3$ (for example, consider $K_3$).

Define $k = \lceil \frac{2m}{3}\rceil$ and let $\mct_\mu$ be the set of $(k-1)$-separations $(C,D)$ in $G$ such that $D$ intersects every tree of $\mu$.
It is straightforward to verify that $\mct_\mu$ satisfies \ref{deftangle3} (see  \cite{RobSeyX}), hence $\mct_\mu$ is a tangle of order $k$.
We call $\mct_\mu$ the \emph{tangle induced by $\mu$}.
If $\eta$ is a submodel or an enlargement of $\mu$, then $\mct_\eta$ is a truncation of $\mct_\mu$.

Let $(G,\gamma)$ be a $\Gamma$-labelled graph.
We say that a $K_m$-model $\mu$ in $G$ is \emph{$\Gamma$-bipartite} if for every choice of four distinct indices $i,j,k,l \in [m]$, we have that $(\mu[v_i,v_j,v_k,v_l],\gamma)$ is a $\Gamma$-bipartite $\Gamma$-labelled graph.
Oppositely, if for every choice of four distinct indices $i,j,k,l\in [m]$, $(\mu[v_i,v_j,v_k,v_l],\gamma)$ contains a $\Gamma$-nonzero cycle, then we say that $\mu$ is \emph{$\Gamma$-odd}.

\begin{remark} \label{rem:Ktbip}
\normalfont
For \emph{directed} group-labelled graphs, the definition of $\Gamma$-bipartite (resp. $\Gamma$-odd) $K_m$-models \cite{HuyJooWol} only require that for every \emph{three} distinct indices $i,j,k \in [m]$, $(\mu[v_i,v_j,v_k],\gamma)$ is $\Gamma$-bipartite (resp. not $\Gamma$-bipartite).
However, the property we actually want of a $\Gamma$-bipartite $K_m$-model $\mu$ is for $(\mu[V(K_m)],\gamma)$, as a $\Gamma$-labelled graph, to be $\Gamma$-bipartite.
And in contrast to the directed setting, the above condition does not suffice for undirected group-labelled graphs.
For example, let $\Gamma=\mbz/3\mbz$, $G=K_m$, $\gamma(e)=1$ for all $e \in E(G)$, and let $\mu$ be a $K_m$-model in $G$.
Then every triangle in $G$ is $\Gamma$-zero, so $(\mu[v_i,v_j,v_k],\gamma)$ is $\Gamma$-bipartite for all distinct $i,j,k \in [m]$, but $(\mu(V(K_m)),\gamma)$ is clearly not $\Gamma$-bipartite for $m\geq 4$.
We will show in Lemma \ref{bipmodelpathlemma} that our definition of $\Gamma$-bipartite $K_m$-models does give this desired property.
\end{remark}

\subsection{Walls} \label{sectionwalls} \label{subsec:wall}
Let $r,s \geq 2$ be integers.
An \emph{$r\times s$-grid} is a graph with vertex set $[r]\times [s]$ and edge set
\begin{equation*}
\left\{(i,j)(i',j'): |i-i'|+|j-j'|=1\right\}.
\end{equation*}
An \emph{elementary $r\times s$-wall} is the subgraph of an $(r+1)\times (2s+2)$-grid obtained by deleting the edges in the set
\begin{equation*}
\left\{(2i-1,2j)(2i,2j): i \in \left[\lceil \tfrac{r}{2}\rceil\right], j \in [s+1]\right\}
\cup
\big\{(2i,2j-1)(2i+1,2j-1): i \in \left[\lceil \tfrac{r-1}{2}\rceil\right], j\in [s+1]\big\},
\end{equation*}
then deleting the two vertices of degree 1.
An \emph{elementary $r$-wall} $W$ is an elementary $r\times r$-wall, and $r$ is the \textit{size} of $W$.
Figure \ref{elem6wall} shows an elementary $6\times 7$-wall.

\begin{figure}[t]
	\centering
	\resizebox{0.6\textwidth}{!}{\begin{tikzpicture}
\colorlet{hellgrau}{black!30!white}
\tikzstyle{smallvx}=[thick,circle,inner sep=0.cm, minimum size=2mm, fill=white, draw=black]
\tikzstyle{smallblackvx}=[thick,circle,inner sep=0.cm, minimum size=2mm, fill=black, draw=black]
\tikzstyle{squarevx}=[thick,rectangle,inner sep=0.cm, minimum size=3mm, fill=white, draw=black]
\tikzstyle{hedge}=[line width=1pt]
\tikzstyle{markline}=[draw=hellgrau,line width=6pt]

\def\wallheight{6}
\def\wallwidth{7}
\def\brickheight{0.7}

\pgfmathtruncatemacro{\lastrow}{\wallheight}
\pgfmathtruncatemacro{\penultimaterow}{\wallheight-1}
\pgfmathtruncatemacro{\lastrowshift}{mod(\wallheight,2)}
\pgfmathtruncatemacro{\lastx}{2*\wallwidth+1}

\def\exhpath{3}
\def\exvpath{2}

\draw[hedge] (\brickheight,0) -- (2*\wallwidth*\brickheight+\brickheight,0);
\foreach \i in {1,...,\penultimaterow}{
  \ifnum \i=\exhpath
    \draw[markline] (0,\i*\brickheight) -- (2*\wallwidth*\brickheight+\brickheight,\i*\brickheight);
  \fi
  {\draw[hedge] (0,\i*\brickheight) -- (2*\wallwidth*\brickheight+\brickheight,\i*\brickheight);}
}
\draw[hedge] (\lastrowshift*\brickheight,\lastrow*\brickheight) to ++(2*\wallwidth*\brickheight,0);

\foreach \j in {0,2,...,\penultimaterow}{
  \foreach \i in {0,...,\wallwidth}{
    \ifnum \i=\exvpath {
      \draw[line width=2.5pt] (2*\i*\brickheight+\brickheight,\j*\brickheight) to ++(0,\brickheight);
      \ifnum \j>0 
        \draw[line width=2.5pt] (2*\i*\brickheight+\brickheight,\j*\brickheight) to ++(-\brickheight,0);
      \fi
    }
    \else
      {\draw[hedge] (2*\i*\brickheight+\brickheight,\j*\brickheight) to ++(0,\brickheight);}
    \fi
  }
}
\foreach \j in {1,3,...,\penultimaterow}{
  \foreach \i in {0,...,\wallwidth}{
    \ifnum \i=\exvpath {
      \draw[line width=2.5pt] (2*\i*\brickheight,\j*\brickheight) to ++(0,\brickheight);
      \draw[line width=2.5pt] (2*\i*\brickheight,\j*\brickheight) to ++(\brickheight,0);
    }
    \else
      \draw[hedge] (2*\i*\brickheight,\j*\brickheight) to ++(0,\brickheight);
    \fi
  }
}
\foreach \i in {1,...,\lastx}{
  \node[smallvx] (w\i w0) at (\i*\brickheight,0){};
}
\foreach \j in {1,...,\penultimaterow}{
  \foreach \i in {0,...,\lastx}{
    \node[smallvx] (w\i w\j) at (\i*\brickheight,\j*\brickheight){};
  }
}
\foreach \i in {1,...,\lastx}{
  \ifodd\i
    \node[smallvx] (w\i w\lastrow) at (\i*\brickheight+\lastrowshift*\brickheight-\brickheight,\lastrow*\brickheight){};
  \else
    \node[smallblackvx] (w\i w\lastrow) at (\i*\brickheight+\lastrowshift*\brickheight-\brickheight,\lastrow*\brickheight){};
  \fi
}
\node[squarevx] (w1 w\lastrow) at (1*\brickheight+\lastrowshift*\brickheight-\brickheight,\lastrow*\brickheight){};
\node[squarevx] (w\lastx w\lastrow) at (\lastx*\brickheight+\lastrowshift*\brickheight-\brickheight,\lastrow*\brickheight){};
\node[squarevx] (w1 w\lastrow) at (2*\brickheight-\lastrowshift*\brickheight-\brickheight,0){};
\node[squarevx] (w\lastx w\lastrow) at (\lastx*\brickheight-\lastrowshift*\brickheight,0){};
\end{tikzpicture}}
\caption{An elementary $6\times 7$-wall. The four corners are marked by square vertices and its top nails are filled black. The fourth horizontal path is highlighted in grey and the third vertical path is marked bold.}
\label{elem6wall}
\end{figure}

Let $W$ be an elementary $r\times s$-wall.
There is a set of $r+1$ disjoint paths called the \emph{horizontal paths} of $W$ such that each path has vertex set $\{(i',j') \in V(W): i'=i\}$ for some $i \in [r+1]$.
Moreover, there is a unique set of $s+1$ disjoint paths from $\{(1,j) \in V(W)\}$ to $\{(r+1,j)\in V(W)\}$; these are called the \emph{vertical paths} of $W$.
An \emph{orientation} of $W$ is an ordering of the horizontal paths and of the vertical paths in one of the four following ways: Fix a planar embedding of $W$ such that the horizontal paths of $W$ are horizontal line segments in the plane, order the horizontal paths from top to bottom, and order the vertical paths from left to right. Given an orientation of $W$, we denote by $P^{(h)}_i, i \in [r+1]$ the $i$-th horizontal path and we denote by $P^{(v)}_j, j\in [s+1]$ the $j$-th vertical path. 
Note that a different orientation reverses the order of the vertical and/or horizontal paths. We usually assume implicitly that each wall has a fixed orientation.

For $(i,j)\in [r]\times [s]$, the \emph{$(i,j)$-th brick} is the facial cycle of length 6 contained in the union $P_i^{(h)}\cup P_{i+1}^{(h)} \cup P_j^{(v)} \cup P_{j+1}^{(v)}$.
The \emph{perimeter} of $W$ is the cycle in the union $P_1^{(h)}\cup P_{r+1}^{(h)} \cup P_1^{(v)} \cup P_{s+1}^{(v)}$.
The \emph{corners} of $W$ are the four vertices that are endpoints of $P_1^{(h)}$ or $P_{r+1}^{(h)}$ (equivalently, the endpoints of $P_1^{(v)}$ or $P_{s+1}^{(v)}$).
The \emph{nails} of $W$ are the vertices of degree 2 that are not corners.
The \emph{top nails} of $W$ are the nails in the first horizontal path of $W$.

An \emph{$r\times s$-wall} or \emph{$r$-wall} is a subdivision of an elementary $r\times s$-wall or $r$-wall respectively.
The corners and nails of an $r$-wall are the vertices corresponding to the corners and nails respectively of the elementary $r$-wall before subdivision.
The \emph{branch vertices} $b(W)$ of a wall $W$ are the vertices of the elementary wall before subdivision, consisting of its corners, nails, and the vertices of degree 3 in $W$.
All other terminology on elementary walls in the preceding paragraphs extend to walls in the obvious way.
Note that, given a wall $W$, there may be many possible choices for its corners and nails.
If this choice is not explicitly described, it is assumed implicitly that a choice of corners and nails accompanies each wall. 
Other definitions which depend on the choice of corners and nails are assumed to be with respect to this implicit choice.

Let $r' \leq r$ and $s'\leq s$.
An \emph{$r'\times s'$-subwall} of an $r\times s$-wall $W$ is an $r'\times s'$-wall $W'$ that is a subgraph of $W$ such that each horizontal and vertical path of $W'$ is a subpath of some horizontal and vertical path of $W$ respectively.
If, in addition, the set of indices $i$ such that $P_i^{(h)}$ contains a horizontal path of $W'$ and the set of indices $j$ such that $P_j^{(v)}$ contains a vertical path of $W'$ both form contiguous subsets of $[r+1]$ and $[s+1]$ respectively, then we say that $W'$ is a \emph{compact} subwall of $W$.
A subwall $W'$ is \emph{$k$-contained} in $W$ if $P_i^{(h)}$ and $P_j^{(v)}$ are disjoint from $W'$ for all $i,j \leq k$ and for all $i > r-k+1$ and $j>s-k+1$.
If $W'$ is 1-contained in $W$, then there is a unique choice of corners and nails of $W'$ such that they have degree 3 in $W$. We call these the \emph{corners and nails of $W'$ with respect to $W$}. The orientation of $W'$ is inherited from $W$; in other words, the ordering of the horizontal/vertical paths of $W'$ agrees with the ordering in $W$.

Let $W$ be an $r$-wall contained in a graph $G$.
If $(C,D)$ is an $r$-separation in $G$, then exactly one side of $(C,D)$ contains a horizontal path of $W$, and the set of $r$-separations $(C,D)$ such that $D$ contains a horizontal path forms a tangle $\mct_W$ of order $r+1$ \cite{RobSeyX}, called the \emph{tangle induced by $W$}.
If $W'$ is a subwall of $W$, then $\mct_{W'}$ is a truncation of $\mct_W$.

We will use the following fundamental result of Robertson and Seymour.

\begin{theorem}[(2.3) in \cite{RobSeyTho}]
\label{tanglewallthm}
For all $r \in \mbn$, there exists $\omega(r) \in \mbn$ such that if a graph $G$ admits a tangle $\mct$ of order $\omega(r)$, then $G$ contains an $r$-wall $W$ such that $\mct_W$ is a truncation of $\mct$.
\end{theorem}

A $\Gamma$-labelled wall $(W,\gamma)$ is \emph{facially $\Gamma$-odd} if every brick is a $\Gamma$-nonzero cycle. We say that $(W,\gamma)$ is \emph{strongly \mbox{$\Gamma$-bipartite}} if it is shift-equivalent to a $\Gamma$-labelled graph $(W,\gamma')$ where every $b(W)$-path in $(W,\gamma')$ is $\Gamma$-zero.
Note that this definition depends on the choice of the corners and nails of $W$.
This is a slightly stronger condition than just requiring $(W,\gamma)$ as a $\Gamma$-labelled graph to be $\Gamma$-bipartite, but the difference is superficial; if $(W,\gamma)$ is a $\Gamma$-bipartite $r$-wall with $r\geq 3$, then by Lemma \ref{lemma3connshiftequivalent} and Proposition \ref{prop:3blockpathcycle}, the $(r-2)$-wall 1-contained in $(W,\gamma)$ with the choice of corners and nails with respect to $W$ is strongly $\Gamma$-bipartite.

\subsubsection{Flat walls}

Let $(C,D)$ be a separation of order $k \leq 3$ such that there is a vertex $v \in V(D-C)$ and $k$ paths from $v$ to $V(C\cap D)$ pairwise disjoint except at $v$.
Let $H$ be the graph obtained from $C$ by adding an edge between each nonadjacent pair of vertices in $C\cap D$.
Then we say that $H$ is an \emph{elementary reduction of $G$ with respect to $(C,D)$}.
Let $X\subseteq V(G)$.
An \emph{$X$-reduction} of $G$ is a graph that can be obtained from $G$ by a sequence of elementary reductions with respect to separations $(C,D)$ for which $X\subseteq V(C)$.

Let $W$ be a wall in a graph $G$ and let $O$ denote the perimeter of $W$.
Suppose there is a separation $(C,D)$ of $G$ such that $V(C\cap D) \subseteq V(O)$, $V(W) \subseteq V(D)$, and the corners and nails of $W$ are in $C$.
If there is a $V(C\cap D)$-reduction of $D$ that can be embedded on a closed disk $\Delta$ so that $V(C\cap D)$ lies on the boundary of $\Delta$ and the order of $V(C\cap D)$ along the boundary of $\Delta$ agrees with the order along $O$, then we say that the wall $W$ is \emph{flat in $G$} and that the separation $(C,D)$ \emph{certifies} that $W$ is flat.
Note that a subwall of a flat wall is also flat.

We can now state the flat wall theorem \cite{RobSeyXIII, Chu}.

\begin{theorem}
[Theorem 2.2 in \cite{Chu}]
\label{flatwalltheorem}
Let $r,t \geq 1$ be integers. 
Then there exists a function $F(r,t)$ such that if a graph $G$ contains an $F(r,t)$-wall $W$, then one of the following outcomes hold:
\begin{enumerate}
	\item
	$G$ contains a $K_t$-model $\mu$ such that $\mct_\mu$ is a truncation of $\mct_W$.
	\item
	There exists $Z \subseteq V(G)$ with $|Z|\leq t-5$ and an $r$-subwall $W'$ of $W$ that is disjoint from $Z$ and flat in $G-Z$. 
\end{enumerate}
\end{theorem}

\subsection{Linkages}
Let $G$ be a graph and let $A \subseteq V(G)$.
A \emph{linkage} is a set of disjoint paths. 
The union of the paths in a linkage $\mcp$ is denoted $\cup\mcp$.
An \emph{$A$-linkage} is a linkage of $A$-paths.
Let $<$ be a linear order on $A$ and let $P$ and $Q$ be disjoint $A$-paths with endpoints $(p_1,p_2)$ and $(q_1,q_2)$ respectively such that $p_1<p_2$ and $q_1<q_2$.
We say that $P$ and $Q$ are \emph{in series} if $p_2 < q_1$ or $q_2 < p_1$, \emph{nested} if $p_1<q_1<q_2<p_2$ or $q_1<p_1<p_2<q_2$, and \emph{crossing} if they are neither in series or nested.
An $A$-linkage $\mcp$ is \emph{in series, nested}, or \emph{crossing} if every pair of paths in $\mcp$ is in series, nested, or crossing, respectively.
An $A$-linkage is \emph{pure} if it is in series, nested, or crossing. 
\begin{lemma}[Lemma 25 in \cite{HuyJooWol}]
\label{purelinkagelemma}
Let $A$ be a linearly ordered set and let $t \in \mbn$. If $\mcp$ is an $A$-linkage with $|\mcp| \geq t^3$, then $\mcp$ contains a pure linkage $\mcp'$ with $|\mcp'|\geq t$.
\end{lemma}
Let $(W,\gamma)$ be a $\Gamma$-labelled wall in $(G,\gamma)$ with the set $N$ of top nails.
A \emph{linkage of $(W,\gamma)$} is an $N$-linkage in $G-(W-N)$.
Note that this definition depends on the choice of the top nails of $W$, which may be implicit.
A linkage of $(W,\gamma)$ is \emph{pure} if it is pure with respect to a linear ordering of $N$ given by the top horizontal path of $(W,\gamma)$.
A linkage $\mcp$ of $(W,\gamma)$ is \emph{$\Gamma$-odd} if $(W\cup P,\gamma)$ contains a $\Gamma$-nonzero cycle containing $P$ for all $P\in\mcp$.
If $(W,\gamma)$ is strongly $\Gamma$-bipartite and $\mcp$ is a $\Gamma$-odd linkage of $(W,\gamma)$, then $(G,\gamma)$ is shift-equivalent to a $\Gamma$-labelled graph $(G,\gamma')$ where every $b(W)$-path $Q$ in $W$ satisfies $\gamma'(Q)=0$ and every path $P$ in $\mcp$ satisfies $\gamma'(P)\neq 0$.

\section{Main theorem} \label{sec:maintheorem}
We are now ready to state our main structure theorem.
Roughly, it says that given a large order tangle in a $\Gamma$-labelled graph $(G,\gamma)$, we can either find many $\Gamma$-nonzero cycles distributed in one of few specified configurations or delete a bounded size vertex set to destroy all $\Gamma$-nonzero cycles in the ``large'' part of the tangle.
This is particularly useful for proving Theorem \ref{nzcyclesEPtheorem} since a minimal counterexample to the family of $\Gamma$-nonzero cycles satisfying the Erd\H{o}s-P\'osa property admits a large tangle $\mct$ such that no $\Gamma$-nonzero cycle is contained in the small side of a separation in $\mct$, as we saw in Lemma \ref{tangleminctex}.

\begin{restatable}{theorem}{flatwallundirectedtheoremres}
\label{flatwallundirectedtheorem}
Let $\Gamma$ be an abelian group and let $r,t\geq 1$ be integers. Then there exist integers $g(r,t)$ and $h(r,t)$, where $h(r,t)\leq g(r,t)-3$, such that if a $\Gamma$-labelled graph $(G,\gamma)$ contains a wall $(W,\gamma)$ of size at least $g(r,t)$, then one of the following outcomes hold:
\begin{enumerate}
	\item[(1)]
	There is a $\Gamma$-odd $K_t$-model $\mu$ in $G$ such that $\mct_\mu$ is a truncation of $\mct_W$.
	\item[(2)]
	There exists $Z \subseteq V(G)$ with $|Z|\leq h(r,t)$ and a flat $50r^{12}$-wall $(W_0,\gamma)$ in $(G-Z,\gamma)$ such that $\mct_{W_0}$ is a truncation of $\mct_W$ and either
	\begin{enumerate}
		\item
		$(W_0,\gamma)$ is facially $\Gamma$-odd, or
		\item
		$(W_0,\gamma)$ is strongly $\Gamma$-bipartite and there is a pure $\Gamma$-odd linkage of $(W_0,\gamma)$ of size $r$.
	\end{enumerate}
	\item[(3)]
	There exists $Z \subseteq V(G)$ with $|Z|\leq h(r,t)$ such that the $\mct_W$-large 3-block of $(G-Z,\gamma)$ is $\Gamma$-bipartite.
\end{enumerate}
\end{restatable}

The proof of Theorem \ref{flatwallundirectedtheorem} proceeds first by applying the flat wall theorem (Theorem \ref{flatwalltheorem}) to obtain one of its two outcomes.
If there is a large $K_m$-model $\pi$ in $G$ such that $\mct_\pi$ is a truncation of $\mct_W$, then by Ramsey's theorem for 4-uniform hypergraphs, we obtain a large submodel $\mu$ of $\pi$ that is either $\Gamma$-odd or $\Gamma$-bipartite.
The first case satisfies outcome (1) of Theorem \ref{flatwallundirectedtheorem}.
In the second case, we show that the $\Gamma$-labelled subgraph $(\mu,\gamma)$ is $\Gamma$-bipartite (see also Remark \ref{rem:Ktbip}).
We then look to enlarge $\mu$ to a $\Gamma$-odd $K_t$-model by choosing a vertex $s_i$ in each tree $\mu(v_i)$ and finding many disjoint $\Gamma$-nonzero $S$-paths, where $S=\{s_1,\dots,s_t\}$.
With an appropriate choice of such paths, we obtain a $\Gamma$-odd enlargement of $\mu$ whose trees contain the union of a pair of trees $\mu(v_i)$ and the $\Gamma$-nonzero $S$-path connecting them (see Figure \ref{fig:oddmodel} in section \ref{sec:clique}).

Finding the appropriate $S$-paths, however, seems to be considerably more difficult in undirected group-labelled graphs compared to the directed setting.
The main obstacle is that Lemma \ref{lemma3connshiftequivalent} requires 3-connectivity for a $\Gamma$-bipartite graph to be shift-equivalent to the labelling $\bm 0$, whereas the directed analog of Lemma \ref{lemma3connshiftequivalent} (which is essentially equivalent to Lemma \ref{lemmaordertwoshiftequivalent}) does not require any connectivity assumptions.
This means that, in the $\Gamma$-bipartite graph $(\mu,\gamma)$, we can only guarantee that paths between \emph{branching} vertices of $\mu$ (defined in section \ref{sec:clique}) are $\Gamma$-zero.
This requires us to choose the vertices $s_i$ more carefully and keep track of how each $s_i$ branches to the other trees of $\mu$ throughout the proof.

In the second outcome of the flat wall theorem, we also employ a Ramsey-type argument to the given wall $(W,\gamma)$ to obtain a smaller wall $(W_0,\gamma)$ that is either facially $\Gamma$-odd or $\Gamma$-bipartite.
The first case satisfies outcome (2)-(a) of Theorem \ref{flatwallundirectedtheorem}.
In the second case, we find many disjoint $\Gamma$-nonzero $N_0$-paths outside of the wall $W_0$, where $N_0$ is the set of top nails of $W_0$, and apply Lemma \ref{purelinkagelemma} to obtain a pure $\Gamma$-odd linkage of $(W_0,\gamma)$.
In this part, the 3-connectivity condition adds only minor obstacles.

In both outcomes, if the desired $\Gamma$-nonzero $S$-paths or $N_0$-paths do not exist, then we show that outcome (3) of Theorem \ref{flatwallundirectedtheorem} is satisfied using Theorem \ref{nzapathslemma}.

\subsection{Proof of Theorem \ref{nzcyclesEPtheorem}}

Theorem \ref{nzcyclesEPtheorem} follows readily from Theorem \ref{flatwallundirectedtheorem} and the tools presented in Section \ref{sec:prelim}.
The proofs of the two statements in Theorem \ref{nzcyclesEPtheorem} are almost identical; they diverge only in the outcome (2)-(b) of Theorem \ref{flatwallundirectedtheorem} with a crossing $\Gamma$-odd linkage of $(W_0,\gamma)$, in which case the existence of disjoint $\Gamma$-nonzero cycles is only guaranteed if $\Gamma$ has no element of order two.
The two proofs will be distinguished only in this case and will otherwise proceed simultaneously.

\label{sec:nzcycles}

\nzcyclesEPtheoremres*

\begin{proof}
	For each positive integer $k$ define $r=r(k)=3k$ and $t=t(k)=4k$.
	Let $\omega$ be the function given by Theorem \ref{tanglewallthm} and let $g$ and $h$ be the functions given by Theorem \ref{flatwallundirectedtheorem}. 
	Let $f:\mbn \to \mbn$ be a function such that $f(k) \geq \omega(g(r,t)) + 2f(k-1)$, $f(k) \geq 3\omega(g(r,t))$, and $f(k) \geq h(r,t)$ for all $k \geq 2$.
	We claim that $f$ is a (half-integral) Erd\H{o}s-P\'osa function for the family of $\Gamma$-nonzero cycles.
	
	For the sake of contradiction, suppose $((G,\gamma),k)$ is a minimal counterexample to $f$ being a (half-integral) Erd\H{o}s-P\'osa function for the family of $\Gamma$-nonzero cycles.
	Then by Lemma \ref{tangleminctex}, $(G,\gamma)$ admits a tangle $\mct$ of order $\omega(g(r,t))$ such that no $\Gamma$-nonzero cycle is contained in the small side of a separation in $\mct$.
	By Theorem \ref{tanglewallthm}, $(G,\gamma)$ contains a $g(r,t)$-wall $(W,\gamma)$ such that $\mct_W$ is a truncation of $\mct$.
	We then apply Theorem \ref{flatwallundirectedtheorem} to $r, t$, and $(W,\gamma)$ to obtain one of its outcomes (1)-(3).
	
	First suppose outcome (1) holds, that there exists a $\Gamma$-odd $K_{4k}$-model $\mu$ in $(G,\gamma)$.
	Take $k$ disjoint $K_4$-submodels $\mu_1,\dots,\mu_k$ of $\mu$. 
	Since $\mu$ is $\Gamma$-odd, we obtain a $\Gamma$-nonzero cycle in each $\mu_i$, hence a packing of $k$ $\Gamma$-nonzero cycles, a contradiction. 
	Similarly, in outcome (2)-(a), a facially $\Gamma$-odd $50r^{12}$-wall contains $k$ disjoint bricks, each being a $\Gamma$-nonzero cycle.
	
	Next suppose outcome (2)-(b) holds, that there exists a strongly $\Gamma$-bipartite wall $(W_0,\gamma)$ and a pure $\Gamma$-odd linkage $\mcl=\{L_1,\dots,L_{3k}\}$ of $(W_0,\gamma)$.
  By the definition of strongly $\Gamma$-bipartite walls and $\Gamma$-odd linkages, we may assume, after possibly applying a sequence of shifting operations in $(G,\gamma)$, that 
	\begin{equation} \label{eqn:oddlinkage}
		\text{every $b(W_0)$-path in $(W_0,\gamma)$ is $\Gamma$-zero and every path in $\mcl$ is $\Gamma$-nonzero.} \tag{$\star$}
	\end{equation}
	Let $R$ denote the first horizontal path of $W_0$.
	
	If $\mcl$ is in series, then clearly $R \cup (\cup\mcl)$ contains $k$ disjoint $\Gamma$-nonzero cycles, each of which consists of a path $L$ in $\mcl$ (which is $\Gamma$-nonzero) and the subpath of $R$ connecting two endpoints of $L$ (which is $\Gamma$-zero by \eqref{eqn:oddlinkage}).
	If $\mcl$ is nested, then we can connect the endpoints of each path $L_i$ in $\mcl$ by using the two vertical paths of $W_0$ that contain the two endpoints of $L_i$ and the $i$-th horizontal path, where $L_i$ is the $i$-th path in $\mcl$ from the center of the nest.
	This again gives $k$ disjoint $\Gamma$-nonzero cycles by \eqref{eqn:oddlinkage}.
	
	Now suppose that $\mcl$ is crossing.
	Let $x_i$ and $y_i$ denote the left and right endpoint of $L_i$ respectively.
	We may assume that $x_i$ is to the left of $x_j$ for all $i<j$.
	
	First consider the case that $\Gamma$ has no element of order two.
	Then we can find $k$ disjoint $\Gamma$-nonzero cycles in $R \cup (\cup\mcl)$ as follows.
	For $i \in [k]$, define the pairwise disjoint subgraphs
	$$H_i := x_{3i-2}Rx_{3i} \cup L_{3i-2} \cup L_{3i-1} \cup L_{3i} \cup y_{3i-2}Ry_{3i}.$$
	Then each $H_i$ contains a $\Gamma$-nonzero cycle; otherwise, since $b(W_0)$-paths in $W_0$ are $\Gamma$-zero, we have $\gamma(L_{3i-2}) = -\gamma(L_{3i-1})$, $\gamma(L_{3i-1})=-\gamma(L_{3i})$, and $\gamma(L_{3i-2})=-\gamma(L_{3i})$ which implies $2\gamma(L_{3i})=0$, a contradiction.
	
	For a general abelian group $\Gamma$, we settle for a half-integral packing of $\Gamma$-nonzero cycles.
	There exist $2k$ paths $Q_1,\dots,Q_{2k}$ where $Q_i$ is the unique $x_i$-$y_i$-path in the union of the two vertical paths of $W_0$ containing $x_i$ or $y_i$ and the $(i+1)$-th horizontal path of $W_0$.
	By \eqref{eqn:oddlinkage}, $Q_i$ is $\Gamma$-zero and $L_i$ is $\Gamma$-nonzero for all $i\in[2k]$. 
	Hence, $\{Q_i\cup L_i: i\in[2k]\}$ is a half-integral packing of $\Gamma$-nonzero cycles of size $k$.
	
	So we may assume that outcome (3) of Theorem \ref{flatwallundirectedtheorem} holds; that is, there exists $Z \subseteq V(G)$ with $|Z| \leq h(r,t)$ such that the $\mct_W$-large 3-block $(B,\gamma_B)$ of $(G-Z,\gamma)$ is $\Gamma$-bipartite.
	Recall that $B$ is 3-connected. By Lemma \ref{lemma3connshiftequivalent}, there is a sequence of shifting operations on $(B,\gamma_B)$ resulting in $(B,\bm 0)$.
	This implies that $(G-Z,\gamma)$ is shift-equivalent (by the same shifting operations) to $(G-Z,\gamma')$ for some $\gamma'$ where every $V(B)$-path in $(G-Z,\gamma')$ is $\Gamma$-zero.
	
	We now show that $Z$ is a hitting set for $\Gamma$-nonzero cycles in $(G,\gamma)$.
	If a cycle $C$ in $G-Z$ has at least two vertices in $V(B)$, then $E(C)$ can be partitioned into $V(B)$-paths.
	But since every $V(B)$-path in $(G-Z,\gamma')$ is $\Gamma$-zero, $C$ is $\Gamma$-zero in $(G-Z,\gamma')$, whence $C$ is $\Gamma$-zero in $(G-Z,\gamma)$ as well (since shifting operations do not change the weights of cycles).
	So if $C$ is a $\Gamma$-nonzero cycle in $(G-Z,\gamma)$, then $C$ has at most one vertex in $V(B)$.
  This implies that $C$ is contained in a $V(B)$-bridge of $G-Z$ (since $C-V(B)$ is in a connected component of $G-Z-V(B)$).
	Let $A \subseteq V(B)$ denote the set of (at most two) attachments of this $V(B)$-bridge.
	This gives a $(h(r,t)+2)$-separation $(X,Y)$ of $G$ where $V(X\cap Y) = Z\cup A$, $C\subseteq X$, and $V(B) \subseteq V(Y)$.
	Moreover, we have $(X,Y) \in \mct_W$ since $(B,\gamma_B)$ is the $\mct_W$-large 3-block of $(G-Z,\gamma)$.
	But this contradicts the assumption that no $\Gamma$-nonzero cycle is contained in the small side of a separation in $\mct_W$.
	Therefore, $Z$ is a hitting set for $\Gamma$-nonzero cycles in $(G,\gamma)$, again a contradiction as $|Z|\leq h(r,t) \le f(k)$. 
\end{proof}

What we have essentially shown in the proof of Theorem \ref{nzcyclesEPtheorem} is that, if a $\Gamma$-labelled graph does not contain a large packing of $\Gamma$-nonzero cycles nor a small hitting set for $\Gamma$-nonzero cycles, then it contains a large strongly $\Gamma$-bipartite wall with a large crossing $\Gamma$-odd linkage.
This last condition, when specialized to $\Gamma = \mbz/2\mbz$, is equivalent to a large \textit{Escher wall} (see \cite{Ree99}); we have thus recovered Reed's theorem (Theorem 1 in \cite{Ree99}) that every graph contains either a large packing of odd cycles, a small hitting set for odd cycles, or a large Escher wall.

\section{Cycles of length $\ell$ mod $m$}
\label{sec:oddprimecycles}
In this section we prove Theorem \ref{cyclesoddprimepowertheorem}, that the family of cycles of length $\ell$ mod $m$ satisfies the Erd\H{o}s-P\'osa property if $m$ is an odd prime power.

A \emph{cycle-chain} of \emph{length} $l$ is a tuple $(P,Q_1,\dots,Q_l)$ consisting of a \emph{core path} $P$ and $l$ disjoint $V(P)$-paths $Q_1,\dots, Q_l$ such that the subpaths $P_i$ of $P$ having the same endpoints as $Q_i$ are pairwise disjoint.
A cycle-chain $(P,Q_1,\dots,Q_l)$ in a $\Gamma$-labelled graph $(G,\gamma)$ is \emph{$\Gamma$-nonzero} if $\gamma(Q_i)\neq \gamma(P_i)$ for all $i \in [l]$.
A \emph{closed} cycle-chain $(C,Q_1,\dots,Q_l)$ of length $l$ consists of a cycle $C$ and $l$ disjoint $V(C)$-paths $Q_1,\dots,Q_l$ such that there is a choice of disjoint subpaths $C_i$ of $C$ with the same endpoints as $Q_i$.
It is $\Gamma$-nonzero if $\gamma(Q_i) \neq \gamma(C_i)$ for all $i \in [l]$.

Long $\Gamma$-nonzero closed cycle-chains can be used to find cycles of certain desired weights.
For example, if $\Gamma=\mbz/p\mbz$ for prime $p$, then it is easy to see that given a long enough $\Gamma$-nonzero closed cycle-chain $(C,Q_1,\dots,Q_l)$, we can reroute $C$ through some of the paths $Q_i$ to obtain a cycle of every possible weight in $\Gamma$ (since every nonzero element of $\Gamma$ is a generator).
The next few lemmas show that, if $\Gamma$ has no element of order two, then we can find many disjoint long $\Gamma$-nonzero closed cycle-chains in outcomes (1) and (2) of Theorem \ref{flatwallundirectedtheorem}.

\begin{lemma}
\label{longnzcyclechaincliquelemma}
Let $\Gamma$ be an abelian group with no element of order two.
Let $l \in \mbn$ and let $\mu$ be a $\Gamma$-odd $K_{5l+1}$-model.
Then there is a $\Gamma$-nonzero cycle-chain $\mcc$ of length $l$ contained in $\mu$ whose core path is a $\mu(v_1)$-$\mu(v_{5l+1})$-path.
Moreover, the cycles in $\mcc$ are disjoint from $\mu(v_1)\cup \mu(v_{5l+1})$.
\end{lemma}
\begin{proof}
We first prove the lemma for $l=1$.
Since $\mu$ is $\Gamma$-odd, there is a $\Gamma$-nonzero cycle $C$ in $\mu[v_2,v_3,v_4,v_5]$, and since $C$ intersects at least three of the trees in $\{\mu(v_2),\mu(v_3),\mu(v_4), \mu(v_5)\}$, we may assume without loss of generality that $C$ intersects $\mu(v_2),\mu(v_3)$, and $\mu(v_4)$.
For each $i \in \{2,3,4\}$, let $P_i$ denote the unique $\mu(v_6)$-$C$-path in $\mu[v_i,v_6]$ and let $w_i$ denote the endpoint of $P_i$ in $C$.
Since $C$ is a $\Gamma$-nonzero cycle, without loss of generality, we may assume that the $w_3$-$w_4$-path in $C$ that is disjoint from $w_2$ is $\Gamma$-nonzero.
Let $R$ denote the unique $\mu(v_1)$-$C$-path in $\mu[v_1,v_2]$.
Then, by Lemma \ref{threepathscyclelemma}(a) and by the assumption that $\Gamma$ has no element of order two, there is a $j \in \{3,4\}$ such that the two $\mu(v_1)$-$\mu(v_6)$-paths in $R \cup C \cup P_j$ have different weights.
This gives the desired $\Gamma$-nonzero cycle-chain of length 1 whose core path is a $\mu(v_1)$-$\mu(v_6)$-path and whose cycle $C$ is disjoint from $\mu(v_1)\cup\mu(v_6)$ (as $C$ is contained in $\mu[v_2,v_3,v_4,v_5]$).

Now for $l > 1$, we apply the $l=1$ case to the $K_6$-submodel 
\[\mu_i := \mu[v_{5(i-1)+1},v_{5(i-1)+2},\dots,v_{5(i-1)+6}]\]
for each $i\in[l]$ to obtain a $\Gamma$-nonzero cycle-chain $\mcc_i$ of length 1 contained in $\mu_i$ whose core path is a $\mu(v_{5(i-1)+1})$-$\mu(v_{5i+1})$-path and whose cycle is disjoint from $\mu(v_{5(i-1)+1})\cup \mu(v_{5i+1})$. 
Note that the core path is internally disjoint from $\mu(v_{5(i-1)+1}) \cup \mu(v_{5i+1})$.
Hence, for each $i\in[l-1]$, the unique path in $\mu(v_{5i+1})$ between the endpoints of the core paths of $\mcc_i$ and $\mcc_{i+1}$ in $\mu(v_{5i+1})$ is internally disjoint from every $\mcc_j$, $j\in[l]$.
By connecting consecutive cycle-chains in the trees $\mu(v_{5i+1})$, $i \in [l-1]$, we obtain a $\Gamma$-nonzero cycle-chain of length $l$  contained in $\mu$ whose core path is a $\mu(v_1)$-$\mu(v_{5l+1})$-path and whose cycles are disjoint from $\mu(v_1)\cup \mu(v_{5l+1})$.
\end{proof}

\begin{cor}\label{closedcccliquecor}
Let $\Gamma$ be an abelian group with no element of order two.
Let $l,k\in \mbn$ and let $\mu$ be a $\Gamma$-odd $K_{(5l+1)k}$-model.
Then $\mu$ contains $k$ disjoint $\Gamma$-nonzero closed cycle-chains each of length $l$.
\end{cor}
\begin{proof}
Let $\mu_1,\dots,\mu_k$ be $k$ disjoint $K_{5l+1}$-submodels of $\mu$ and, for each $i\in[k]$, let $\mu_i(x^i)$ and $\mu_i(y^i)$ denote two trees of $\mu_i$. 
By Lemma \ref{longnzcyclechaincliquelemma}, each $\mu_i$ contains a $\Gamma$-nonzero cycle-chain $\mcc_i$ of length $l$ whose core path is a $\mu_i(x^i)$-$\mu_i(y^i)$-path and whose cycles are disjoint from $\mu_i(x^i) \cup \mu_i(y^i)$.
By connecting the two ends of the core path through $\mu_i[x^i,y^i]$, we obtain $k$ disjoint $\Gamma$-nonzero closed cycle-chains, each of length $l$.
\end{proof}

\begin{lemma}
\label{longnzcyclechainwalllemma}
Let $\Gamma$ be an abelian group with no element of order two.
Let $(W,\gamma)$ be a facially $\Gamma$-odd $3l \times 2$-wall and let $P^{(h)}_i$ denote the $i$-th horizontal path of $W$ for each $i \in [3l+1]$.
Then there is a $\Gamma$-nonzero cycle-chain $\mcc$ of length $l$ in $(W,\gamma)$ whose core path is a $P^{(h)}_1$-$P^{(h)}_{3l+1}$-path.
Moreover, the cycles in $\mcc$ are disjoint from $P^{(h)}_1\cup P^{(h)}_{3l+1}$.
\end{lemma}
\begin{proof}
We first prove the lemma for $l=1$.
Let $(W,\gamma)$ be a facially $\Gamma$-odd $3\times 2$-wall.
Let $B_{i,j}$ denote the $(i,j)$-th brick of $W$ for $i\in[3]$ and $j\in[2]$.
We assume without loss of generality that the wall is oriented in such a way that $B_{2,1}$ shares an edge with $B_{1,1}$ and $B_{1,2}$.
Let $w_1,w_2$, and $w_3$ denote the three vertices on $B_{2,1}$ that have degree 3 in $B_{1,2}\cup B_{2,1}\cup B_{2,2}$.
Since $W$ is facially $\Gamma$-odd, by Lemma \ref{threepathscyclelemma}(b), there is a distinct pair $i,j \in [3]$ such that the two $w_i$-$w_j$-paths in $B_{2,1}$ have different weights.
In each of the three possible cases, it is easy to see that there is a $\Gamma$-nonzero cycle-chain of length 1 whose core path is a $P^{(h)}_1$-$P^{(h)}_4$-path and whose cycle is $B_{2,1}$ (see Figure \ref{Fig1chainwall}).

Now for $l > 1$, let $W_i$ be the $3\times 2$-subwall of $W$ whose first and last horizontal path is $P^{(h)}_{3i-2}$ and $P^{(h)}_{3i+1}$ respectively.
We apply the $l=1$ case of the lemma to each $W_i$ to obtain a $\Gamma$-nonzero cycle-chain of length 1 whose core path is a $P^{(h)}_{3i-2}$-$P^{(h)}_{3i+1}$-path and whose cycle is disjoint from $P^{(h)}_{3i-2} \cup P^{(h)}_{3i+1}$.
Connecting consecutive cycle-chains in $P^{(h)}_{3i+1}$, $i \in [l-1]$, we obtain the desired $\Gamma$-nonzero cycle-chain of length $l$.
\end{proof}

\begin{cor}\label{closedccwallcor}
Let $\Gamma$ be an abelian group with no element of order two.
Let $(W,\gamma)$ be a facially $\Gamma$-odd $3l\times(4k-1)$-wall.
Then $(W,\gamma)$ contains $k$ disjoint $\Gamma$-nonzero closed cycle-chains of length $l$.
\end{cor}
\begin{proof}
We have $k$ disjoint compact $3l\times 3$-subwalls $(W_1,\gamma),\dots,(W_k,\gamma)$ of $(W,\gamma)$, each of which are facially $\Gamma$-odd.
We find a $\Gamma$-nonzero cycle-chain of length $l$ in a compact $3l\times 2$-subwall of $W_i$ and connect its ends using the extra vertical path to obtain a $\Gamma$-nonzero closed cycle-chain of length $l$ in $(W_i,\gamma)$ for each $i\in[k]$.
\end{proof}

\begin{figure}[t]
	\centering
	\resizebox{0.9\textwidth}{!}{\begin{tikzpicture}

\colorlet{hellgrau}{black!30!white}
\tikzstyle{smallvx}=[thick,circle,inner sep=0.cm, minimum size=2mm, fill=white, draw=black]
\tikzstyle{smallblackvx}=[thick,circle,inner sep=0.cm, minimum size=2mm, fill=black, draw=black]
\tikzstyle{squarevx}=[thick,rectangle,inner sep=0.cm, minimum size=3mm, fill=white, draw=black]
\tikzstyle{hedge}=[line width=2pt]
\tikzstyle{markline}=[draw=hellgrau,line width=7pt]

\def\wallheight{3}
\def\wallwidth{2}
\def\gridwidth{2*\wallwidth+1}

\pgfmathtruncatemacro{\penultimaterow}{\wallheight-1}
\pgfmathtruncatemacro{\lastrowshift}{mod(\wallheight+1,2)}

\draw[] (0,0) -- (2*\wallwidth,0);
\draw[] (\lastrowshift,\wallheight) -- (\lastrowshift+2*\wallwidth,\wallheight);
\foreach \i in {1,...,\penultimaterow} {
	\draw[] (0,\i) -- (\gridwidth,\i);
}

\foreach \i in {0,...,\wallwidth}{
	\foreach \j in {1,...,\wallheight} {
		\draw[] ({2*\i+Mod(\j+1,2)},\j-1) -- ({2*\i+Mod(\j+1,2)},\j);
	}
}

\draw[markline] (2,2) -- (3,2);

\node[smallblackvx] at (2,2){};
\node[smallblackvx] at (3,2){};
\node[smallblackvx] at (3,1){};

\draw[hedge] (1,1) -- (3,1);
\draw[hedge] (3,2) -- (3,1);
\draw[hedge] (3,2) -- (1,2);
\draw[hedge] (1,1) -- (1,2);

\node at (2,1.5){$B_{2,1}'$};
\node at (4,1.5){$B_{2,2}'$};
\node at (3,2.5){$B_{1,2}'$};

\draw[hedge] (4,0) -- (4,1);
\draw[hedge] (4,1) -- (5,1); 
\draw[hedge] (5,1) -- (5,2);
\draw[hedge] (3,2) -- (5,2);
\draw[hedge] (2,2) -- (2,3);
\node[squarevx] at (4,0){};
\node[squarevx] at (2,3){};

\end{tikzpicture}
\qquad\qquad
\begin{tikzpicture}

\colorlet{hellgrau}{black!30!white}
\tikzstyle{smallvx}=[thick,circle,inner sep=0.cm, minimum size=2mm, fill=white, draw=black]
\tikzstyle{smallblackvx}=[thick,circle,inner sep=0.cm, minimum size=2mm, fill=black, draw=black]
\tikzstyle{squarevx}=[thick,rectangle,inner sep=0.cm, minimum size=3mm, fill=white, draw=black]
\tikzstyle{hedge}=[line width=2pt]
\tikzstyle{markline}=[draw=hellgrau,line width=7pt]

\def\wallheight{3}
\def\wallwidth{2}
\def\gridwidth{2*\wallwidth+1}

\pgfmathtruncatemacro{\penultimaterow}{\wallheight-1}
\pgfmathtruncatemacro{\lastrowshift}{mod(\wallheight+1,2)}

\draw[] (0,0) -- (2*\wallwidth,0);
\draw[] (\lastrowshift,\wallheight) -- (\lastrowshift+2*\wallwidth,\wallheight);
\foreach \i in {1,...,\penultimaterow} {
	\draw[] (0,\i) -- (\gridwidth,\i);
}

\foreach \i in {0,...,\wallwidth}{
	\foreach \j in {1,...,\wallheight} {
		\draw[] ({2*\i+Mod(\j+1,2)},\j-1) -- ({2*\i+Mod(\j+1,2)},\j);
	}
}

\draw[markline] (3,1) -- (3,2);

\node[smallblackvx] at (2,2){};
\node[smallblackvx] at (3,2){};
\node[smallblackvx] at (3,1){};

\draw[hedge] (1,1) -- (3,1);
\draw[hedge] (3,2) -- (3,1);
\draw[hedge] (3,2) -- (1,2);
\draw[hedge] (1,1) -- (1,2);

\draw[hedge] (4,0) -- (4,1);
\draw[hedge] (4,1) -- (3,1); 
\draw[hedge] (4,2) -- (4,3);
\draw[hedge] (3,2) -- (4,2);
\node[squarevx] at (4,0){};
\node[squarevx] at (4,3){};

\end{tikzpicture}
\qquad\qquad
\begin{tikzpicture}

\colorlet{hellgrau}{black!30!white}
\tikzstyle{smallvx}=[thick,circle,inner sep=0.cm, minimum size=2mm, fill=white, draw=black]
\tikzstyle{smallblackvx}=[thick,circle,inner sep=0.cm, minimum size=2mm, fill=black, draw=black]
\tikzstyle{squarevx}=[thick,rectangle,inner sep=0.cm, minimum size=3mm, fill=white, draw=black]
\tikzstyle{hedge}=[line width=2pt]
\tikzstyle{markline}=[draw=hellgrau,line width=7pt]

\def\wallheight{3}
\def\wallwidth{2}
\def\gridwidth{2*\wallwidth+1}

\pgfmathtruncatemacro{\penultimaterow}{\wallheight-1}
\pgfmathtruncatemacro{\lastrowshift}{mod(\wallheight+1,2)}

\draw[] (0,0) -- (2*\wallwidth,0);
\draw[] (\lastrowshift,\wallheight) -- (\lastrowshift+2*\wallwidth,\wallheight);
\foreach \i in {1,...,\penultimaterow} {
	\draw[] (0,\i) -- (\gridwidth,\i);
}

\foreach \i in {0,...,\wallwidth}{
	\foreach \j in {1,...,\wallheight} {
		\draw[] ({2*\i+Mod(\j+1,2)},\j-1) -- ({2*\i+Mod(\j+1,2)},\j);
	}
}

\draw[markline] (1,2) -- (2,2);
\draw[markline] (1,2) -- (1,1);
\draw[markline] (1,1) -- (3,1);

\node[smallblackvx] at (2,2){};
\node[smallblackvx] at (3,2){};
\node[smallblackvx] at (3,1){};

\draw[hedge] (1,1) -- (3,1);
\draw[hedge] (3,2) -- (3,1);
\draw[hedge] (3,2) -- (1,2);
\draw[hedge] (1,1) -- (1,2);

\draw[hedge] (4,0) -- (4,1);
\draw[hedge] (4,1) -- (3,1); 
\draw[hedge] (2,2) -- (2,3);
\node[squarevx] at (4,0){};
\node[squarevx] at (2,3){};

\end{tikzpicture}}
	\caption{The three black vertices are $w_1,w_2$, and $w_3$. Since $B_{2,1}$ is a $\Gamma$-nonzero cycle, at least one of these three cycle-chains is $\Gamma$-nonzero.}
	\label{Fig1chainwall}
\end{figure}

\begin{lemma}\label{longnzcyclechainlinkagelemma}
Let $\Gamma$ be an abelian group with no element of order two.
Let $(W,\gamma)$ be a strongly $\Gamma$-bipartite wall and let $\mcl$ be a pure $\Gamma$-odd linkage of $(W,\gamma)$ with $|\mcl|=3l$.
Then there is a $\Gamma$-nonzero cycle-chain of length $l$ contained in $P^{(h)}_1 \cup (\cup \mcl)$ whose core path is a subpath of $P^{(h)}_1$.
Moreover, if $\mcl$ is nested or crossing, then the core path intersects exactly one endpoint of each path in $\mcl$.
\end{lemma}
\begin{proof}
Since $\Gamma$ has no element of order two and $(W,\gamma)$ is strongly $\Gamma$-bipartite, every $b(W)$-path in $(W,\gamma)$ is $\Gamma$-zero and every path in $\mcl$ is $\Gamma$-nonzero.

If $\mcl$ is in series, then the conclusion is trivial as $P^{(h)}_1\cup (\cup\mcl)$ itself is a $\Gamma$-nonzero cycle-chain with core path $P^{(h)}_1$.
So we assume that $\mcl$ is nested or crossing.
Let $L_1,\dots,L_{3l}$ denote the paths of $\mcl$ and let $x_i$ and $y_i$ denote the left and right endpoint of $L_i$, $i \in [3l]$.
Then there exist disjoint subpaths $R_x$ and $R_y$ of $P^{(h)}_1$ such that $\{x_1,\dots,x_{3l}\} \subseteq V(R_x)$ and $\{y_1,\dots,y_{3l}\} \subseteq V(R_y)$.
We may assume that $x_i$ is positioned to the left of $x_j$ for $i<j$.

First consider the case $l=1$.
We claim that there exist $i,j \in [3]$, $i<j$, such that $\gamma(L_i)\neq -\gamma(L_j)$.
Indeed, otherwise we have $\gamma(L_1)=-\gamma(L_2)=\gamma(L_3)=-\gamma(L_1)$, which gives $2\gamma(L_1)=0$, a contradiction as $\Gamma$ has no element of order two.
Now choose such $1\leq i<j \leq 3$ with $\gamma(L_i)\neq -\gamma(L_j)$.
Then $(x_iR_xx_j,x_iL_iy_iR_yy_jL_jx_j)$ is a $\Gamma$-nonzero cycle-chain of length 1 whose core path $x_iR_xx_j$ is a subpath of $P^{(h)}_1$ intersecting exactly one endpoint of each path in $\mcl$.

Now for $l>1$, we apply the $l=1$ case above to obtain a $\Gamma$-nonzero cycle-chain of length 1 contained in 
$$x_{3i-2}R_xx_{3i} \cup L_{3i-2}\cup L_{3i-1} \cup L_{3i} \cup y_{3i-2}R_yy_{3i}$$
for each $i \in [l]$.
Connecting consecutive cycle-chains along $R_x$, we obtain the desired $\Gamma$-nonzero cycle-chain of length $l$.
\end{proof}

\begin{cor}\label{closedcclinkagecor}
Let $\Gamma$ be an abelian group with no element of order two.
Let $(W,\gamma)$ be a strongly $\Gamma$-bipartite wall with a pure $\Gamma$-odd linkage $\mcl$ with $|\mcl|=3kl$.
Then $(W\cup (\cup\mcl),\gamma)$ contains $k$ disjoint $\Gamma$-nonzero closed cycle-chains each of length $l$.
\end{cor}
\begin{proof}
The $\Gamma$-nonzero cycle-chain of length $kl$ obtained via Lemma \ref{longnzcyclechainlinkagelemma} contains $k$ disjoint $\Gamma$-nonzero cycle-chains $\mcc_1,\dots,\mcc_k$ each of length $l$, each of whose core path is a subpath of $P^{(h)}_1$.
Using $P^{(h)}_2$ and suitable $P^{(h)}_1$-$P^{(h)}_2$-paths in $W$, we can reroute and extend the core paths of each $\mcc_i$ to obtain $k$ disjoint $\Gamma$-nonzero closed cycle-chains of length $l$.
\end{proof}

We are now ready to prove Theorem \ref{cyclesoddprimepowertheorem}.

\begin{theorem}[Theorem \ref{cyclesoddprimepowertheorem} restated]
Let $p$ be an odd prime and let $a$ be a nonnegative integer.
Then there exists a function $f_{p,a}:\mbn\to\mbn$ such that for every $\ell \in \mbz/p^a\mbz$ and every positive integer $k$, we have that every $\mbz/p^a\mbz$-labelled graph $(G,\gamma)$ contains either $k$ disjoint cycles of weight $\ell$ or a hitting set $Z$ for cycles of weight $\ell$ with $|Z|\le f_{p,a}(k)$.
\end{theorem}
\begin{proof}
If $a=0$, then $\mbz/p^a\mbz$ is the trivial group and $\ell=0$.
Here, the cycles of weight $\ell$ are just the cycles of the underlying graph, and the theorem follows from Erd\H{o}s and P\'osa's original result \cite{ErdPos}.

Let $a\geq 1$ be the smallest integer such that for some $\ell \in \Gamma:=\mbz/p^a\mbz$, the Erd\H{o}s-P\'osa property fails to hold for cycles of weight $\ell$. 
Fix such an element $\ell \in \Gamma$ and let $q = p^{2a-1}(p-1)$.

We define $f_{p,a}(k)$ as follows.
Let $r=3kq$ and $t=(5q+1)k$. 
Let $g$ and $h$ be the functions from Theorem \ref{flatwallundirectedtheorem}.
We may choose $f_{p,a}(k)$ large enough so that, if $((G,\gamma),k)$ is a minimal counterexample to $f_{p,a}$ being an Erd\H{o}s-P\'osa function for the family of cycles of weight $\ell$, then $(G,\gamma)$ contains a $g(r,t)$-wall $(W,\gamma)$ such that no cycle of weight $\ell$ is contained in the small side of a separation in $\mct_W$ (by applying Lemma \ref{tangleminctex} and Theorem \ref{tanglewallthm}).
Subject to this, we also choose $f_{p,a}(k)$ so that $f_{p,a}(k) > f_{p,a-1}(k) + h(r,t)$ for all $k$.

Let $((G,\gamma),k)$ be a minimal counterexample.
Given an integer $n$, we write $\bar n$ to denote its congruence class modulo $p^a$.
Let $\Gamma'$ be the quotient group $\Gamma/\langle \bar p\rangle$, where $\langle \bar p \rangle$ is the subgroup generated by $\bar p$.
Let $\gamma'$ be the $\Gamma'$-labelling of $G$ defined by $\gamma'(e) = \langle \bar p \rangle + \gamma(e)$ for all $e \in E(G)$ (where $\langle \bar p \rangle + \gamma(e)$ denotes the coset $\{\rho + \gamma(e): \rho\in\langle \bar p\rangle\}$).

We apply Theorem \ref{flatwallundirectedtheorem} to the $\Gamma'$-labelled graph $(G,\gamma')$.
In outcomes (1) and (2), we obtain a $\Gamma'$-odd $K_{(5q+1)k}$-model, a facially $\Gamma'$-odd $3q\times 4k$-wall, or a strongly $\Gamma'$-bipartite wall with a pure $\Gamma'$-odd linkage of size $3kq$.
In these cases we apply Corollary \ref{closedcccliquecor}, \ref{closedccwallcor}, and \ref{closedcclinkagecor} respectively to obtain $k$ disjoint $\Gamma'$-nonzero closed cycle-chains each of length $q=p^{2a-1}(p-1)$.

Let $\mcc = (C,Q_1,\dots,Q_q)$ be one such closed cycle-chain and for each $i\in[q]$, let $P_i$ denote the subpath of $C$ with the same endpoints as $Q_i$ such that $P_1,\dots,P_q$ are pairwise disjoint.
Since $\mcc$ is $\Gamma'$-nonzero, we have $\gamma'(Q_i) \neq \gamma'(P_i)$, hence $\gamma(Q_i)-\gamma(P_i) \not\in \langle \bar p\rangle$.
Note that there are exactly $p^{a-1}(p-1)$ elements in $\Gamma \setminus \langle \bar p \rangle$.
Since $\mcc$ has length $p^{2a-1}(p-1)$, it follows that there is a subset $I \subseteq [p^{2a-1}(p-1)]$ with $|I|=p^a$ such that, for some $\alpha \in \Gamma\setminus \langle \bar p \rangle$, we have $\gamma(Q_i)-\gamma(P_i) = \alpha$ for all $i \in I$.
Since $\alpha \not\in \langle \bar p\rangle$, $\alpha$ generates $\Gamma$ (additively), whence the closed cycle-chain $(C, Q_i:i\in I)$ of length $p^a$ contains a cycle of weight $\ell$.
Since there are $k$ disjoint such cycle chains, we thus obtain $k$ disjoint cycles of weight $\ell$ in outcomes (1) and (2) of Theorem \ref{flatwallundirectedtheorem}, a contradiction.

We may therefore assume that outcome (3) holds; that is, there exists $Z \subseteq V(G)$ with $|Z| \leq h(r,t)$ such that the $\mct_W$-large 3-block $(B',\gamma_B')$ of $(G-Z,\gamma')$ is $\Gamma'$-bipartite.
Since $B'$ is 3-connected and $\Gamma'$ has no element of order two, we have $\gamma_B' = \bm 0$ by Lemma \ref{lemma3connshiftequivalent}.
This implies that the weight of every $V(B')$-path in $(G-Z,\gamma)$ is in $\langle \bar p \rangle$.

If $\ell \not\in \langle \bar p \rangle$, then $Z$ is a hitting set for cycles of weight $\ell$ since otherwise a cycle of weight $\ell$ in $(G-Z,\gamma)$ would be contained in a $V(B')$-bridge and hence in the small side of a separation in $\mct_W$.
Since $(G,\gamma)$ does not contain a hitting set of size $h(r,t) < f_{p,a}(k)$, we may assume that $\ell \in \langle \bar p\rangle$.

Let $(B,\gamma_B)$ be the $\mct_W$-large 3-block of $(G-Z,\gamma)$.
Note that $V(B) = V(B')$ and that $(B,\gamma_B)$ is a $\langle \bar p \rangle$-labelled graph.
Since $\langle \bar p \rangle \cong \mbz/p^{a-1}\mbz$, by our choice of $a$, $(B,\gamma_B)$ contains either $k$ disjoint cycles of weight $\ell$ or a hitting set $Y$ with $|Y| \leq f_{p,a-1}(k)$. 
The first case is a contradiction and in the second case $Y \cup Z$ is a hitting set for cycles of weight $\ell$ in $(G,\gamma)$ with $|Y\cup Z|\leq f_{p,a-1}(k)+h(r,t) \le f_{p,a}(k)$, again a contradiction.
\end{proof}

\section{Proof of Theorem \ref{flatwallundirectedtheorem}: large $K_t$-model}
\label{sec:clique}

The remainder of the paper is dedicated to proving Theorem \ref{flatwallundirectedtheorem}.
In this section we deal with the first outcome of the flat wall theorem where $G$ contains a large $K_t$-model.
As discussed in section \ref{sec:maintheorem}, we need some additional definitions and lemmas related to the trees of a $K_t$-model.

Let $T$ be a tree and let $U \subseteq V(T)$.
Then the smallest subtree of $T$ containing all vertices in $U$ is the \emph{$U$-subtree of $T$}.
For example, if $|U|=2$, say $U=\{u,v\}$, then $U$-subtree of $T$ is the path $uTv$.
Suppose $T$ is a tree in a graph $G$ and $F \subseteq E(G)$ is such that each edge in $F$ has exactly one endpoint in $T$ and no two edges of $F$ share an endpoint outside of $T$ (so that $T\cup F$ is a tree).
Then the \emph{$F$-extension of $T$} is the $V(F)$-subtree of $T\cup F$, where $V(F)$ is the set of endpoints of the edges in $F$.
For $n \in \mbn$, an \emph{$n$-star} is a graph isomorphic to a subdivision of $K_{1,n}$. 
If $n \geq 3$, the \emph{center} of an $n$-star is the unique vertex of degree greater than 2 and a \emph{leg}  is a path from its center to a leaf.

Let $\mu$ be a $K_m$-model in a $\Gamma$-labelled graph $(G,\gamma)$ where $V(K_m)=\{v_1,\dots,v_m\}$. 
For distinct $i,j \in [m]$, let us denote the two endpoints of the edge $\mu(v_iv_j)$ in the trees $\mu(v_i)$ and $\mu(v_j)$ by $\mu(v_iv_j)_i$ and $\mu(v_iv_j)_j$ respectively.
Fix $i\in [m]$ and let $d \in \mbn$.
We say that a vertex $s \in V(\mu(v_i))$ is \emph{$d$-central} if no connected component of $\mu(v_i)-s$ contains $\mu(v_iv_j)_i$ for at least $m-1-d$ distinct indices $j \in [m]-\{i\}$.
In other words, if $e \in E(\mu(v_i))$ is incident to $s$, then the connected component of $\mu(v_i)-e$ containing $s$ contains $\mu(v_iv_j)_i$ for more than $d$ indices $j \in [m]-\{i\}$.

It is easy to see that a $d$-central vertex always exists if $d < \frac{m-1}{2}$: 
start from an arbitrary vertex $s$ in $\mu(v_i)$ and, as long as the current vertex is not $d$-central, move towards the (unique) component $T$ of $\mu(v_i)-s$ for which there are at least $m-1-d > \frac{m-1}{2}$ indices $j$ such that $\mu(v_iv_j)_i \in V(T)$. 
Since $d < \frac{m-1}{2}$, this process cannot backtrack and must end eventually.

Let $\mcc_i^d$ denote the set of $d$-central vertices in $\mu(v_i)$.
The vertices of $\mcc_i^d$ form a subtree in $\mu(v_i)$: if $s',s''$ are distinct vertices in $\mcc_i^d$ and $s$ is in the interior of the path $s'\mu(v_i)s''$, then each connected component $T$ of $\mu(v_i)-s$ is contained in a connected component of either $\mu(v_i)-s'$ or $\mu(v_i)-s''$, so there are less than $m-1-d$ indices $j$ such that $\mu(v_iv_j)_i \in V(T)$.  

Let $u \in V(\mu(v_i))$ and let $j_1,j_2,j_3 \in [m]-\{i\}$ be distinct.
If the $\{\mu(v_iv_{j_1}),\mu(v_iv_{j_2}),\mu(v_iv_{j_3})\}$-extension of $\mu(v_i)$ is a 3-star centered at $u$, then we say that $u$ \emph{branches to} $\{\mu(v_{j_1}),\mu(v_{j_2}),\mu(v_{j_3})\}$, or simply to $\{v_{j_1},v_{j_2},v_{j_3}\}$ if there is no ambiguity of the model.
If $Y \subseteq [m]-\{i\}$, we say that $u$ \emph{branches avoiding $Y$} if there exist distinct $j_1,j_2,j_3 \in [m]-\{i\}-Y$ such that $u$ branches to $\{v_{j_1},v_{j_2},v_{j_3}\}$.
We say that a vertex $u \in V(\mu(v_i))$ is \emph{$d$-branching} if $u$ branches avoiding $Y$ for all $Y \subseteq [m]-\{i\}$ with $|Y| \leq d$.
If $u$ is 0-branching, we simply say that $u$ is \emph{branching}.
The set of branching vertices of a $K_m$-model $\mu$ is denoted $b(\mu)$.
The following proposition is immediate.
\begin{prop} \label{prop:modelbranch3conn}
	Let $\mu$ be a $K_m$-model with $m\geq 4$.
	Then the union of all $b(\mu)$-paths in $\mu$ is a subdivision of a 3-connected graph $H$ where $V(H)=b(\mu)$ and the edges of $H$ correspond to the $b(\mu)$-paths in $\mu$. 
\end{prop}
If $m\geq 4$, then clearly each tree $\mu(v_i)$ contains a branching vertex.
If $d\geq 1$ however, then a $d$-branching vertex need not always exist. 
For example, $\mu(v_i)$ could be a path with each vertex incident to only a few edges of the form $\mu(v_iv_j)$. 
The following lemma shows that this is essentially the only obstruction:
\begin{lemma}
\label{lbranchinglemma}
Let $\mu$ be a $K_m$-model with $m\geq 4$.
Let $0 < d < \frac{m-1}{2}$ be an integer and suppose $\mu(v_i)$ does not contain a $d$-branching vertex. Then 
\begin{enumerate}
	\item
	the $\mcc_i^{d}$-subtree of $\mu(v_i)$ is a path $R$, and
	\item
	for each $s \in \mcc_i^d$, there are at most $3d$ indices $j \in [m]-\{i\}$ such that the (possibly trivial) $\mu(v_iv_j)_i$-$\mcc_i^d$-path in $\mu(v_i)$ ends at $s$. 
	In particular, at least  $m-1-6d$ of these paths end at an internal vertex of $R$.
\end{enumerate}
\end{lemma}

\begin{proof}
Suppose that the $\mcc_i^{d}$-subtree of $\mu(v_i)$ is not a path. 
Then there exists a vertex $u \in \mcc_i^d$ that is adjacent in $\mu(v_i)$ to at least three vertices in $\mcc_i^d$, say $s_1,s_2$, and $s_3$. 
For each $k \in [3]$, since $s_k$ is $d$-central, there are more than $d$ indices $j \in [m]-\{i\}$ such that $\mu(v_iv_j)_i \in V(T_k)$, where $T_k$ is the connected component of $\mu(v_i)-u$ containing $s_k$.
Thus, for all $Y \subseteq [m]-\{i\}$ with $|Y| \leq d$, there is an index $j_k \not\in Y$ such that $\mu(v_iv_{j_k})_i \in V(T_k)$.
This implies that $u$ is $d$-branching and proves the first statement of the lemma.

Now let $s \in \mcc_i^d$. 
For each connected component $T$ of $\mu(v_i)-s$ not containing a vertex in $\mcc_i^d$, there are at most $d$ indices $j \in [m]-\{i\}$ such that $\mu(v_iv_j)_i \in V(T)$, since otherwise the neighbour of $s$ in $T$ would be $d$-central.
Let $J$ denote the set of indices $j \in [m]-\{i\}$ such that the $\mu(v_iv_j)_i$-$\mcc_i^d$-path in $\mu(v_i)$ ends at $s$.
The second statement of the lemma asserts that $|J|\leq 3d$.

Suppose $|J| \geq 3d+1$.
We show that $s$ must then be $d$-branching.
Let $Y \subseteq [m]-\{i\}$ with $|Y| \leq d$ and let $c$ denote the number of indices $j \in J-Y$ such that $s = \mu(v_iv_j)_i$. 
Note that $|J-Y|\geq 2d+1$.

If $c\geq 3$, then clearly $s$ branches avoiding $Y$, so we may assume $c \in \{0,1,2\}$.
Then there are at least $2d-c+1$ indices $j \in J-Y$ such that $\mu(v_iv_j)_i$ is contained in a connected component of $\mu(v_i)-s$ not containing a vertex in $\mcc_i^d$. 
But each such component contains $\mu(v_iv_j)_i$ for at most $d$ indices $j \in J-Y$, so there are at least $3-c$ distinct connected components of $\mu(v_i)-s$ containing $\mu(v_iv_j)_i$ for some $j \in J-Y$. 
These $3-c$ components together with the $c$ edges on $s$ imply that $s$ branches avoiding $Y$.
Hence $s$ is $d$-branching and this proves the second statement of the lemma.
\end{proof}


We now show, as discussed in Remark \ref{rem:Ktbip}, that if $\mu$ is a $\Gamma$-bipartite $K_m$-model, $m\geq 4$, then $(\mu,\gamma)$ is a $\Gamma$-bipartite $\Gamma$-labelled graph.

\begin{lemma}
\label{bipmodelpathlemma}
Let $m\geq 4$ and let $\mu$ be a $\Gamma$-bipartite $K_m$-model in a $\Gamma$-labelled graph $(G,\gamma)$. Then $(\mu[V(K_m)],\gamma)$ is a $\Gamma$-bipartite $\Gamma$-labelled graph.
\end{lemma}
\begin{proof}
Write $V(K_m)=\{v_1,\dots,v_m\}$. 
Let us first prove the following claim.
\begin{claim}
If $P$ is a $b(\mu)$-path in $\mu$, then $2\gamma(P)=0$.
\end{claim}
\begin{subproof}
If $m=4$, then $\mu[V(K_m)]$ is $\Gamma$-bipartite by definition and the claim follows by Proposition \ref{prop:modelbranch3conn} and Lemma \ref{lemma3connshiftequivalent}.

So suppose $m\geq 5$ and let $P$ be a $b(\mu)$-path in $\mu$ with endpoints $u$ and $w$.
First suppose that $u$ and $w$ are in different trees of $\mu$, say in $\mu(v_1)$ and $\mu(v_2)$ respectively.
Since $u$ is branching, there are two indices in $[m]-\{1,2\}$, say 3 and 4, such that $u$ branches to $\{v_2,v_3,v_4\}$.
Thus $u$ is also a branching vertex in the $K_4$-submodel $\eta$ of $\mu$ restricted to $\{v_1,v_2,v_3,v_4\}$.
Since the claim holds for $m=4$, if $w$ is also branching in $\eta$, then $2\gamma(P)=0$ as desired.
Otherwise, since $w$ is branching in $\mu$ but not in $\eta$, there exists another index in $[m]-\{1,2,3,4\}$, say 5, such that $w$ branches to both $\{v_1,v_3,v_5\}$ and $\{v_1,v_4,v_5\}$.
Furthermore, $u$ must branch to at least one of $\{v_2,v_3,v_5\}$ or $\{v_2,v_4,v_5\}$.
Without loss of generality, assume that $u$ branches to $\{v_2,v_3,v_5\}$.
Then $u$ and $w$ are both branching in the $K_4$-submodel of $\mu$ restricted to $\{v_1,v_2,v_3,v_5\}$, so $2\gamma(P)=0$.
See Figure \ref{fig:BipK5a}.

So we may assume that $u$ and $w$ are in the same tree, say $\mu(v_1)$.
Since $u$ and $w$ are both branching in $\mu$, there are four indices in $[m]-\{1\}$, say 2, 3, 4, and 5, such that $u$ branches to both $\{v_2,v_3,v_4\}$ and $\{v_2,v_3,v_5\}$ and $w$ branches to both $\{v_2,v_4,v_5\}$ and $\{v_3,v_4,v_5\}$.

For the rest of the proof of this claim, we only consider the $K_5$-submodel of $\mu$ restricted to $\{v_1,\dots,v_5\}$.
For notational convenience we simply assume that $\mu$ is a $K_5$-model.

Let $\eta$ be the $K_4$-submodel of $\mu$ restricted to $\{v_2,v_3,v_4,v_5\}$ and, for $i \in \{2,3,4,5\}$, let $x_i$ be the unique vertex in $\mu(v_i)$ that is branching in $\eta$. 
Let $Q_{ij}$ denote the unique $x_i$-$x_j$-path in $\mu[v_i,v_j]$ for $i,j \in \{2,3,4,5\}$.
Note that $2\gamma(Q_{ij})=0$ by the $m=4$ case of the claim.

Let $P_i$ denote the $u$-$x_i$-path in $\mu[v_1,v_i]$ for $i\in\{2,3\}$ and let $y_i$ be the closest vertex to $u$ on $P_i$ such that $y_i \in V(\mu(v_i))$ and $y_i$ is branching in $\mu$. 
Similarly, for $j \in \{4,5\}$, let $P_j$ denote the $w$-$x_j$-path in $\mu[v_1,v_j]$ and let $y_j$ be the closest vertex to $w$ on $P_j$ such that $y_j \in V(\mu(v_j))$ and $y_j$ is branching in $\mu$.
Note that $2\gamma(uP_2y_2)=2\gamma(uP_3y_3)=2\gamma(wP_4y_4)=2\gamma(wP_5y_5)=0$ since these are $b(\mu)$-paths with endpoints in different trees.

Suppose $y_2 \in V(Q_{24} \cup Q_{25})$.
Then $y_2$ branches to $\{v_1,v_4,v_5\}$, so it is a branching vertex in the $K_4$-submodel $\pi$ of $\mu$ restricted to $\{v_1,v_2,v_4,v_5\}$.
Since $w$ is also branching in $\pi$, we have $2\gamma(y_2P_2uPw)=0$ by the $m=4$ case.
But as previously noted, we also have $2\gamma(uP_2y_2)=0$, which implies $2\gamma(P)=0$.
Therefore we may assume that $y_2 \in V(Q_{23})-x_2$ and, by symmetry, $y_3 \in V(Q_{23})-x_3$.
Similarly, we may assume that $y_j \in V(Q_{45})-x_j$ for $j \in \{4,5\}$.
See Figure \ref{fig:BipK5b}.

Since the two cycles $P \cup P_4 \cup Q_{24} \cup P_2$ and $P\cup P_4 \cup Q_{34} \cup P_3$ are in $\mu[v_1,v_2,v_3,v_4]$, they are both $\Gamma$-zero, and so $\gamma(Q_{24})+\gamma(P_2) = \gamma(Q_{34})+\gamma(P_3)$.
But $P_2 \cup Q_{24} \cup Q_{34} \cup P_3$ is also a cycle in $\mu[v_1,v_2,v_3,v_4]$, so $\gamma(P_2)+\gamma(Q_{24})+\gamma(Q_{34})+\gamma(P_3)=0$, which gives $0 = 2\gamma(P_2)+2\gamma(Q_{24}) = 2\gamma(P_2)$ since $2\gamma(Q_{24})=0$.
By symmetry, we have $2\gamma(P_2)=2\gamma(P_4)=0$.
Now since $P \cup P_4 \cup Q_{24} \cup P_2$ is a $\Gamma$-zero cycle, it follows that $2\gamma(P)=2\gamma(P)+2\gamma(P_2)+2\gamma(Q_{24})+2\gamma(P_4) = 0$.
This completes the proof of the claim.
\end{subproof}

\begin{figure}[t]
	\centering
	\begin{subfigure}{0.43\textwidth} 
		\centering
		\begin{center}
	\resizebox{0.95\textwidth}{!}{
		\begin{tikzpicture}[scale=0.85]
			\tikzstyle{smallvx}=[thick,circle,inner sep=0.cm, minimum size=2mm, fill=black, draw=black]
			\tikzstyle{tinyvx}=[thick,circle,inner sep=0.cm, minimum size=0.1mm, fill=black, draw=black]
			\tikzstyle{hedge}=[line width=1pt]
			\tikzstyle{jedge}=[line width=1.5pt]
			
			\node[ellipse, draw=black, text height=.75cm, text width=1.5cm] at (-1.7,6) [label=above:$\mu(v_1)$] {};
			\node[circle, draw=black, text width=1.5cm] at (1.4,5.8) [label=above:$\mu(v_2)$] {};
			
			\node[smallvx] (u) at (-1,6) [label=above:$u$] {};
			\node[smallvx] (w) at (1,6) [label=above:$w$] {};
			\node[circle, draw=black] (x3) at (0,4) {$\mu(v_3)$};
			\node[circle, draw=black] (x4) at (0,2) {$\mu(v_4)$};
			\node[circle, draw=black] (x5) at (0,0) {$\mu(v_5)$};
			\node[tinyvx] (z) at (1.5,5.5) {};
			
			\draw[jedge] (u) edge node[above] {$P$} (w);
			\draw[hedge] (u) edge[out=160, in=180, looseness=1.6] node[left, pos=0.3] {} (x4);
			\draw[hedge] (u) edge[out=220, in=180, looseness=1.5] node[left] {} (x3);
			\draw[hedge] (z) edge[out=-40, in=0, looseness=1.2] (x4);
			\draw[hedge] (z) edge[out=-80,in=0, looseness=1] (x3);
			\draw[hedge] (z) edge (w);
			\draw[hedge] (w) edge[out=20, in=0, looseness=1.3] (x5);
			\draw[hedge] (x5) edge[out=180, in=160, looseness=1.3] (u);
		\end{tikzpicture}
	}
\end{center}
		\caption{}
		\label{fig:BipK5a}
	\end{subfigure}
	\begin{subfigure}{0.51\textwidth} 
		\centering
		\begin{center}
	\resizebox{0.95\textwidth}{!}{
		\begin{tikzpicture}[scale=0.85]
			\tikzstyle{smallvx}=[thick,circle,inner sep=0.cm, minimum size=2mm, fill=black, draw=black]
			\tikzstyle{tinyvx}=[thick,circle,inner sep=0.cm, minimum size=1.5mm, fill=black, draw=black]
			\tikzstyle{hedge}=[line width=1pt]
			\tikzstyle{jedge}=[line width=1.5pt]
			
			\node[ellipse, draw=black, text width=2.2cm, text height=0.7cm] at (2,6) [label=above:$\mu(v_1)$] {};
			\draw (0.3,0.3) circle (1.2);
			\draw (0.3,3.7) circle (1.2);
			\draw (3.7,0.3) circle (1.2);
			\draw (3.7,3.7) circle (1.2);
			
			\node[smallvx] (u) at (1,6) [label=above:$u$] {};
			\node[smallvx] (w) at (3,6) [label=above:$w$] {};
			\node[smallvx] (x2) at (0,4) [label=above:$x_2$] {};
			\node[smallvx] (x3) at (0,0) [label=below:$x_3$] {};
			\node[smallvx] (x4) at (4,4) [label=above:$x_4$] {};
			\node[smallvx] (x5) at (4,0) [label=below:$x_5$] {};
			
			\node[tinyvx] (y2) at (0,3) [label=right:\small $y_2^{\ }$] {};
			\node[tinyvx] (y4) at (4,3) [label=left :\small $y_4^{\ }$] {};
			\node[tinyvx] (y3) at (0,1) [label=right:\small $y_3^{\ }$] {};
			\node[tinyvx] (y5) at (4,1) [label=left :\small $y_5^{\ }$] {};
			
			\draw[jedge] (u) edge node[above] {$P$} (w);
			\draw[jedge] (x2) edge node[left] {$Q_{23}$} (x3);
			\draw[jedge] (x2) edge node[above] {$Q_{24}$} (x4);
			\draw[jedge] (x2) edge (x5);
			\draw[jedge] (x3) edge (x4);
			\draw[jedge] (x3) edge (x5);
			\draw[jedge] (x4) edge node[right] {$Q_{45}$} (x5);
			
			\draw[hedge] (u) edge[bend right=60, looseness=1.2] node[left, pos=0.3] {$P_2$} (0,3);
			\draw[hedge] (u) edge[bend right=100, looseness=1.5] node[left, pos=0.35] {$P_3$} (0,1);
			\draw[hedge] (w) edge[bend left=60, looseness=1.2] node[right, pos=0.3] {$P_4$} (4,3);
			\draw[hedge] (w) edge[bend left=100, looseness=1.5] node[right, pos=0.35] {$P_5$} (4,1);
			\draw[dashed] (u) edge[bend right] (0.8,4);
		\end{tikzpicture}
	}
\end{center}
		\caption{}
		\label{fig:BipK5b}
	\end{subfigure}
	\caption{A $b(\mu)$-path $P$ in a $K_5$-model $\mu$. Each ellipse indicates a tree $\mu(v_i)$.}
	\label{FigBipK5}
\end{figure}

Now suppose $\mu[V(K_m)]$ contains a $\Gamma$-nonzero cycle and choose such a cycle $C$ minimizing the number $\ell$ of edges of the form $\mu(v_iv_j)$. 
Then $\ell \geq 5$ by the definition of a $\Gamma$-bipartite $K_m$-model.
Let $\mu(v_{i_1}),\mu(v_{i_2})$, and $\mu(v_{i_3})$ be three consecutive trees visited by $C$ in that order.
Then $\mu[v_{i_1},v_{i_3}]$ contains a $V(C)$-path $Q$ that is not an edge of $C$, and the two cycles $C_1,C_2$ in $C\cup Q$ distinct from $C$ both contain less than $\ell$ edges of the form $\mu(v_iv_j)$.
Hence $\gamma(C_1)=\gamma(C_2)=0$ by minimality of $\ell$.
Moreover, the two vertices in $V(C)\cap V(Q)$ are clearly branching in $\mu$, so $Q$ is a concatenation of $b(\mu)$-paths in $\mu$, and hence $2\gamma(Q)=0$ by the Claim. 
We thus have $0=\gamma(C_1)+\gamma(C_2) = \gamma(C)+2\gamma(Q)=\gamma(C)$, a contradiction.
This completes the proof of the lemma.
\end{proof}

Next we prove two technical lemmas which will be used in the proof of Lemma \ref{oddmodellemma}.

\begin{lemma} \label{lem:twoattachedpaths}
	Let $\Gamma$ be an abelian group and let $(G,\gamma)$ be a $\Gamma$-labelled graph.
	Let $P$ be a $\Gamma$-nonzero path in $(G,\gamma)$ with endpoints $s_1,s_2$.
	Let $s'_1,s'_2\in V(G)-V(P)$ and for each $i\in[2]$, let $Q_i$  be a path from $s'_i$ to $V(P)-\{s_1,s_2\}$ such that $Q_1$ is disjoint from $Q_2$.
	Then there is a $\Gamma$-nonzero $\{s_1,s_2,s'_1,s'_2\}$-path in $P\cup Q_1\cup Q_2$ with different endpoints than $P$.
\end{lemma}
\begin{proof}
	Note that $P\cup Q_1\cup Q_2$ is a tree with four leaves $\{s_1,s_2,s'_1,s'_2\}$ and two vertices of degree 3, say $x_1$ and $x_2$ where $x_i$ is the endpoint of $Q_i$ in $V(P)$ for each $i\in[2]$.
	Let us assume without loss of generality that $s_1,x_1,x_2,s_2$ occur in this order on $P$.
	If the four $\{s_1,s_2\}$-$\{s'_1,s'_2\}$-paths in $P\cup Q_1\cup Q_2$ are all $\Gamma$-zero, then for each $i\in[2]$ we have $\gamma(s_1 P x_i)=\gamma(x_i P s_2) = -\gamma(Q_i)$, which gives $2\gamma(Q_i) = -\gamma(P)\neq 0$.
	Then $\gamma(x_1 P x_2) = \gamma(x_1 Ps_2)-\gamma(x_2Ps_2) =-\gamma(Q_1)+\gamma(Q_2)$, but this implies that the $s'_1$-$s'_2$-path $s'_1 Q_1x_1Px_2Q_2s'_2$ has weight $\gamma(Q_1) +(-\gamma(Q_1)+\gamma(Q_2)) + \gamma(Q_2) = 2\gamma(Q_2)\neq 0$.
\end{proof}

In the following lemma, we use the notation arising in the proof of Lemma \ref{oddmodellemma}.

\begin{lemma} \label{lem:sixpaths}
	Let $\Gamma$ be an abelian group.
	Let $(G',\gamma')$ be a $\Gamma$-labelled graph and let $s_1, s_2, s'_{j_1^1}, s'_{j_2^1}, s'_{j_1^2}, s'_{j_2^2}$ be six distinct vertices in $V(G')$.
	Let $(Q_1^1,Q_2^1,Q_1^2,Q_2^2,Q,P_1)$ be a tuple of six paths in $(G',\gamma')$ satisfying the following properties (see Figure \ref{FigNZSpath}):
	\begin{enumerate}
		\itemsep 0.2em \parskip 0em  \partopsep=0pt \parsep 0em  
		\item
		For each $i,k\in [2]$, $Q_k^i$ is an $s_i$-$s_{j_k^i}'$-path.
		\item
		$Q$ is a $\Gamma$-zero $s_1$-$s_2$-path.
		\item
		The five paths $Q_1^1,Q_2^1,Q_1^2,Q_2^2, Q$ are disjoint except that for each $i\in[2]$, the vertex $s_i$ belongs to exactly three paths $Q_1^i, Q_2^i, Q$.
		\item
		$P_1$ is a $\Gamma$-nonzero $s_1$-$s_2$-path.
	\end{enumerate}
	Then the union of the six paths $Q_1^1,Q_2^1,Q_1^2,Q_2^2,Q,P_1$ contains a $\Gamma$-nonzero $\{s'_{j_1^1}, s'_{j_2^1}, s'_{j_1^2}, s'_{j_2^2}\}$-path.
\end{lemma}
\begin{proof}
	Let $S'' = \{s'_{j_1^1}, s'_{j_2^1}, s'_{j_1^2}, s'_{j_2^2}\}$ and let $T = Q_1^1 \cup Q_2^1 \cup Q_1^2 \cup Q_2^2 \cup Q$.
	Suppose there does not exist a $\Gamma$-nonzero $S''$-path in $T\cup P_1$.
	Then the edges of $P_1$ can be partitioned into a sequence of maximal paths contained in $T$ and $V(T)$-paths not contained in $T$.
  Note that $P_1$ has at least one edge not in $E(T)$, since otherwise we would have $P_1=Q$, a contradiction as $\gamma'(P_1)\neq 0=\gamma'(Q)$.
	Hence, there exists a subpath $P$ of $P_1$ that is a $V(T)$-path not contained in $T$.
	Let $u$ and $v$ denote the two endpoints of $P$.

	\begin{figure}[t]
		\centering
		\begin{tikzpicture}[scale=0.8]
\tikzstyle{smallvx}=[thick,circle,inner sep=0.cm, minimum size=2mm, fill=black, draw=black]
\tikzstyle{hedge}=[line width=1pt]

\node[smallvx] (sj11) at (0,4) [label=left:$s_{j_1^1}'$] {};
\node[smallvx] (sj12) at (0,0) [label=left:$s_{j_2^1}'$] {};
\node[smallvx] (s1) at (2,2) [label=left:$s_1$] {};
\node[smallvx] (s2) at (5,2) [label=right:$s_2$] {};
\node[smallvx] (sj21) at (7,4) [label=right:$s_{j_1^2}'$] {};
\node[smallvx] (sj22) at (7,0) [label=right:$s_{j_2^2}'$] {};

\draw[hedge] (s1) edge node[below] {$Q$} (s2);
\draw[hedge] (s1) edge node[above] {$\ Q_1^1$} (sj11);
\draw[hedge] (s1) edge node[below] {$\ Q_2^1$} (sj12);
\draw[hedge] (s2) edge node[above] {$Q_1^2\ \ $} (sj21);
\draw[hedge] (s2) edge node[below] {$Q_2^2\ \ $} (sj22);

\draw[dashed] (1.3,1.3) edge[bend left=45] (0.5,3.5);
\draw[dashed] (3.8,2) edge[bend left=75, looseness=1.8] (4.6,2);
\draw[dashed] (1.5,2.5) edge[bend left] (5.8,2.8);
\end{tikzpicture}
		\caption{The tree $T=Q_1^1 \cup Q_2^1 \cup Q_1^2 \cup Q_2^2 \cup Q$. The dashed lines indicate some possible $V(T)$-subpaths $P$ of $P_1$.}
		\label{FigNZSpath}
	\end{figure}

	Suppose that $u$ and $v$ both lie in $Q$.
	If $\gamma'(uQv) \neq \gamma'(P)$, then we get a $\Gamma$-nonzero $S''$-path say from $s'_{j_1^1}$ to $s'_{j_1^2}$ by rerouting the path $s_{j_1^1}'Ts_{j_1^2}'$ along $P$, a contradiction.
	On the other hand, if $\gamma'(uQv) = \gamma'(P)$, then we can reroute $Q$ through $P$ and still satisfy all the conditions of the lemma, while reducing the size of the union of the six paths.
	Therefore, we may assume that  $u$ and $v$ do not both lie in $Q$.
	Similarly, we may assume that $u$ and $v$ do not both lie in any one path $Q_k^i$. 
	
	Next suppose that $u$ and $v$ are in $Q_1^i \cup Q_2^i$ for some $i \in [2]$. 
	Assume without loss of generality that $u \in V(Q_1^i)-s_i$ and $v\in V(Q_2^i)-s_i$.
	Since there is no $\Gamma$-nonzero $S''$-path in $T\cup P$, by Lemma \ref{threepathscyclelemma}(a), we have $2\gamma'(P)=2\gamma'(uQ_1^is_i)=2\gamma'(vQ_2^is_i)=0$.
	Moreover, we have $\gamma'(P)=\gamma'(uTv)$, since otherwise there is a $\Gamma$-nonzero $s'_{j_1^i}$-$s'_{j_2^i}$-path in $T\cup P$. 
	Similarly, if one endpoint of $P$ is in $Q_k^i$ and the other is in $Q$, or if one endpoint is in $Q_k^1$ and the other is in $Q_\ell^2$, then $2\gamma'(P)=2\gamma'(uTs_i)=2\gamma'(vTs_i)=0$ for each $i\in[2]$ (since $\gamma'(Q)=0$), and $\gamma'(P)=\gamma'(uTv)$. 
	Thus, if $U$ denotes the set consisting of every vertex that is an endpoint of a subpath of $P_1$ that is a $V(T)$-path not contained in $T$, then for all $u\in U$, we have $2\gamma'(uTs_i)=0$ for each $i\in[2]$.
	This then implies that every path $R$ in $T$ with both endpoints in $U \cup \{s_1,s_2\}$ satisfies $2\gamma'(R)=0$.
	
	Now let $(s_1=u_1, u_2, \dots, u_n = s_2)$ denote the sequence of vertices in $U\cup\{s_1,s_2\}$ that occur on $P_1$ in this order.
	Let $W=u_1Tu_2Tu_3\dots u_{n-1}Tu_n$, which is a walk from $s_1$ to $s_2$ contained in $T$.
	We have shown above that $\gamma'(u_iTu_{i+1}) = \gamma'(u_iP_1u_{i+1})$ and $2\gamma'(u_iTu_{i+1})=0$ for all $i\in[n-1]$.
	The first equality implies that $\gamma'(W) := \sum_{i=1}^{n-1}\gamma'(u_iTu_{i+1}) = \gamma'(P_1) \neq 0$, whereas the second equality implies that $\gamma'(W) = \gamma'(s_1Ts_2) = \gamma'(Q)=0$, a contradiction.
\end{proof}

Let $R_4(n,m)$ denote the Ramsey numbers for 4-uniform hypergraphs. In other words, if $r \geq R_4(n,m)$, then for every red-blue colouring of the hyperedges of the complete 4-uniform hypergraph on $r$ vertices, there is either a red complete hypergraph on $n$ vertices or a blue complete hypergraph on $m$ vertices. The existence of $R_4(n,m)$ was proved by Ramsey \cite{Ramsey}.

We now prove the main lemma of this section.
\begin{lemma}
\label{oddmodellemma}
Let $\Gamma$ be an abelian group and let $t \geq 2$ be an integer.
Let $(G,\gamma)$ be a $\Gamma$-labelled graph containing a $K_{R_4(t,m(t))}$-model $\pi$ where $m(t) = 50(150t^4)^4+1+300t^4$. 
Then either there is a $\Gamma$-odd $K_t$ model in $(G,\gamma)$ that is an enlargement of $\pi$, or there exists $X \subseteq V(G)$ with $|X| < 50(150t^4)^4$ such that the $\mct_\pi$-large 3-block of $(G-X,\gamma)$ is $\Gamma$-bipartite.
\end{lemma}
\begin{proof} 
Let $\mch$ be the complete 4-uniform hypergraph on the vertex set of $K_{R_4(t,m(t))}$. 
Colour a hyperedge $\{w,x,y,z\}$ of $\mch$ red if $\pi[w,x,y,z]$ contains a $\Gamma$-nonzero cycle and blue if $\pi[w,x,y,z]$ is $\Gamma$-bipartite. 
Then there is either a $\Gamma$-odd $K_t$-submodel or a $\Gamma$-bipartite $K_{m(t)}$-submodel of $\pi$.
In the first case we are done (since submodels of $\pi$ are also enlargements of $\pi$), so we assume the existence of a $\Gamma$-bipartite $K_{m(t)}$-model $\mu$ that is an enlargement of $\pi$.
Write $V(K_{m(t)}) = \{v_1,\dots,v_{m(t)}\}$.
By Lemma \ref{bipmodelpathlemma}, $(\mu,\gamma)$ is a $\Gamma$-bipartite $\Gamma$-labelled graph and we may assume by Lemma \ref{lemma3connshiftequivalent} and Proposition \ref{prop:modelbranch3conn} that, after possibly shifting,
\begin{equation}\label{branchingpathzeroequation}
\text{every $b(\mu)$-path in $(\mu,\gamma)$ is $\Gamma$-zero.} \tag{$*$}
\end{equation}

For each $i=1,2,\dots, 50(150t^4)^4$, sequentially in this order, pick a vertex $s_i \in \mu(v_i)$ as follows. 
If $\mu(v_i)$ has a $50t^4$-branching vertex, then define $s_i$ to be such a vertex. 
Otherwise, by Lemma \ref{lbranchinglemma}, the $\mcc_i^{50t^4}$-subtree of $\mu(v_i)$ is a path $R$ and there are at least $50(150t^4)^4$ indices $j \in [m(t)]-\{i\}$ such that the $\mu(v_iv_j)_i$-$R$-path in $\mu(v_i)$ ends at an internal vertex of $R$.
Pick one such index $j$ such that
\begin{equation*}
j \not \in \{\kappa(i'):i'<i \text{ and $\kappa(i')$ was previously defined}\},
\end{equation*}
and define $\kappa(i)=j$.
Such an index $j$ always exists as long as $i \leq 50(150t^4)^4$.
Define $s_i$ to be the (internal) vertex of $R$ such that the $\mu(v_iv_{\kappa(i)})_i$-$R$-path in $\mu(v_i)$ ends at $s_i$.
Note that $s_i$ is branching in $\mu$ as it is an internal vertex of $R$.

Define $S = \{s_i: i \in [50(150t^4)^4]\}$.
Since each $s_i$ is branching in $\mu$, we have $S \subseteq b(\mu)$.
By Theorem \ref{nzapathslemma}, either there exist $150t^4$ disjoint $\Gamma$-nonzero $S$-paths in $(G,\gamma)$ or there exists $X \subseteq V(G)$ with $|X| \leq 50(150t^4)^4-3$ such that $(G-X,\gamma)$ does not contain a $\Gamma$-nonzero $S$-path.

\begin{claim}
If $X\subseteq V(G)$, $|X| \leq 50(150t^4)^4-3$, and $(G-X,\gamma)$ does not contain a $\Gamma$-nonzero $S$-path, then the $\mct_\pi$-large 3-block $(B,\gamma_B)$ of $(G-X,\gamma)$ is $\Gamma$-bipartite.
\end{claim}

\begin{subproof} 
Since $|S|=50(150t^4)^4$, there are three vertices of $S$, say $s_1,s_2,s_3$ without loss of generality, such that $X$ is disjoint from $\mu(v_1) \cup \mu(v_2) \cup \mu(v_3)$.
Note that for each $i\in[3]$, $\mu(v_i)$ intersects $V(B)$ by the definition of $\mct_\pi$.
For $i \in [3]$, define $x_i \in V(\mu(v_i))$ to be $s_i$ if $s_i \in V(B)$ and, otherwise, a closest vertex to $s_i$ in $\mu(v_i)$ that is in $V(B)$.

Now suppose that $(B,\gamma_B)$ contains a simple $\Gamma$-nonzero cycle $C_B$.
Since $B$ is 3-connected, there exist three disjoint $V(C_B)$-$\{x_1,x_2,x_3\}$-paths $P_1',P_2',P_3'$ in $B$. 
Applying Proposition \ref{prop:3blockpathcycle}, we obtain a $\Gamma$-nonzero cycle $C$ of $(G-X,\gamma)$ corresponding to $C_B$ and three disjoint $V(C_B)$-$\{x_1,x_2,x_3\}$-paths  $P_1,P_2,P_3$ in $(G-X,\gamma)$ corresponding to $P_1',P_2',P_3'$ respectively.
Since $X$ is disjoint from $\mu(v_i)$ for each $i\in[3]$, the three paths $x_i\mu(v_i)s_i$ together with $C\cup P_1\cup P_2\cup P_3$ gives three disjoint \mbox{$V(C)$-$\{s_1,s_2,s_3\}$-paths} in $(G-X,\gamma)$.
By Lemma \ref{lem:threeACpathsnonzero}, there exists a $\Gamma$-nonzero $S$-path, a contradiction. 
Thus $(B,\gamma_B)$ has no simple $\Gamma$-nonzero cycle and by Lemma \ref{lemma3connshiftequivalent}, $(B,\gamma_B)$ is $\Gamma$-bipartite.
\end{subproof}

So we may assume that there exist $150t^4$ disjoint $\Gamma$-nonzero $S$-paths in $(G,\gamma)$.
Let $\mcp=\{P_1,\dots,P_{150t^4}\}$ be a set of $150t^4$ disjoint $\Gamma$-nonzero $S$-paths minimizing the number of edges in $\cup \mcp$ not in a tree $\mu(v_i)$ of $\mu$; that is, we minimize
 $$\left|\left(\textstyle\bigcup_{j=1}^{150t^4} E(P_j)\right) - \left(\textstyle\bigcup _{j=1}^{m(t)}E(\mu(v_j))\right)\right|.$$ 
By relabelling indices in $[50(150t^4)^4]$ and updating $\kappa$ accordingly, we may assume that $P_i$ has endpoints $s_{2i-1}$ and $s_{2i}$ for all $i\in [150t^4]$. 
Note that $|S|=50(150t^4)^4$ and $|\mcp|=150t^4$.

\begin{claim}
\label{pathsintersectingtreesclaim}
There are at most $|\mcp|$ indices $j$ such that $2|\mcp|< j \leq |S|$ (i.e.~$s_j$ is not an endpoint of a path in $\mcp$) and $\mu(v_j)$ intersects a path in $\mcp$. 
\end{claim}

\begin{subproof}
Suppose there are more than $|\mcp|$ indices $j$ such that $j>2|\mcp|$ and $\mu(v_j)$ intersects a path in $\mcp$. 
For each such index $j$, pick a closest vertex $x_j$ to $s_j$ in $\mu(v_j)$ such that $x_j$ is in a path in $\mcp$.
Then there exist two such indices $j_1$ and $j_2$ such that $x_{j_1}$ and $x_{j_2}$ belong to one path $P$ of $\mcp$. 
For each $i\in[2]$, let $Q_i = s_{j_i}\mu(v_{j_i})x_{j_i}$.
Then each $Q_i$ is disjoint from every path in $\mcp-\{P\}$, and every $S$-path contained in $P\cup Q_1\cup Q_2$ distinct from $P$ has fewer edges not in a tree of $\mu$ than $P$.
But by Lemma \ref{lem:twoattachedpaths}, $P\cup Q_1\cup Q_2$ contains a $\Gamma$-nonzero $S$-path $P'$ distinct from $P$, so $(\mcp-\{P\})\cup \{P'\}$ contradicts our choice of $\mcp$.
\end{subproof}

By relabelling among indices $j$ with $2|\mcp| < j \leq |S|$ and updating $\kappa$ accordingly, we may assume that no path in $\mcp$ contains a vertex in $\mu(v_j)$ for all $j$ with $3|\mcp| < j \leq |S|$.
We now construct a minor $G'$ of $G$ as follows. 
Let $J' = \{j: 3|\mcp| < j \leq |S|\}$ and for each $j \in J'$, contract the tree $\mu(v_j)$ into a single vertex $s_j'$ and delete all loops. 
Let $S'=\{s_j':j \in J'\}$.

Define a $\Gamma$-labelling  $\gamma'$ of $G'$ as follows. 
Let $e=x'y' \in E(G')$.
If $e$ is not incident to $S'$, then define $\gamma'(e)=\gamma(e)$. 
If $x'=s_j'$ and $y' \not\in S'$, then let $x$ be the endpoint of $e$ in $G$ in $\mu(v_j)$ and define $\gamma'(e) = \gamma(e)+\gamma(x\mu(v_j)s_j)$. 
Similarly, if $x'=s_j'$ and $y'=s_k'$, then let $x$ and $y$ be the corresponding endpoints of $e$ in $G$ and define $\gamma'(e) = \gamma(e)+\gamma(x\mu(v_j)s_j)+\gamma(y\mu(v_k)s_k)$.
Then each $S'$-path $P'$ in $(G',\gamma')$ with endpoints $s_j'$ and $s_k'$ corresponds to an $s_j$-$s_k$-path $P$ in $(G,\gamma)$ of the same weight, obtained by extending the endpoints of $P'$ in $G$ along $\mu(v_j)$ and $\mu(v_k)$ to $s_j$ and $s_k$ respectively.

Let $\mu'$ be the resulting $K_{m(t)}$-model in $G'$ obtained from $\mu$. In other words, $\mu'(v_i)=\mu(v_i)$ for $i\not\in J'$, $\mu'(v_j) = \{s_j'\}$ for $j \in J'$, and $\mu'(v_iv_j)=\mu(v_iv_j)$ for all $i,j \in [m(t)]$.
Note that for $i \not\in J'$, the $d$-central and $d$-branching vertices in $\mu(v_i)$ and $\mu'(v_i)$ are the same.

\begin{claim}
\label{s'pathsclaim}
There exist $t$ disjoint $\Gamma$-nonzero $S'$-paths in $(G',\gamma')$.
\end{claim}

\begin{subproof}
Suppose not. By Theorem \ref{nzapathslemma}, there exists $Y \subseteq V(G')$ with $|Y|\leq 50t^4-4$ such that $(G'-Y,\gamma')$ does not contain a $\Gamma$-nonzero $S'$-path.
Recall that no path in $\mcp$ intersects $\mu(v_j)$ for $j\in J'$, so the paths in $\mcp$ are unaffected by the minor operations used to obtain $G'$. 
We may thus consider $\mcp$ also as a linkage in $(G',\gamma')$ whose paths are disjoint from $S'$.

Since $|Y| < 50t^4$ and $|\mcp|=150t^4$, there are more than $100t^4$ paths $P_i$ that are disjoint from $Y-S'$.
Among these paths there are more than $50t^4$ paths $P_i$ such that $Y$ is also disjoint from $\mu'(v_{2i-1})\cup \mu'(v_{2i})$.
Among these, there is a path $P_i$ such that $Y$ is also disjoint from $\mu'(v_{\kappa(2i-1)})\cup \mu'(v_{\kappa(2i)})$, where we define $\mu'(v_{\kappa(i)}) = \emptyset$ if $\kappa$ is not defined on $i$.
By relabelling the paths in $\mcp$ and updating the indices in $[2|\mcp|]$ and $\kappa$ accordingly, we may assume that $Y$ is disjoint from $P_1 \cup \mu'(v_1) \cup \mu'(v_2) \cup \mu'(v_{\kappa(1)}) \cup \mu'(v_{\kappa(2)})$.

Let $Y' \subseteq [m(t)]-\{1\}$ be the set of indices $j$ such that either $\mu'(v_j)$ contains a vertex in $Y$ or $j$ is equal to the index $\kappa(2)$, if defined.
Then $|Y'|\leq |Y|+1 \leq 50t^4-3$.

Define $j_1^1,j_2^1,j_3^1$ as follows.
\begin{description}
	\item[Case 1:]
	$s_1$ is $50t^4$-branching in $\mu'$.

	By the definition of a $50t^4$-branching vertex, there exist $j_1^1,j_2^1,j_3^1 \in [m(t)] - \{1\} - Y'$ such that $s_1$ branches to $\{\mu'(v_{j_1^1}^{\ }),\mu'(v_{j_2^1}^{\ }),\mu'(v_{j_3^1}^{\ })\}$ in $\mu'$.

	\item[Case 2:]
	$s_1$ is not $50t^4$-branching in $\mu'$.

	By Lemma \ref{lbranchinglemma}, the $\mcc_1^{50t^4}$-subtree of $\mu'(v_1)$ is a path $R$.
	By definition, $s_1$ is an internal vertex of $R$ and the $\mu'(v_1v_{\kappa(1)})_1$-$R$-path in $\mu'(v_1)$ ends at $s_1$.
	Since $s_1$ is an internal vertex of $R$, there are exactly two connected components $R_1$ and $R_2$ of $\mu'(v_1)-s_1$ containing a vertex in $\mcc_1^{50t^4}$, and for each $k \in [2]$, there are more than $50t^4$ indices $j$ such that $\mu'(v_1v_j)_1 \in V(R_k)$ (by the definition of a $50t^4$-central vertex).
	Choose one such index $j_k^1 \not\in Y'$ for each $k \in [2]$ and define $j_3^1 = \kappa(1)$.
	Then $s_1$ branches to $\{\mu'(v_{j_1^1}^{\ }),\mu'(v_{j_2^1}^{\ }),\mu'(v_{j_3^1}^{\ })\}$.
\end{description}
We choose $j_1^2,j_2^2,j_3^2$ in a similar manner for $\mu'(v_2)$ with the additional condition that $j_k^2 \not\in \{j_1^1,j_2^1,j_3^1\}$.
\begin{description}
	\item[Case 1:]
	$s_2$ is $50t^4$-branching in $\mu'$.
	
	Since $|Y'| \leq 50t^4-3$, we may choose $j_1^2,j_2^2,j_3^2 \in [m(t)] - \{2\} - (Y' \cup \{j_1^1,j_2^1,j_3^1\})$ such that $s_2$ branches to $\{\mu'(v_{j_1^2}^{\ }),\mu'(v_{j_2^2}^{\ }),\mu'(v_{j_3^2}^{\ })\}$ in $\mu'$.

	\item[Case 2:]
	$s_2$ is not $50t^4$-branching in $\mu'$.
	
	The $\mcc_2^{50t^4}$-subtree of $\mu'(v_2)$ is a path $R$ and $s_2$ is an internal vertex of $R$.
	Let $R_1$ and $R_2$ denote the two connected components of $\mu'(v_2)-s_2$ containing a vertex in $\mcc_2^{50t^4}$.
	Then for each $k\in[2]$, there are more than $50t^4$ indices $j$ such that $\mu'(v_2v_j)_2 \in V(R_k)$, and since $|Y'|\leq 50t^4-3$, we may choose one such index $j_k^2 \not\in Y'\cup \{j_1^1,j_2^1,j_3^1\}$.
	Define $j_3^2 = \kappa(2)$. 
\end{description} 

Note that $\mu'(v_{j_k^i})$ is disjoint from $Y$ for all $i\in[2]$ and $k\in[3]$.
For $i \in [2]$, let $T_i$ denote the $\big\{\mu'(v_iv_{j_1^i}^{\ }),\mu'(v_iv_{j_2^i}^{\ }),\mu'(v_iv_{j_3^i}^{\ })\big\}$-extension of $\mu'(v_i)$.
Then $T_i$ is a 3-star centered at $s_i$.
We modify $T_i$ by extending its legs if necessary so that its leaves are in $S'$ by the following procedure. 
For each $i \in [2]$ and $k \in [3]$, if $j_k^i \not\in J'$, then choose a new $\ell_k^i \in J' - \{j_1^1,j_2^1,j_3^1,j_1^2,j_2^2,j_3^2\}$ so that the $\ell_k^i$ are distinct and $Y$ is disjoint from each $\mu'(v_{\ell_k^i})$.
Extend the leg in $T_i$ ending with $\mu'(v_iv_{j_k^i})$ through $\mu'(v_{j_k^i})$ and through the edge $\mu'(v_{j_k^i}v_{\ell_k^i})$.
Since $\mu'(v_{j_k^i})$ is disjoint from $Y$, $T_i$ is still disjoint from $Y$.
Redefine $j_k^i$ to be $\ell_k^i$.

After this procedure, $\{s_{j_1^1}',s_{j_2^1}',s_{j_3^1}'\}$ and $\{s_{j_1^2}',s_{j_2^2}',s_{j_3^2}'\}$ are the leaves of $T_1$ and $T_2$ respectively, the six leaves are in $S'-Y'$ and distinct, and $T_1$ is disjoint from $T_2$.
For each $i \in [2]$ and $k \in [3]$, let $Q_k^i$ denote the path from $s_i$ to $s_{j_k^i}'$ in $T_i$.
Now consider the unique path $Q$ from $s_1$ to $s_2$ in $\mu'[v_1,v_2]$. 
Then for each $i\in [2]$, at least two of the paths $Q_1^i,Q_2^i,Q_3^i$ intersect $Q$ only at $s_i$.
Without loss of generality, assume that $Q_1^i,Q_2^i$ intersect $Q$ only at $s_i$.
Note that $\gamma'(Q)=\gamma(Q)=0$ by (\ref{branchingpathzeroequation}).
Hence, the five paths $(Q_1^1,Q_2^1, Q_1^2, Q_2^2, Q)$ satisfy the first three conditions of Lemma \ref{lem:sixpaths}.

Recall that $P_1$ (considered as a path in $(G',\gamma')$) is a $\Gamma$-nonzero path from $s_1$ to $s_2$ in $(G'-Y,\gamma')$ disjoint from $S'$.
We thus have a tuple of paths $(Q_1^1,Q_2^1,Q_1^2,Q_2^2,Q,P_1)$ in $(G'-Y,\gamma')$ satisfying the four conditions of Lemma \ref{lem:sixpaths}, from which it follows that there is a $\Gamma$-nonzero $\{s'_{j_1^1}, s'_{j_2^1}, s'_{j_1^2}, s'_{j_2^2}\}$-path (hence a $\Gamma$-nonzero $S'$-path) in $(G'-Y,\gamma')$, a contradiction.
This completes the proof of the claim.
\end{subproof}

Let $\mcq' = \{Q_1',\dots,Q_t'\}$ be a set of $t$ disjoint $\Gamma$-nonzero $S'$-paths in $(G',\gamma')$ minimizing the number of edges in $\cup\mcq'$ not contained in a tree $\mu'(v_k)$ of $\mu'$; that is, we minimize
$$\left|\left(\textstyle\bigcup_{i=1}^{t} E(Q_i')\right) - \left(\textstyle \bigcup_{k=1}^{[m(t)]} E(\mu'(v_k)) \right) \right|.$$ 
Let $j_1,\dots,j_{2t}\in J'$ be the indices such that $Q_i'$ has endpoints $s_{j_{2i-1}^{\ }}'$ and $s_{j_{2i}^{\ }}'$.
Since each $Q_i'$ is an $S'$-path, no path in $\mcq'$ contains a vertex $s_j'$ with $j \in J' - \{j_1,\dots,j_{2t}\}$.

\begin{claim} \label{claim:mcq3t}
There are at most $3t$ indices $k \not\in J'$ such that $\mu'(v_k)$ intersects a path in $\mcq'$.
\end{claim}

\begin{subproof}
Suppose there are more than $3t$ indices $k\not\in J'$ such that $\mu'(v_k)$ intersects a path in $\mcq'$.
Choose $3t+1$ such indices $k_1,\dots,k_{3t+1}$.
Let $\ell_1,\dots,\ell_{3t+1} \in J'-\{j_1,\dots,j_{2t}\}$ be distinct.
For each $i\in[3t+1]$, let $x_{k_i}$ be a vertex in $V(\mu'(v_{k_i}))$ that is in a path in $\mcq'$, chosen to be closest to $s'_{\ell_i}$ in $\mu'[v_{k_i}, v_{\ell_i}]$.
Then there is a path $Q'$ in $\mcq'$ containing at least four of the vertices $x_{k_i}$, say $x_{k_1},x_{k_2},x_{k_3},x_{k_4}$.
Let us assume without loss of generality that the vertices $x_{k_1},x_{k_2},x_{k_3},x_{k_4}$ occur in this order on $Q'$.
Define $R_2=x_{k_2}\mu'[v_{k_2},v_{\ell_2}]s'_{\ell_2}$ and $R_3=x_{k_3}\mu'[v_{k_3},v_{\ell_3}]s'_{\ell_3}$.
Note that $R_2$ and $R_3$ are disjoint from every path in $\mcq'-\{Q'\}$, and every $S'$-path in $Q'\cup R_2\cup R_3$ distinct from $Q'$ contains fewer edges not contained in a tree $\mu'(v_k)$ than $Q'$.
But by Lemma \ref{lem:twoattachedpaths}, $Q'\cup R_2\cup R_3$ contains a $\Gamma$-nonzero $S'$-path $Q''$ distinct from $Q'$, so $(\mcq-\{Q'\})\cup\{Q''\}$ contradicts our choice of $\mcq'$.
\end{subproof}

We now use $\mcq'$ to construct a $\Gamma$-odd $K_t$-model in $(G,\gamma)$ that is an enlargement of $\mu$.
Recall that each $S'$-path $Q_i'\in\mcq'$ corresponds to an $S$-path $Q_i$ in $(G,\gamma)$ of the same weight, obtained by extending the endpoints of $Q_i'$ in $G$ along $\mu(v_{j_{2i-1}^{\ }})$ and $\mu(v_{j_{2i}^{\ }})$ to $s_{j_{2i-1}^{\ }}$ and $s_{j_{2i}^{\ }}$ respectively.
Note in particular that $Q_i$ is disjoint from $\mu(v_{j_k})$ for all $k \in [2t]-\{2i-1,2i\}$, and that $Q_i\cap \mu(v_{j_{2i-1}^{\ }})$ and $Q_i\cap \mu(v_{j_{2i}^{\ }})$ are (possibly trivial) paths.

Let $\mcq=\{Q_1,\dots,Q_t\}$.
For simplicity of notation, we relabel the indices in $[50(150t^4)^4]$ so that $Q_i$ has endpoints $s_{2i-1}$ and $s_{2i}$ (we will not need $\mcp$ nor $J'$ for the remainder of the proof).
After this relabelling, we have that for each $i \in [t]$, $Q_i$ is disjoint from $\mu(v_j)$ for $j \in [2t] - \{2i-1,2i\}$, and that $Q_i \cap \mu(v_{2i-1})$ and $Q_i \cap \mu(v_{2i})$ are paths. 
Furthermore, by Claim \ref{claim:mcq3t}, we may assume that for all $j > 5t$, $\mu(v_j)$ does not intersect a path in $\mcq$ .

To construct the $\Gamma$-odd $K_t$-model, first define $L = \emptyset$.
For each $j = 1,\dots,2t$, sequentially, choose an index $\ell_j\not\in [5t]$ to add to $L$ as follows. Note that we keep $|L|\leq 2t$ throughout. Also recall that the definition of the vertex $s_j$ depends on whether it is $50t^4$-branching in $\mu$.
\begin{description}
	\item[Case 1:]
	$s_j$ is $50t^4$-branching in $\mu$.

	Since $Q_{\lceil j/2 \rceil} \cap \mu(v_j)$ is a path, we may choose $\ell_j \not\in L\cup [5t]$ such that the $s_j$-$\mu(v_{\ell_j})$-path in $\mu[v_j,v_{\ell_j}]$ is internally disjoint from $Q_{\lceil j/2\rceil}$. 
	Add $\ell_j$ to $L$.

	\item[Case 2:]
	$s_j$ is not $50t^4$-branching in $\mu$.

	By Lemma \ref{lbranchinglemma}, the  $\mcc_j^{50t^4}$-subtree of $\mu(v_j)$ is a path and $s_j$ is an internal vertex of this path. 
	Then there is a connected component $T$ of $\mu(v_j)-s_j$ containing a vertex in $\mcc_j^{50t^4}$ such that $T$ is disjoint from $Q_{\lceil j/2 \rceil}$ (since $Q_{\lceil j/2 \rceil} \cap \mu(v_j)$ is a path). 
	Moreover, there are more than $50t^4$ indices $k$ such that $\mu(v_jv_k)_j \in V(T)$. 
	Choose one such index $\ell_j$ such that $\ell_j \not\in L \cup [5t]$, and add $\ell_j$ to $L$.
\end{description}		

Define a $K_t$-model $\eta$ as follows. 
Let $\{w_1,\dots,w_t\}$ denote the vertices of $K_t$.
For each $i\in[t]$, define the tree 
\begin{equation*}
\eta(w_i) = \mu(v_{\ell_{2i-1}^{\ }}) \cup \mu(v_{\ell_{2i-1}^{\ }}v_{2i-1}) \cup\mu(v_{2i-1}) \cup Q_i \cup \mu(v_{2i}) \cup \mu(v_{2i}v_{\ell_{2i}^{\ }}) \cup \mu(v_{\ell_{2i}^{\ }}).
\end{equation*}
Note that the path from $\mu(v_{\ell_{2i-1}^{\ }})$ to $\mu(v_{\ell_{2i}^{\ }})$ in $\eta(w_i)$ contains $Q_i$ as a subpath.
For $i < j$, define the edges of $\eta$ as $\eta(w_iw_j) = \mu(v_{\ell_{2i}^{\ }}v_{\ell_{2j-1}^{\ }})$.
Clearly, $\eta$ is a $K_t$-model that is an enlargement of $\mu$ and hence of $\pi$.

\begin{figure}
	\centering
	\resizebox{0.6\textwidth}{!}{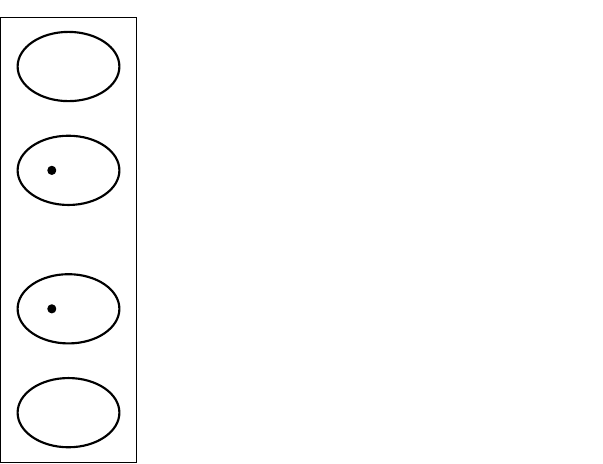}
	\caption{Three trees of the $\Gamma$-odd $K_t$-model $\eta$. Each ellipse represents a tree of $\mu$ and each rectangle represents a tree of $\eta$. The three dashed lines indicate the edges $\eta(w_iw_j)$, $\eta(w_iw_k)$, and $\eta(w_jw_k)$.}
	\label{fig:oddmodel}
\end{figure}

To show that $\eta$ is $\Gamma$-odd, we show in fact that the cycle contained in any three trees of $\eta$ is $\Gamma$-nonzero.
See Figure \ref{fig:oddmodel}.
Let $1\leq i<j<k \leq t$ and let $C$ be the unique cycle in $\eta[w_i,w_j,w_k]$.
Note that $C$ contains the edges $\{\eta(w_iw_j), \eta(w_iw_k), \eta(w_jw_k)\} = \{\mu(v_{\ell_{2i}}v_{\ell_{2j-1}}), \mu(v_{\ell_{2i}}v_{\ell_{2k-1}}), \mu(v_{\ell_{2j}}v_{\ell_{2k-1}})\}$, so we have $C\cap \eta(w_i) \subseteq \mu(v_{\ell_{2i}})$ and $C\cap \eta(w_k) \subseteq \mu(v_{\ell_{2k-1}})$.
Moreover, $C$ is the union of $Q_j$ (which is $\Gamma$-nonzero) and the $s_{2j-1}$-$s_{2j}$-path in $\mu$ through
\begin{align*}
	\mu(v_{2j-1}),\ \mu(v_{\ell_{2j-1}}),\ \mu(v_{\ell_{2i}}),\ \mu(v_{\ell_{2k-1}}),\ \mu(v_{\ell_{2j}}),\ \mu(v_{2j}) 
\end{align*}
in this order.
By (\ref{branchingpathzeroequation}), we have $\gamma(E(C)-E(Q_j))=0$.
Hence, $\gamma(C) = \gamma(Q_j) \neq 0$ and $\eta$ is a $\Gamma$-odd $K_t$-model that is an enlargement of $\pi$.
\end{proof}

\section{Proof of Theorem \ref{flatwallundirectedtheorem}: large flat wall}
\label{sec:flatwall}

In this section we deal with the second outcome of the flat wall theorem.

Let $W$ be an $r$-wall contained in a graph $G$ and let $P$ be a $b(W)$-path in $W$ with endpoints $x,y \in b(W)$ that is not contained in the perimeter of $W$.
Suppose there is an $x$-$y$-path $R$ in $G$ such that $R$ is disjoint from $W-P$, and let $W'$ be the wall obtained from $W$ by replacing $P$ with $R$.
Then we say that $W'$ is a \emph{local rerouting} of $W$.
Note that if $W'$ is a local rerouting of $W$, then $\mct_{W'}=\mct_W$ and, if $W$ is flat, then so is $W'$.

\begin{lemma}
\label{walllemma1}
Let $t \geq 4$ be an integer and let $\Gamma$ be an abelian group.
Let $(G,\gamma)$ be a $\Gamma$-labelled graph containing a flat $(t+2)^2$-wall $(W,\gamma)$.
Then there is a flat $t$-wall $(W_1,\gamma)$ with certifying separation $(C_1,D_1)$ such that $\mct_{W_1}$ is a truncation of $\mct_W$ and either
\begin{enumerate}
	\item[(i)]
	$(W_1,\gamma)$ is facially $\Gamma$-odd, or
	\item[(ii)] 
	the 3-block of $(D_1,\gamma)$ containing the degree 3 vertices of $(W_1,\gamma)$ is $\Gamma$-bipartite.
\end{enumerate}
\end{lemma}
\begin{proof}
Let $(W_2,\gamma)$ denote the $t$-subwall of $(W,\gamma)$ contained in the union of the $((t+2)i-t)$-th horizontal and vertical paths of $(W,\gamma)$, $i \in [t+1]$.
Note that $(W_2,\gamma)$ is flat and 1-contained in $(W,\gamma)$ since $(t+2)(t+1)-t < (t+2)^2$.
For $i,j\in[t+1]$, let $Q_i^{(h)}$ and $Q_j^{(v)}$ denote the $i$-th horizontal path and the $j$-th vertical path of $W_2$ respectively, and for $i,j\in[t]$ let $B_{i,j}$ denote the $(i,j)$-th brick of $W_2$.
Then the union of the $t+3$ horizontal and vertical paths of $W$ intersecting $B_{i,j}$ contains a $(t+2)$-subwall of $W$, which 1-contains a flat $t$-subwall $W_{i,j}$ of $W$. 
Let $(C_{i,j},D_{i,j})$ be a certifying separation for $W_{i,j}$ minimizing $|V(D_{i,j})|$.
Then $D_{i,j}$ is disjoint from $D_{i',j'}$ for $(i,j)\neq (i',j')$.

If for some $i,j\in[t]$, the 3-block of $(D_{i,j},\gamma)$ containing the degree 3 vertices of $(W_{i,j},\gamma)$ is $\Gamma$-bipartite, then $(W_{i,j},\gamma)$ and $(C_{i,j},D_{i,j})$ satisfies outcome (ii).
So we may assume that the 3-block of $(D_{i,j},\gamma)$ containing the degree 3 vertices of $(W_{i,j},\gamma)$ contains a $\Gamma$-nonzero cycle $O_{i,j}$ for all $i,j \in [t]$. 

For $i,j\in [t]$, let $P_{i,j}$ denote the subpath of $Q_{j+1}^{(v)}$ that is a $Q_i^{(h)}$-$Q_{i+1}^{(h)}$-path.
Let $H_{i,j}$ be the union of $D_{i,j}$, $P_{i,j}$, and the subpaths of horizontal paths of $W$ that are $W_{i,j}$-$P_{i,j}$-paths.
Then there are three disjoint paths from the interior of $P_{i,j}$ to $O_{i,j}$ in $H_{i,j}$, so by Lemma \ref{threepathscyclelemma}(b), there is a path $R_{i,j}$ in $H_{i,j}$ having the same endpoints as $P_{i,j}$ such that $\gamma(P_{i,j})\neq \gamma(R_{i,j})$.
Note that replacing $P_{i,j}$ with $R_{i,j}$ yields a local rerouting of $(W_2,\gamma)$.

We then obtain a facially $\Gamma$-odd $t$-wall $(W_1,\gamma)$ from $(W_2,\gamma)$ by a sequence of local reroutings where, for each $(i,j) \in [t]^2$ in lexicographic order, we replace $P_{i,j}$ with $R_{i,j}$ if necessary to make the $(i,j)$-th brick $\Gamma$-nonzero.
\end{proof}

\begin{lemma}
\label{flatwalllemma}
Let $r \geq 4$ be an integer and let $\Gamma$ be an abelian group.
Let $(G,\gamma)$ be a $\Gamma$-labelled graph containing a flat $(150r^{12}+2)^2$-wall $(W,\gamma)$.
Then one of the following outcomes hold:
\begin{enumerate}
	\item[(1)]
	There is a flat $50r^{12}$-wall $(W_1,\gamma)$ such that $\mct_{W_1}$ is a truncation of $\mct_W$ and either
	\begin{enumerate}
		\item[(i)]
		$(W_1,\gamma)$ is facially $\Gamma$-odd, or
		\item[(ii)] 
		$(W_1,\gamma)$ is strongly $\Gamma$-bipartite and there is a pure $\Gamma$-odd linkage of $(W_1,\gamma)$ of size $r$.
	\end{enumerate}
	\item[(2)]
	There exists $Z \subseteq V(G)$ with $|Z| < 50r^{12}$ such that the $\mct_W$-large 3-block of $(G-Z,\gamma)$ is $\Gamma$-bipartite.
\end{enumerate}
\end{lemma}
\begin{proof}
Applying Lemma \ref{walllemma1} with $t=150r^{12}$, we obtain a flat $150r^{12}$-wall $(W_0,\gamma)$ with top nails $N_0$ and certifying separation $(C_0,D_0)$ satisfying the conclusion of Lemma \ref{walllemma1}.
If $(W_0,\gamma)$ is facially $\Gamma$-odd, then outcome (1)-(i) is satisfied by taking a compact $50r^{12}$-subwall.
So we may assume that the 3-block $(B_0,\gamma_0)$ of $(D_0,\gamma)$ containing the vertices of degree 3 in $(W_0,\gamma)$ is $\Gamma$-bipartite.
Since $W_0$ is a $150r^{12}$-wall and $r\geq 4$, we have $|V(B_0)|\geq 4$, hence $B_0$ is 3-connected.
By Lemma \ref{lemma3connshiftequivalent}, we may assume by possibly shifting in $(G,\gamma)$ that 
\begin{equation} \label{eqn:wall3block0}
	\text{every $V(B_0)$-path in $(D_0,\gamma)$ is $\Gamma$-zero.} \tag{$\dagger$} 
\end{equation}
Also note that given any 1-contained subwall $W'$ of $W_0$ with the choice of nails and corners with respect to $W_0$, its branch vertices all have degree 3 in $W_0$, so $b(W') \subseteq V(B_0)$ and every $b(W')$-path in $(D_0,\gamma)$ is $\Gamma$-zero by \eqref{eqn:wall3block0}.

Since $(W_0,\gamma)$ is a flat $150r^{12}$-wall, it $50r^{12}$-contains a flat $50r^{12}$-subwall $(W_1,\gamma)$.
Let $N_1$ denote its top nails with respect to $(W_0,\gamma)$ and let $(C_1,D_1)$ be a certifying separation for $(W_1,\gamma)$ such that $|V(D_1)|$ is minimized.
Note that $D_1 \subseteq D_0$ and hence, by \eqref{eqn:wall3block0}, every path in $(D_1,\gamma)$ between branch vertices of $(W_1,\gamma)$ is $\Gamma$-zero.
In particular, $(W_1,\gamma)$ is strongly $\Gamma$-bipartite.

If there exist $r^3$ disjoint $\Gamma$-nonzero $N_1$-paths in $(G-(V(D_1)-N_1),\gamma)$, then by Lemma \ref{purelinkagelemma}, there is a pure linkage $\mcp$ of $(W_1,\gamma)$ whose paths are $\Gamma$-nonzero.
Since every $N_1$-path in $(D_1,\gamma)$ is $\Gamma$-zero, $\mcp$ is $\Gamma$-odd and outcome (1)-(ii) is satisfied.
So by Theorem \ref{nzapathslemma}, we may assume that there exists $Z \subseteq V(G-(V(D_1)-N_1))$ with $|Z| \leq 50r^{12}-3$ such that $(G-(V(D_1)-N_1)-Z,\gamma)$ does not contain a $\Gamma$-nonzero $N_1$-path.

Since $|Z| \leq 50r^{12}-3$ and $W_1$ is $50r^{12}$-contained in $W_0$, there are two vertical paths of $W_0$, one to the left of $W_1$ and one to the right, and two horizontal paths of $W_0$, one above $W_1$ and one below, such that the four paths are all disjoint from $Z$ and not contained in the perimeter of $W_0$.
Let $O$ denote the unique cycle in the union of these four paths.
See Figure \ref{FigFlatWall}.
Since $W_1$ is a $50r^{12}$-wall, there are two disjoint $V(O)$-$N_1$-paths $P_1$ and $P_2$ that are subpaths of vertical paths of $W_0$ disjoint from $Z$.
Note that $\gamma(P_1)=\gamma(P_2)=0$ since the endpoints of each $P_i$ have degree 3 in $W_0$.

\begin{figure}
	\centering
\begingroup%
  \makeatletter%
  \providecommand\color[2][]{%
    \errmessage{(Inkscape) Color is used for the text in Inkscape, but the package 'color.sty' is not loaded}%
    \renewcommand\color[2][]{}%
  }%
  \providecommand\transparent[1]{%
    \errmessage{(Inkscape) Transparency is used (non-zero) for the text in Inkscape, but the package 'transparent.sty' is not loaded}%
    \renewcommand\transparent[1]{}%
  }%
  \providecommand\rotatebox[2]{#2}%
  \newcommand*\fsize{\dimexpr\f@size pt\relax}%
  \newcommand*\lineheight[1]{\fontsize{\fsize}{#1\fsize}\selectfont}%
  \ifx\svgwidth\undefined%
    \setlength{\unitlength}{227.45196853bp}%
    \ifx\svgscale\undefined%
      \relax%
    \else%
      \setlength{\unitlength}{\unitlength * \real{\svgscale}}%
    \fi%
  \else%
    \setlength{\unitlength}{\svgwidth}%
  \fi%
  \global\let\svgwidth\undefined%
  \global\let\svgscale\undefined%
  \makeatother%
  \begin{picture}(1,0.60431202)%
    \lineheight{1}%
    \setlength\tabcolsep{0pt}%
    \put(0,0){\includegraphics[width=\unitlength,page=1]{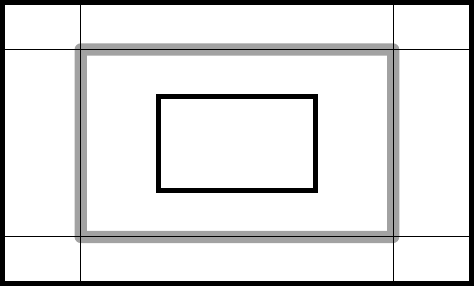}}%
    \put(0.02200983,0.28541886){\color[rgb]{0,0,0}\makebox(0,0)[lt]{\lineheight{1.25}\smash{\begin{tabular}[t]{l}$W_0$\end{tabular}}}}%
    \put(0.3509592,0.28570601){\color[rgb]{0,0,0}\makebox(0,0)[lt]{\lineheight{1.25}\smash{\begin{tabular}[t]{l}$W_1$\end{tabular}}}}%
    \put(0.18725168,0.28570059){\color[rgb]{0,0,0}\makebox(0,0)[lt]{\lineheight{1.25}\smash{\begin{tabular}[t]{l}$W^\circ$\end{tabular}}}}%
    \put(0.84722098,0.28490446){\color[rgb]{0,0,0}\makebox(0,0)[lt]{\lineheight{1.25}\smash{\begin{tabular}[t]{l}$O$\end{tabular}}}}%
    \put(0,0){\includegraphics[width=\unitlength,page=2]{FigFlatWalls.pdf}}%
    \put(0.42745728,0.43405199){\color[rgb]{0,0,0}\makebox(0,0)[lt]{\lineheight{1.25}\smash{\begin{tabular}[t]{l}$P_1$\end{tabular}}}}%
    \put(0.59232723,0.43405199){\color[rgb]{0,0,0}\makebox(0,0)[lt]{\lineheight{1.25}\smash{\begin{tabular}[t]{l}$P_2$\end{tabular}}}}%
  \end{picture}%
\endgroup%

	\caption{$W_0$ is a flat $150r^{12}$-wall which $50r^{12}$-contains a compact $50r^{12}$-subwall $W_1$. The cycle $O$ (highlighted in grey) around $W_1$ is disjoint from $Z$ and forms the perimeter of a compact subwall $W^\circ$ of $W_0$.
		The two $V(O)$-$N_1$-paths $P_1$ and $P_2$ are also disjoint from $Z$.}
	\label{FigFlatWall}
\end{figure}

Let $W^\circ$ denote the compact subwall of $W_0$ whose perimeter is $O$, and let $(C^\circ, D^\circ)$ be a certifying separation for $W^\circ$ in $G$ minimizing $|V(C^\circ \cap D^\circ) - V(B_0)|$.
Note that $W_1 \subseteq W^\circ \subseteq D^\circ \subseteq D_0$ and $V(C^\circ\cap D^\circ) \subseteq V(O)$.
We claim that $V(C^\circ \cap D^\circ) \subseteq V(B_0)$.
Suppose otherwise and let $v \in V(C^\circ \cap D^\circ) - V(B_0)$.
Then $v$ is contained in the interior of a $V(B_0)$-bridge $\mcb$ of $D_0$ with two attachments, say $a_1,a_2 \in V(B_0) \cap V(O)$.
But then $(C^\circ - (\mcb - \{a_1,a_2\}), D^\circ \cup \mcb)$ is also a certifying separation for $W^\circ$ with fewer vertices in $V(C^\circ \cap D^\circ) - V(B_0)$, contradicting our choice of $(C^\circ,D^\circ)$.
It follows from \eqref{eqn:wall3block0} that every $V(C^\circ \cap D^\circ)$-path in $(D^\circ,\gamma)$ is $\Gamma$-zero.

We now show that outcome (2) is satisfied.
Let $(B,\gamma_B)$ be the $\mct_{W_1}$-large 3-block of $(G-Z,\gamma)$. 
Note that $B$ is 3-connected and contains all vertices of degree 3 in $W_1$.
Suppose contrary to outcome (2) that $(B,\gamma_B)$ contains a simple $\Gamma$-nonzero cycle $S'$ (using Lemma \ref{lemma3connshiftequivalent}), and let $S$ be a $\Gamma$-nonzero cycle in $(G-Z,\gamma)$ corresponding to $S'$ as given by Proposition \ref{prop:3blockpathcycle}.

We claim that $S \not\subseteq D^\circ$.
Otherwise, there exist three disjoint $V(C^\circ \cap D^\circ)$-$V(S)$-paths in $D^\circ$, since $S$ is 3-connected to $b(W_1)$ which is in turn highly connected to $V(C^\circ \cap D^\circ)$. By Lemma \ref{lem:threeACpathsnonzero}, we obtain a $\Gamma$-nonzero $V(C^\circ \cap D^\circ)$-path in $(D^\circ,\gamma)$, a contradiction.

We also claim that $S \not\subseteq C^\circ$.
Otherwise, there exist three disjoint $V(C^\circ \cap D^\circ)$-$V(S)$-paths in $(C^\circ-Z,\gamma)$ (since there are three disjoint paths from $S$ to $b(W_1)$ which must go through $V(C^\circ \cap D^\circ)$).
By Lemma \ref{lem:threeACpathsnonzero}, there exists a $\Gamma$-nonzero $V(C^\circ \cap D^\circ)$-path in $C^\circ$.
Extending the endpoints of this path along $O \cup P_1\cup P_2$, we obtain a $\Gamma$-nonzero $N_1$-path in $(G-(V(D_1)-N_1)-Z,\gamma)$, a contradiction.

Therefore, $S$ intersects both $C^\circ -  D^\circ$ and $D^\circ - C^\circ$, so the edges of $S$ can be partitioned into a number of $V(C^\circ \cap D^\circ)$-paths, each contained in either $C^\circ$ or $D^\circ$.
Those in $D^\circ$ are all $\Gamma$-zero by \eqref{eqn:wall3block0}, so $C^\circ$ contains a $\Gamma$-nonzero $V(C^\circ \cap D^\circ)$-path.
This similarly gives a contradictory $\Gamma$-nonzero $N_1$-path in $(G-(V(D_1)-N_1)-Z,\gamma)$, and therefore outcome (2) holds.
\end{proof}

\section{Proof of Theorem \ref{flatwallundirectedtheorem}}
\label{sec:mainproof}

The proof of Theorem \ref{flatwallundirectedtheorem}, restated below, now follows readily from Lemma \ref{oddmodellemma} and Lemma \ref{flatwalllemma}.
\flatwallundirectedtheoremres*

\begin{proof}
Let $r'=(150r^{12}+2)^2$ and $t'=R_4(t,m(t))$ where $m(t)=50(150t^4)^4+1+300t^4$.
Let $g:\mbn\times\mbn\to\mbn$ be a function such that $g(r,t)\geq F(r',t')$ where $F$ is the function from Theorem \ref{flatwalltheorem}, and let $h(r,t)=t' + 50(150t^4)^4 + 50r^{12}$.
Note that we may assume $g(r,t) \ge h(r,t)+3$.

Suppose $(G,\gamma)$ contains an $g(r,t)$-wall $W$.
By Theorem \ref{flatwalltheorem}, either $(G,\gamma)$ contains a $K_{t'}$-model $\pi$ such that $\mct_\pi$ is a truncation of $\mct_W$ or there exists $X \subseteq V(G)$ with $|X| \leq t'-5$ and an $r'$-subwall $W'$ of $W$ that is disjoint from $X$ and flat in $G-X$.

Suppose we are in the first case, that there is a $K_{R_4(t,m(t))}$-model $\pi$ such that $\mct_\pi$ is a truncation of $\mct_W$. By Lemma \ref{oddmodellemma}, either there is a $\Gamma$-odd $K_t$-model $\mu$ in $G$ such that $\mct_\mu$ is a truncation of $\mct_\pi$ (hence of $\mct_W$), or there exists $Y \subseteq V(G)$ with $|Y| < 50(150t^4)^4$ such that the $\mct_\pi$-large 3-block of $(G-Y,\gamma)$ is $\Gamma$-bipartite.
Since $\mct_\pi$ is a truncation of $\mct_W$, this 3-block is also $\mct_W$-large.
The first outcome satisfies (1).
The second outcome satisfies (3) with $Z=Y$.

Now suppose we are in the second case, that there exists $X \subseteq V(G)$ with $|X|\leq t'-5$ and a flat $(150r^{12}+2)^2$-wall $(W',\gamma)$ in $(G-X,\gamma)$ such that $\mct_{W'}$ is a truncation of $\mct_W$.
If there exists $Y \subseteq V(G-X)$ with $|Y| < 50r^{12}$ such that the $\mct_{W'}$-large 3-block of $(G-X-Y,\gamma)$ is $\Gamma$-bipartite, then this 3-block is also $\mct_W$-large, so (3) is satisfied with $Z=X\cup Y$.
Otherwise, by Lemma \ref{flatwalllemma}, there is a flat $50r^{12}$-wall $(W_1,\gamma)$ such that $\mct_{W_1}$ is a truncation of $\mct_W$ and either $(W_1,\gamma)$ is facially $\Gamma$-odd or $(W_1,\gamma)$ is strongly $\Gamma$-bipartite and there is pure $\Gamma$-odd linkage of $(W_1,\gamma)$ of size $r$.
These two outcomes satisfy (2)-(a) and (2)-(b) respectively.
\end{proof}

\subsection*{Acknowledgements}
We thank the anonymous referees for their careful reading of the paper and the numerous helpful comments.


\begin{thebibliography}{99}

\def\JCTB{{\it J.~Combin.\ Theory Ser.\ B}}
\def\CMUC{{\it Comment. Math. Univ. Carol.}}
\def\TAMS{{\it Trans.\ Amer.\ Math.\ Soc.}}
\def\JAMS{{\it J.~Amer.\ Math.\ Soc.}}
\def\PAMS{{\it Proc. Amer. Math. Soc.}}
\def\DM{{\it Discrete Math.}}
\def\CM{{\it Contemporary Math.}}
\def\GC{{\it Graphs and Combin.}}
\def\COM{{\it Combinatorica}}
\def\JGT{{\it J.~Graph Theory}}
\def\JAlgorithms{{\it J.~Algorithms}}
\def\SIAMDM{{\it SIAM J.~Disc.\ Math.}}
\def\CPC{{\it Combinatorics, Probability and Computing}}
\def\EJC{Electron.\ J.~Combin.}
\def\EuropJC{\it Europ.\ J.~Combin.}


\bibitem{BruHeiJoo} H.~Bruhn, M.~Heinlein, and F.~Joos,
Frames, $A$-Paths, and the Erd\H{o}s--P\'osa Property,
{\SIAMDM}, {\bf 32}(2) (2018), 1246-1260.

\bibitem{BruUlm} H.~Bruhn and A.~Ulmer,
Packing $A$-Paths of Length Zero Modulo Four,
{\EuropJC}, {\bf 99} (2022).

\bibitem{CarDieHamHun} J.~Carmesin, R.~Diestel, M.~Hamann, and F.~Hundertmark,
$k$-Blocks: a connectivity invariant for graphs,
{\SIAMDM}, {\bf 28}(4) (2014), 1876-1891.

\bibitem{CarDieHunSte} J.~Carmesin, R.~Diestel, F.~Hundertmark, and M.~Stein,
Connectivity and tree structure in finite graphs,
{\COM}, {\bf 34}(1) (2014), 1-35.

\bibitem{Chu} J.~Chuzhoy,
Improved Bounds for the Flat Wall Theorem,
{\it Proc. SODA} (2015), 256-275.

\bibitem{DejNeu} I.~J.~Dejter and V.~Neumann-Lara,
Unboundedness for Generalized Odd Cyclic Transversality,
{\it Combinatorics (Eger, 1987), Colloq. Math. Soc. J\'anos Bolyai} {\bf 52} (1987), 195-203.

\bibitem{ErdPos} P.~Erd\H{o}s and L.~P\'osa,
On independent circuits contained in a graph,
{\it Canad. J. Math.} {\bf 17} (1965), 347-352.

\bibitem{GeeGer} J.~Geelen and B.~Gerards,
Excluding a group-labelled graph,
{\JCTB} {\bf 99}(1) (2009), 247-253.
  
\bibitem{GeeGerReeSeyVet} J.~Geelen, B.~Gerards, B.~Reed, P.~Seymour, and A.~Vetta,
On the odd-minor variant of Hadwiger's conjecture,
{\JCTB} {\bf 99}(1) (2009), 20-29.

\bibitem{GHKKO} J.~P.~Gollin, K.~Hendrey, K.~Kawarabayashi, O.~Kwon, and S.~Oum,
A unified half-integral Erd\H{o}s-P\'osa theorem for cycles in graphs labelled by multiple abelian groups,
\texttt{arXiv:2102.01986}.

\bibitem{HuyJooWol} T.~Huynh, F.~Joos, and P.~Wollan, 
A Unified Erd\H{o}s-P\'osa Theorem for Constrained Cycles,
{\COM} {\bf 39} (2019), 91-133.

\bibitem{Mad} W.~Mader, 
\"Uber $n$-fach zusammenh\"angende Eckenmengen in Graphen,
{\JCTB} {\bf 25} (1978), 74-93.

\bibitem{Ramsey} F.~P.~Ramsey,
On a problem of formal logic,
{\it Proc. London Math. Soc.} {\bf 30} (1930), 264-286.

\bibitem{RayThi} J-F.~Raymond and D.~Thilikos,
Recent techniques and results on the Erd\H{o}s–P\'osa property,
{\it Discrete Appl. Math.} {\bf 231} (2017), 25-43.

\bibitem{Ree99} B.~Reed,
Mangoes and Blueberries,
{\COM} {\bf 19} (1999), 267-296.

\bibitem{RobSeyX} N.~Robertson and P.~Seymour,
Graph minors. X. Obstructions to tree-decomposition,
{\JCTB} {\bf 52} (1991), 153-190.

\bibitem{RobSeyXIII} N.~Robertson and P.~Seymour,
Graph minors. XIII. The disjoint paths problem,
{\JCTB} {\bf 63} (1995), 65-110.

\bibitem{RobSeyTho} N.~Robertson, P.~Seymour, and R.~Thomas,
Quickly Excluding a Planar Graph,
{\JCTB} {\bf 62} (1994), 323-348.

\bibitem{ThoYoob} R.~Thomas, Y.~Yoo,
Packing $A$-paths of length zero modulo a prime,
\texttt{arXiv:2009.12230}.

\bibitem{Tho88} C.~Thomassen,
On the presence of disjoint subgraphs of a specified type,
{\JGT} {\bf 12} (1988), 101-111.

\bibitem{Tut} W.~T.~Tutte,
Connectivity in graphs,
{\it University of Toronto Press} (1966).

\bibitem{WolCycle} P.~Wollan,
Packing cycles with modularity constraints,
{\COM} {\bf 31} (2011), 95-126.

\bibitem{WolPath} P.~Wollan,
Packing non-zero $A$-paths in an undirected model of group labeled graphs,
{\JCTB} {\bf 100} (2010), 141-150.


\end{thebibliography}
\end{document}